\documentclass{amsart}
\pdfoutput=1
\allowdisplaybreaks
\usepackage{mathpazo}
\usepackage{amsmath,amsfonts,amssymb}
\usepackage{amsrefs}
\usepackage{amsthm}
\usepackage{latexsym,amsmath,amssymb,amsfonts}
\usepackage{rotating}
\usepackage{mathrsfs}
\usepackage{amscd}
\usepackage{hyperref} %% pdflatex
\usepackage{euscript}
\usepackage{hhline}
\usepackage{graphicx,epstopdf}
\usepackage{epsfig}
\usepackage{xcolor}
\usepackage[all,color]{xy}
\usepackage{tikz}
\usetikzlibrary{cd}

\newlength{\fighskip} \fighskip=2pt
\newlength{\figvskip} \figvskip=3pt

\usepackage{hyperref}

\advance\textheight by \topskip
\numberwithin{equation}{section}

\newcommand{\abracket}[1]{\left\langle#1\right\rangle}
\newcommand{\bbracket}[1]{\left[#1\right]}
\newcommand{\fbracket}[1]{\left\{#1\right\}}
\newcommand{\bracket}[1]{\left(#1\right)}

\newcommand{\LS}[1]{(\!(#1)\!)}
\newcommand{\PS}[1]{[[#1]]}

\newcommand{\mc}{\mathcal}

\newcommand{\pa}{\partial}

\newcommand{\OO}{{\mathcal O}}

\newcommand{\CE}{Chevalley-Eilenberg }

\newcommand{\into}{\hookrightarrow}

\newcommand{\iso}{\cong}

\newcommand{\g}{\mathfrak{g}}
\newcommand{\LR}[1]{{\left(#1\right)}}

\newcommand{\Yang}{{\color{red} Yang: }}

\renewcommand{\d}{d_{2k}}

\renewcommand{\*}{\bullet}
\DeclareMathOperator{\Lie}{Lie}

\DeclareMathOperator{\End}{End}

\DeclareMathOperator{\Sym}{Sym}
\DeclareMathOperator{\Hom}{Hom}

\DeclareMathOperator{\Tr}{Tr}

\DeclareMathOperator{\supp}{supp}

\newcommand{\C}{\mathbb{C}}

\newcommand{\R}{\mathbb{R}}

\newcommand{\Z}{\mathbb{Z}}
\newcommand{\E}{\mathbb{E}}
\newcommand{\W}{\mathcal W}

\newcommand{\orderofg}{p}

\DeclareMathOperator{\gl}{\mathfrak{gl}}

\DeclareMathOperator{\Sp}{Sp}

\DeclareMathOperator{\Aut}{Aut}

\DeclareMathOperator{\GL}{GL}
\DeclareMathOperator{\PGL}{PGL}

\DeclareMathOperator{\tr}{tr}

\DeclareMathOperator{\pr}{pr}
\DeclareMathOperator{\Ch}{Ch}
\DeclareMathOperator{\trace}{tr}
\DeclareMathOperator{\Id}{Id}

\DeclareMathOperator{\ad}{ad}
\DeclareMathOperator{\Conj}{Conj}
\DeclareMathOperator{\mult}{mult}

\DeclareMathOperator{\desc}{desc}

\DeclareMathOperator{\Conf}{Conf}
\DeclareMathOperator{\Cyc}{Cyc}

\newcommand{\uTr}{\widehat{\Tr}}
\newcommand{\gog}{\mathfrak{g}}

\newcommand{\HC}{Harish-Chandra }
\renewcommand{\sp}{\mathfrak{sp}}
\theoremstyle{plain}
\newtheorem{thm}{Theorem}[section]
\newtheorem{lem}[thm]{Lemma}
\newtheorem{prop}[thm]{Proposition}

\theoremstyle{definition}
\newtheorem{defn}[thm]{Definition}
\newtheorem{eg}[thm]{Example}
\theoremstyle{remark}
\newtheorem{rem}[thm]{Remark}

\allowdisplaybreaks[4]  %%%%%%%%%%% choose 1-4, where 4 is the strongest desire to breack

\begin{document}

 \title{\textbf  Topological Quantum Mechanics on Orbifolds and Orbifold Index}
  \author{Si Li,  Peng Yang}
  \date{}

  \maketitle

\begin{abstract} 

In this paper, we study topological quantum mechanical models on symplectic orbifolds.  The correlation map gives an explicit orbifold version of quantum HKR map. The exact semi-classical approximation in this model leads to a geometric and quantum field theoretic interpretation of the orbifold algebraic index.
\end{abstract}

\tableofcontents

%%%%%%%%%%%%%%%%%%%%%%%%%%%%%%%%%%%%%%%%%%%
%                                         %
%                                         %
%                                         %
%                                         %
%                                         %
%%%%%%%%%%%%%%%%%%%%%%%%%%%%%%%%%%%%%%%%%%%

\iffalse
\Yang 
replace $\mathcal O$ by $\mathcal O_g$
\fi

\section{Introduction}

In \cite{Localized}, it was shown that correlations in topological quantum mechanics on $S^1$ explains an explicit map from Hochschild chains of Weyl algebras to differential forms on the phase space
$$
\abracket{-}:\quad  C_{-\bullet}(\mathcal{W}_{2n})
\rightarrow
 {\hat{\Omega }^{-\bullet }_{2n}}
$$
which intertwines the Hochschild differential $b$ with the BV operator $\hbar \Delta$, and  intertwines the Connes operator $B$ with the de Rham differential $d$. Here the BV operator $\Delta$ is the Lie derivative with respect the standard Poisson bi-vector. In physics content, the correlation map $\abracket{-}$ is ``integrating out" massive modes to end up with an effective function on zero modes. This gives a quasi-isomorphism between the Hochschild chain complex $(C_{-\bullet}(\mathcal{W}_{2n}),b)$ and the BV complex $ (\hat{\Omega }^{-\bullet }_{2n}, \hbar \Delta)$, which can be viewed as a quantization of the classical HKR map. Composing with a Berezin integration, it gives the Feigin-Felder-Shoiket formula  \cite{trace} for the Hochchild cocycle of Weyl algebra. A further investigation with the $S^1$-action in this model leads to a semi-classical approach to the algebraic index theorem as formulated by Fedosov \cite{Fedbook} and Nest-Tsygan \cite{Nest-Tsygan}.

In this paper, we generalize the above construction to orbifold phase space.  We show that the analogous correlation map gives an explicit formula for an orbifold version of quantum HKR map. The method of exact semi-classical approximation in this model leads to a geometric and quantum field theoretic interpretation of the orbifold algebraic index theorem by Fedosov-Schulze-Tarkhanov \cite{fst} and  Pflaum-Posthuma-Tang \cite{Pflaum2007An}.

\iffalse
In this paper, we follow  the general framework proposed in \cite{Localized} to study the quantum geometry of $\sigma$-models when they are effectively localized to small quantum fluctuations around constant maps. 
The $\sigma$-model has interaction terms and usually becomes free when the target becomes flat. The algebra of observables of the free quantum field theory forms a factorization algebra  \cite{Kevin-book,kevin-owen}. When the theory is consistent at quantum level, it can be glued consistently to form a sheaf on a curved target. For example,   Fedosov's flat connection \cite{Fedosov-DQ} is showed in \cite{GLL} to  produce the interaction term of effective quantum mechanics, leading to a deformation quantization of  symplectic manifolds. In \cite{Localized}, the gluing geometry is represented by a \HC pair $(\g, K)$. Using the interaction term $\widehat\Theta$, the interactive theory on the free target provides a Lie algebra cochain in $C^\*(\g, K; -)$ and glues via the descent map
$$
\desc: C^\*(\g, K; -)\to \Omega^\*(X, P\times_K -).
$$
The expectation value of observables are computed semi-classically. This leads a trace map and $\Tr(1)$ is the algebraic index.
 
\fi

\subsection{A one-dimensional \texorpdfstring{$\sigma$}{sigma}-model with orbifold target}
Let us explain how to formulate topological quantum mechanical model with target a linear symplectic orbifold. 

Let $G$ be a finite group which acts linearly on a $2n$-dimensional vector space $V$, and 
 $\omega$ be a symplectic pairing on $V$ which is compatible with the $G$-action.  $(V/G,\omega)$ is the local model of a symplectic orbifold. 

Consider a one-dimensional $\sigma$-model 
$$
S^1\to (V/G,\omega)
$$
describing maps from a circle to a linear symplectic orbifold. We will treat $V/G$ as a stack quotient. Then a map $f:S^1\to V/G$ can be described as 
 a $G$-equivarient map
$$P \to V$$
where $P$ is some $G$-principal bundle on $S^1$. Define the 
associated vector bundle $\mathcal L:= P\times _G V$, then $f$ can be furthermore identified with a section of $\mathcal L$. 

$G$-principal bundles on $S^1$ are all flat, and classified by the conjugate class $\Conj(G):=G/_{\ad}G$. For each $[g]\in \Conj(G)$, with a representative  $g$, there corresponds to a $G$-principal bundle $P_g$ with monodromy $g$ and its associated vector bundle $\mathcal L_g:= P_g\times _G V$. Thus,  the mapping space can be equivalently described by 
$$\bigoplus_{[g]\in \Conj(G)} \Gamma(S^1, \mathcal L_g).$$

Applying the AKSZ \cite{AKSZ} construction, we arrive at the following space of fields 
$$ \mathcal E=
\bigoplus_{[g]\in \Conj(G)} \Omega^\bullet(S^1, \mathcal L_g)
$$
with the BRST differential $d=d_{dR}$ the de Rham differential  and a 
symplectic pairing of degree $-1$
$$  (\varphi_1, \varphi_2):=\int_{S^1} \omega(\varphi_1, \varphi_2), \quad \varphi_1, \varphi_2\in \mathcal E $$

We consider the free action given by 
$$
S[\varphi]=\frac12(\varphi, d\varphi)=\int_{S^1} \omega(\varphi,d \varphi), \qquad \varphi \in \mathcal E.
$$
These data describe a topological quantum mechanical system on $S^1$. 

Denote
$\mathcal E_g=\Omega^\bullet(S^1, \mathcal L_g)$, which is called the $g$-twisted sector. We have
$$\mathcal E=\bigoplus_{[g]\in \Conj(G)}\mathcal E_g.$$
The (-1)-shifted symplectic pairing is nondegenerate on each $\mathcal E_g$. The theory naturally splits into $g$-twisted  theories, each of which has field space $\mathcal E_g$ 
,  corresponding to maps of the circle to   the fixed locus of $g$. Globally, this suggests us to consider the inertia orbifold of an orbifold.

\begin{center}
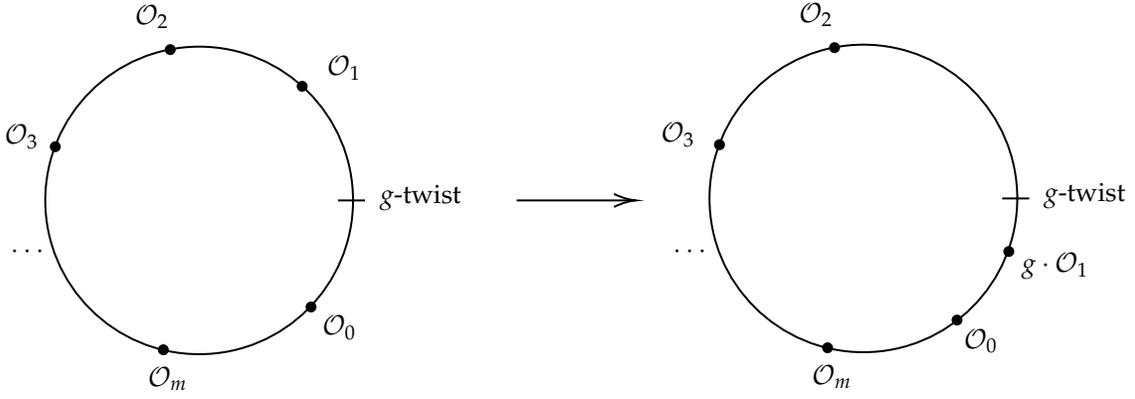
\begin{figure}
\tikzset{every picture/.style={line width=0.75pt}} %set default line width to 0.75pt        

\begin{tikzpicture}[x=0.75pt,y=0.75pt,yscale=-1,xscale=1]

%Shape: Circle [id:dp4605099287172436] 
\draw   (71,144.62) .. controls (71,101.75) and (105.75,67) .. (148.62,67) .. controls (191.48,67) and (226.23,101.75) .. (226.23,144.62) .. controls (226.23,187.48) and (191.48,222.23) .. (148.62,222.23) .. controls (105.75,222.23) and (71,187.48) .. (71,144.62) -- cycle ;
%Shape: Circle [id:dp8675283395126523] 
\draw  [fill={rgb, 255:red, 0; green, 0; blue, 0 }  ,fill opacity=1 ] (198.3,87) .. controls (198.3,85.76) and (199.31,84.75) .. (200.55,84.75) .. controls (201.79,84.75) and (202.8,85.76) .. (202.8,87) .. controls (202.8,88.24) and (201.79,89.25) .. (200.55,89.25) .. controls (199.31,89.25) and (198.3,88.24) .. (198.3,87) -- cycle ;
%Shape: Circle [id:dp33349790193549667] 
\draw  [fill={rgb, 255:red, 0; green, 0; blue, 0 }  ,fill opacity=1 ] (131.8,68.5) .. controls (131.8,67.26) and (132.81,66.25) .. (134.05,66.25) .. controls (135.29,66.25) and (136.3,67.26) .. (136.3,68.5) .. controls (136.3,69.74) and (135.29,70.75) .. (134.05,70.75) .. controls (132.81,70.75) and (131.8,69.74) .. (131.8,68.5) -- cycle ;
%Shape: Circle [id:dp2757944691103186] 
\draw  [fill={rgb, 255:red, 0; green, 0; blue, 0 }  ,fill opacity=1 ] (73.8,117.5) .. controls (73.8,116.26) and (74.81,115.25) .. (76.05,115.25) .. controls (77.29,115.25) and (78.3,116.26) .. (78.3,117.5) .. controls (78.3,118.74) and (77.29,119.75) .. (76.05,119.75) .. controls (74.81,119.75) and (73.8,118.74) .. (73.8,117.5) -- cycle ;
%Shape: Circle [id:dp028776516894871373] 
\draw  [fill={rgb, 255:red, 0; green, 0; blue, 0 }  ,fill opacity=1 ] (128.3,220) .. controls (128.3,218.76) and (129.31,217.75) .. (130.55,217.75) .. controls (131.79,217.75) and (132.8,218.76) .. (132.8,220) .. controls (132.8,221.24) and (131.79,222.25) .. (130.55,222.25) .. controls (129.31,222.25) and (128.3,221.24) .. (128.3,220) -- cycle ;
%Shape: Circle [id:dp8544287400628858] 
\draw  [fill={rgb, 255:red, 0; green, 0; blue, 0 }  ,fill opacity=1 ] (202.8,198.4) .. controls (202.8,197.16) and (203.81,196.15) .. (205.05,196.15) .. controls (206.29,196.15) and (207.3,197.16) .. (207.3,198.4) .. controls (207.3,199.64) and (206.29,200.65) .. (205.05,200.65) .. controls (203.81,200.65) and (202.8,199.64) .. (202.8,198.4) -- cycle ;
%Straight Lines [id:da6634265005224809] 
\draw    (218.62,144.62) -- (232.25,144.62) ;
%Shape: Circle [id:dp5841754942400138] 
\draw   (406,143.62) .. controls (406,100.75) and (440.75,66) .. (483.62,66) .. controls (526.48,66) and (561.23,100.75) .. (561.23,143.62) .. controls (561.23,186.48) and (526.48,221.23) .. (483.62,221.23) .. controls (440.75,221.23) and (406,186.48) .. (406,143.62) -- cycle ;
%Shape: Circle [id:dp8940539323938441] 
\draw  [fill={rgb, 255:red, 0; green, 0; blue, 0 }  ,fill opacity=1 ] (554.7,170.4) .. controls (554.7,169.16) and (555.71,168.15) .. (556.95,168.15) .. controls (558.19,168.15) and (559.2,169.16) .. (559.2,170.4) .. controls (559.2,171.64) and (558.19,172.65) .. (556.95,172.65) .. controls (555.71,172.65) and (554.7,171.64) .. (554.7,170.4) -- cycle ;
%Shape: Circle [id:dp3982684139081867] 
\draw  [fill={rgb, 255:red, 0; green, 0; blue, 0 }  ,fill opacity=1 ] (466.8,67.5) .. controls (466.8,66.26) and (467.81,65.25) .. (469.05,65.25) .. controls (470.29,65.25) and (471.3,66.26) .. (471.3,67.5) .. controls (471.3,68.74) and (470.29,69.75) .. (469.05,69.75) .. controls (467.81,69.75) and (466.8,68.74) .. (466.8,67.5) -- cycle ;
%Shape: Circle [id:dp3486223496020515] 
\draw  [fill={rgb, 255:red, 0; green, 0; blue, 0 }  ,fill opacity=1 ] (408.8,116.5) .. controls (408.8,115.26) and (409.81,114.25) .. (411.05,114.25) .. controls (412.29,114.25) and (413.3,115.26) .. (413.3,116.5) .. controls (413.3,117.74) and (412.29,118.75) .. (411.05,118.75) .. controls (409.81,118.75) and (408.8,117.74) .. (408.8,116.5) -- cycle ;
%Shape: Circle [id:dp9690903462642281] 
\draw  [fill={rgb, 255:red, 0; green, 0; blue, 0 }  ,fill opacity=1 ] (463.3,219) .. controls (463.3,217.76) and (464.31,216.75) .. (465.55,216.75) .. controls (466.79,216.75) and (467.8,217.76) .. (467.8,219) .. controls (467.8,220.24) and (466.79,221.25) .. (465.55,221.25) .. controls (464.31,221.25) and (463.3,220.24) .. (463.3,219) -- cycle ;
%Shape: Circle [id:dp4480956517013791] 
\draw  [fill={rgb, 255:red, 0; green, 0; blue, 0 }  ,fill opacity=1 ] (528.6,205) .. controls (528.6,203.76) and (529.61,202.75) .. (530.85,202.75) .. controls (532.09,202.75) and (533.1,203.76) .. (533.1,205) .. controls (533.1,206.24) and (532.09,207.25) .. (530.85,207.25) .. controls (529.61,207.25) and (528.6,206.24) .. (528.6,205) -- cycle ;
%Straight Lines [id:da3821194293417409] 
\draw    (553.62,143.62) -- (567.25,143.62) ;
%Straight Lines [id:da565648715959082] 
\draw    (308.75,144.62) -- (366.75,144.62) ;
\draw [shift={(368.75,144.62)}, rotate = 180] [color={rgb, 255:red, 0; green, 0; blue, 0 }  ][line width=0.75]    (10.93,-3.29) .. controls (6.95,-1.4) and (3.31,-0.3) .. (0,0) .. controls (3.31,0.3) and (6.95,1.4) .. (10.93,3.29)   ;

% Text Node
\draw (209.3,201.4) node [anchor=north west][inner sep=0.75pt]   [align=left] {$\mathcal O_0$};
% Text Node
\draw (211.67,70) node [anchor=north west][inner sep=0.75pt]   [align=left] {$\mathcal O_1$};
% Text Node
\draw (114.67,44.5) node [anchor=north west][inner sep=0.75pt]   [align=left] {$\mathcal O_2$};
% Text Node
\draw (238.17,135.5) node [anchor=north west][inner sep=0.75pt]   [align=left] {$g$-twist};
% Text Node
\draw (121.19,227.12) node [anchor=north west][inner sep=0.75pt]   [align=left] {$\mathcal O_m$};
% Text Node
\draw (532.85,208) node [anchor=north west][inner sep=0.75pt]   [align=left] {$\mathcal O_0$};
% Text Node
\draw (449.67,43.5) node [anchor=north west][inner sep=0.75pt]   [align=left] {$\mathcal O_2$};
% Text Node
\draw (562,169.4) node [anchor=north west][inner sep=0.75pt]   [align=left] {$g\cdot \mathcal O_1$};
% Text Node
\draw (573.17,134.5) node [anchor=north west][inner sep=0.75pt]   [align=left] {$g$-twist};
% Text Node
\draw (456.19,226.12) node [anchor=north west][inner sep=0.75pt]   [align=left] {$\mathcal O_m$};
% Text Node
\draw (49,105) node [anchor=north west][inner sep=0.75pt]   [align=left] {$\mathcal O_3$};
% Text Node
\draw (380,104) node [anchor=north west][inner sep=0.75pt]   [align=left] {$\mathcal O_3$};
% Text Node
\draw (52,166) node [anchor=north west][inner sep=0.75pt]   [align=left] {$\cdots$};
% Text Node
\draw (386,166) node [anchor=north west][inner sep=0.75pt]   [align=left] {$\cdots$};

\end{tikzpicture}
\caption{An observable $\mathcal  O_1$ passing through the $g$-twist}\label{twist}
\end{figure}
\end{center}

The algebra of local observables is the Weyl algebra $\mathcal W_{2n}$ with the Moyal-Weyl product. It carries a natural $G$-action. In the $g$-twisted theory, there would be a point defect describing the $g$-action when observables pass through it. See figure \ref{twist}.

Then a slight modification of \cite{Localized} produces a correlation map in the $g$-twisted theory (see Definition \ref{defn-correlation})
$$
\langle-\rangle_{free}^g:
C_{-\bullet}(\gl_r(\mathcal{W}_{2n}),g)^{C(g)}
\rightarrow
 \bracket{\hat{\Omega }^{-\bullet }_{2k}}^{C(g)}.
$$
Here $C_{-\bullet}(\gl_r(\mathcal{W}_{2n}),g)$ are the $g$-twisted Hochschild chains (see Definition \ref{defn-g-Hochschild}).  $C(g)$ is the centralizer of $g$. Again, $\langle-\rangle_{free}^g$  intertwines the $g$-twisted Hochschild differential $b_g$ with the BV operator $\hbar \Delta_g$, and  intertwines the $g$-twisted Connes operator $B_g$ with the de Rham differential $d$ (see Lemma \ref{lem-g-intertwine}).  Thus $\langle-\rangle_{free}^g$ can be viewed as an explicit formula for the $g$-twisted version of quantum HKR map.

\subsection{The HKR map: classical v.s. quantum, manifolds v.s. orbifolds}
In this section, we review the correlation map in \cite{Localized} and compare it with the orbifold case. 

\begin{itemize}
\item[(1)]
Consider formal power series and formal differential forms in $n$ variables (we put $\hbar$ coefficients for convenience)
$$
A_{2n}=\mathbb{C}[[y^i]]\LS{\hbar}, \quad \hat{\Omega}^{-\*}_{2n}=\mathbb{C}[[y^i,dy^i]]\LS{\hbar}
$$
which can be viewed as the algebra of formal functions and differential forms on the infinitesimal neighborhood of the origin in $\R^{2n}$. Here the 1-forms $dy^i$ have cohomology degree $-1$ and anticommute with each other. 
The HKR theorem gives a quasi-isomorphism from the Hochschild chain complex of $A_{2n}$ to $\hat{\Omega}^{-\*}_{2n}$
$$
C_{-\bullet}(A_{2n}) \to \hat{\Omega}^{-\*}_{2n}
$$
which intertwines the Hochschild differential $b$ with $0$. 

\item[(2)]
Let  $(\R^{2n},\omega)$ be a symplectic vector space with linear coordinates $y^i$. The Weyl algebra 
$$
\mathcal{W}_{2n}:=\mathbb{C}[[y^i]]\LS{\hbar}
$$
can be seen as a deformation quantized algebra of $A_{2n}$, equipped with the Moyal-Weyl product 
$$
f\star g=m(e^{\hbar\Pi}(f\otimes g)), \quad \Pi=\omega^{-1}.
$$
Here the Poisson bi-vector $\Pi=\sum\limits_{i,j}\omega^{ij}\pa_i\wedge \pa_j$ acts on $f\otimes g$ as $\sum\limits_{i,j}\omega^{ij}\pa_i f \otimes \pa_jg$, and
$$
m: A_{2n}\otimes A_{2n}\to A_{2n}
$$
is the multiplication map.  Correlations in this model lead to a quasi-isomorphism \cite{Localized} 
$$
\langle-\rangle_{free}: C_{-\bullet}(\mathcal{W}_{2n}) \to \hat{\Omega}^{-\*}_{2n}
$$
which intertwines the Hochschild differential $b$ with the BV operator $\hbar\Delta$, and intertwines the Connes operator $B$ with the de Rham differential $d$.

\item[(3)]
We can couple the system with a rank $r$ vector bundle. This means observables are matrix valued and we will take a trace at the end. We get a quasi-isomorphism \cite{Localized} 
$$
\langle-\rangle_{free}: C_{-\bullet}(\gl_r(\mathcal{W}_{2n})) \to \hat{\Omega}^{-\*}_{2n}.
$$

\item[(4)]
In this paper we consider the orbifold case. The algebra of formal functions on the singular space $\R^{2n}/G$ has a smooth resolution by the semi-product $A_{2n}\rtimes G$. The Hochschild chain complex of such a twisted group algebra naturally splits:
$$
C_{-\bullet} (A\rtimes G) 
\simeq \LR{\bigoplus_{g\in G} C_{-\bullet} (A,g)}^G
\simeq \bigoplus_{[g]\in \Conj(G)} C_{-\bullet} (A,g) ^{C(g)}.
$$
%\Yang Here for different $g$'s in a conjugate class $[g]$, the chain compleces $C_{-\bullet} (A,g)$ are naturally isomorphic by conjugating. 
Similar things hold for  cyclic/periodic cyclic chains. Here a term $C_{-\bullet} (A,g) ^{C(g)}$ describes observables on the fixed point set of %the conjugate class 
$g$, and for different $g$'s representing a conjugate class $[g]$, there are natural isomorphisms linking these chains by conjugating. %This suggests us to consider the inertia orbifold globally. 
Let us choose linear coordinates %$y^1,\cdots,y^k,z^{2k+1},\cdots,z^{2n}$
$y^i,z^j$, such that $y^i$ are coordinates on the $g$-fixed locus%, as in section \ref{TQM}
. So we can write 
$$
A_{2n}=\mathbb{C}[[y^i,z^j]]\LS{\hbar},
\quad \hat{\Omega}^{-\*}_{{2n}}=\C\PS{y^i,z^j,dy^i,dz^j}\LS{\hbar} 
$$
and we denote the formal de Rham algebra on $g$-fixed locus by
$$
\hat{\Omega}^{-\*}_{2k}:=\C\PS{y^i,dy^i}\LS{\hbar}.
$$
Then the HKR map gives a quasi-isomorphism
$$
C_{-\bullet}(A_{2n},g) \to \hat{\Omega}^{-\*}_{2k}.
$$
The $g$-twisted correlation map gives a quasi-isomorphism (see Definition \ref{defn-correlation})
$$
\langle-\rangle_{free}^g:
C_{-\bullet}(\gl_r(\mathcal{W}_{2n}),g)^{C(g)}
\rightarrow
 \bracket{\hat{\Omega }^{-\bullet }_{2k}}^{C(g)}.
$$
%Here the matrix coefficients need a $g$-twisted trace \ref{Trace}.

\end{itemize}

\subsection{Relation with orbifold algebraic index}
By incorporating $\langle-\rangle_{free}^g$ with the universal flat connection $\widehat\Theta$  in \cite{Localized} and by composing with a Berezin integral, we obtain the universal  trace map  (see Definition \ref{defn-universal-trace})
$$
\widehat{\Tr}_g := \int_{BV}\abracket{-}_{int}^g
\in
C_{\mathrm{Lie}}^\bullet \bracket{\mathfrak{g};
\Hom_{\mathbb{C}(\!(\hbar)\!)}\bracket{
CC_{-\bullet}^{per}(\gl_r(\mathcal{W}_{2n}),g)^{C(g)}
,
\mathbb{K}
}}, \quad \mathbb{K}=\C\LS{\hbar}[u,u^{-1}]$$
which lies in Lie algebra cochains. The universal index is
$$
\uTr_g(1)
$$
which can be computed semi-classically in analogy with the $S^1$-equivariant method (see Proposition \ref{ONELOOP}). 

By the descend construction, this leads to the algebraic index theorem for orbifolds \cites{Pflaum2007An,fst}
$$ \Tr(1)=\int_{\wedge X}e^{-\omega_{\hbar}/\hbar}\frac{\hat{A}(\wedge X)\Ch_g(E)}{m\det (1-{g_\perp}^{-1}e^{-R^\perp})}.$$
This gives a quantum field theoretical interpretation of the orbifold algebraic index.

\noindent \textbf{Acknowledgment}. The authors would like to thank Xiang Tang and Tianqing Zhu for helpful communications.  This work of S.~L. is supported by the National Key R\&D Program of China  (NO. 2020YFA0713000). 

\noindent \textbf{Conventions}. 
Given two elements $A,B$ in a graded algebra, the commutator $[A,B]$  always means a graded commutator (here $|\cdot|$ is the degree):  
$$
[A,B]=A\cdot B-(-1)^{|A|\cdot|B|}B\cdot A.
$$

\iffalse

The free expectation and propagator is used in \cite{fe:g-index} to calculate the $G$-index and further in \cite{Pflaum2007An} to obtain the algebraic index theorem for orbifolds by using large N techniques. Traces on the deformation quantized algebra  of an orbifold is classified in  \cite{nppt}. Our results can be seen as a  quantum field theoretical construction of them.

\fi

\section{Topological quantum mechanics on orbifolds}\label{section:1}

\subsection{%\texorpdfstring{$\sigma$}{sigma}-model on orbifolds / 
Twisted sectors}

We are interested in $\sigma$-model on orbifolds. In quantum field theory, a $\sigma$-model describes dynamics on the mapping space  
$$
\phi: \Sigma \to X
$$ 
from a source geometry $\Sigma$ to a target geometry $X$. 

Let us consider the case when $X$ is an orbifold which is modeled by a global quotient 
$$
X= M/G.
$$  
Here $M$ is a smooth manifold with an action by a finite group $G$.  A map $\Sigma\to X$ to an orbifold is described by a principal $G$-bundle $P$ on $\Sigma$ together with a $G$-equivariant map $P\to X$
$$
\xymatrix{
P \ar[r]\ar[d]& X\\
\Sigma &
}
$$
 Then a $\sigma$-model $\Sigma \to X$ is equivalently a $G$-gauged $\sigma$-model with  field space as maps $P\to X$. 

Principal $G$-bundles on $\Sigma$ are generally not trivial. Then the theory splits into $\sigma$-models with $P$ in an isomorphism class of principal $G$-bundles on $\Sigma$ and a possibly smaller gauge group $H\subset G$. These disconnected components  are called twisted sectors.

\begin{eg} We will mainly consider the quantum mechanical case when $\Sigma$ is one-dimensional. 
\begin{itemize}
\item Let $\Sigma=I$ be an interval. A $G$-principal bundle on $I$ is trivial. Thus, a path $I\to M/G$ in $M/G$ is equivalently a path $I\to M$ in $M$, up to a choice of lift of base point. See figure \ref{fig:circle}.
\item Let $\Sigma=S^1$ be a circle. As we see in the introduction, the theory splits into $|\Conj(G)|$ theories. Following \cite{Localized}, we will be also interested in the low-energy limit when the size of a circle in an orbifold $S^1\to X$  tends to $0$. When $S^1$  shrinks to a point, there  only remains a point in $x\in M$ together with an automorphism $g$ which fixes $x$, or in other words, a point $x$ in the   $g$-fixed locus $M^g$. Thus   the mapping space becomes the inertia orbifold $\wedge X$ of $X$ in the low-energy limit.
\end{itemize}
\end{eg}

\iffalse
\vskip 5em

\begin{center}
\begin{tabular}{|c|c|}
\hline
Orbifold $X=M/G$ & Manifold $M$ \\
\hline   
path $I\to X$  & (fix a lift of base point) $I\to M$ \\
\hline
loop $S^1\to X$  & (fix a lift of base point) $x:I=[0,1]\to M$, $x(0)=gx(1)$ \\
\hline
constant map to $X$ & $g$-fixed locus $M^g$ \\
\hline
\end{tabular}
\end{center}
\fi

\begin{center}
\begin{figure}
    \centering
\tikzset{every picture/.style={line width=0.75pt}} %set default line width to 0.75pt        

\begin{tikzpicture}[x=0.75pt,y=0.75pt,yscale=-1,xscale=1]
%uncomment if require: \path (0,311); %set diagram left start at 0, and has height of 311

%Shape: Polygon Curved [id:ds9360718235020834] 
\draw   (112.38,189.85) .. controls (105.38,178.85) and (185.38,89.85) .. (197.38,89.85) .. controls (209.38,89.85) and (289.38,177.85) .. (283.38,188.85) .. controls (277.38,199.85) and (238.38,225.85) .. (195.38,225.85) .. controls (152.38,225.85) and (119.38,200.85) .. (112.38,189.85) -- cycle ;
%Curve Lines [id:da3384366589661617] 
\draw    (166.38,167.85) .. controls (179.38,144.85) and (225.45,149.72) .. (226.38,166.85) .. controls (227.32,183.98) and (179.38,191.85) .. (166.38,167.85) -- cycle ;
%Shape: Ellipse [id:dp4287457911910242] 
\draw  [color={rgb, 255:red, 0; green, 0; blue, 0 }  ,draw opacity=1 ][fill={rgb, 255:red, 0; green, 0; blue, 0 }  ,fill opacity=1 ] (164.38,167.85) .. controls (164.38,166.75) and (165.28,165.85) .. (166.38,165.85) .. controls (167.49,165.85) and (168.38,166.75) .. (168.38,167.85) .. controls (168.38,168.95) and (167.49,169.85) .. (166.38,169.85) .. controls (165.28,169.85) and (164.38,168.95) .. (164.38,167.85) -- cycle ;
%Shape: Ellipse [id:dp6480560865151198] 
\draw   (420,161.43) .. controls (420,120.24) and (474.93,86.85) .. (542.69,86.85) .. controls (610.45,86.85) and (665.38,120.24) .. (665.38,161.43) .. controls (665.38,202.61) and (610.45,236) .. (542.69,236) .. controls (474.93,236) and (420,202.61) .. (420,161.43) -- cycle ;
%Curve Lines [id:da37402576287806244] 
\draw    (490.38,162.85) .. controls (513.38,140.85) and (520.38,137.85) .. (535.38,142.85) .. controls (550.38,147.85) and (572.38,177.85) .. (589.38,158.85) ;
%Shape: Ellipse [id:dp45626905896217396] 
\draw  [color={rgb, 255:red, 0; green, 0; blue, 0 }  ,draw opacity=1 ][fill={rgb, 255:red, 0; green, 0; blue, 0 }  ,fill opacity=1 ] (488.38,162.85) .. controls (488.38,161.75) and (489.28,160.85) .. (490.38,160.85) .. controls (491.49,160.85) and (492.38,161.75) .. (492.38,162.85) .. controls (492.38,163.95) and (491.49,164.85) .. (490.38,164.85) .. controls (489.28,164.85) and (488.38,163.95) .. (488.38,162.85) -- cycle ;
%Shape: Ellipse [id:dp7245946272322412] 
\draw  [color={rgb, 255:red, 0; green, 0; blue, 0 }  ,draw opacity=1 ][fill={rgb, 255:red, 0; green, 0; blue, 0 }  ,fill opacity=1 ] (587.38,158.85) .. controls (587.38,157.75) and (588.28,156.85) .. (589.38,156.85) .. controls (590.49,156.85) and (591.38,157.75) .. (591.38,158.85) .. controls (591.38,159.95) and (590.49,160.85) .. (589.38,160.85) .. controls (588.28,160.85) and (587.38,159.95) .. (587.38,158.85) -- cycle ;

% Text Node
\draw (160,181) node [anchor=north west][inner sep=0.75pt]   [align=left] {$x$};
% Text Node
\draw (482.2,175.4) node [anchor=north west][inner sep=0.75pt]   [align=left] {$\widetilde x$};
% Text Node
\draw (587.2,170.8) node [anchor=north west][inner sep=0.75pt]   [align=left] {$g \widetilde x$};
% Text Node
\draw (155,258) node [anchor=north west][inner sep=0.75pt]   [align=left] {a circle in $M/G$};
% Text Node
\draw (470,258) node [anchor=north west][inner sep=0.75pt]   [align=left] {a lift of the circle to $M$};

\end{tikzpicture}
    \caption{A $g$-twisted sector}
    \label{fig:circle}
\end{figure}
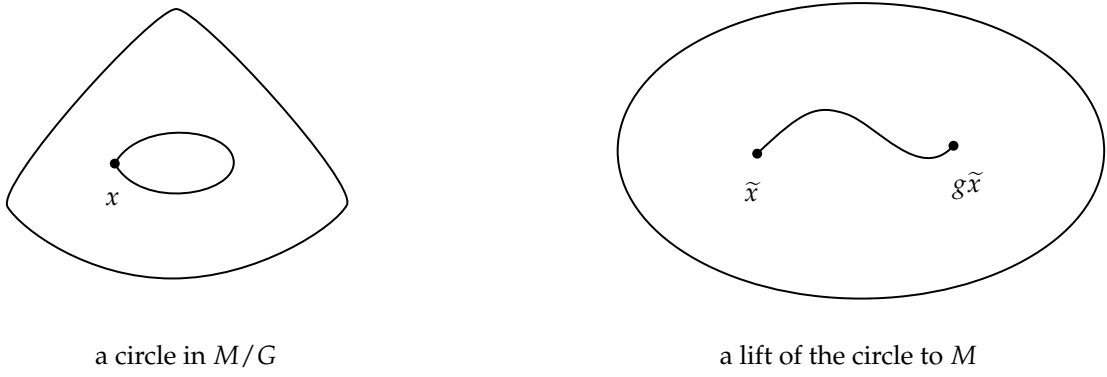
\end{center}

\subsection{Topological quantum mechanics on orbifolds} \label{TQM}
Following our notations in the introduction, we will study the topological quantum mechanics on orbifolds via the AKSZ approach, which can be viewed as a $\sigma$-model describing maps
$$
S^1_{dR}\to V/G.
$$
Here $S^1_{dR}$ is the locally-ringed space whose underlying topology is the circle $S^1$, with the structure sheaf being the de Rham complex of forms on $S^1$. The target is a $2n$-dimensional symplectic vector space $(V,\omega)$ with an action of a finite group $G\to \End(V)$ compatible with $\omega$.

The space of fields   in the $g$-twisted sector is
$$
\mathcal E_g=\Omega^\bullet(S^1, \mathcal L_g).
$$ 
 The action functional in the $g$-twisted sector is the free one as described in the introduction
$$
S[\varphi]=\frac12(\varphi, d\varphi)=\int_{S^1} \omega(\varphi,d \varphi), \qquad \varphi \in \mathcal E_g.
$$

Before we move on to discuss the quantum theory, let us fix our notations that will be used throughout this paper. Denote the order of $g$ by $\orderofg$. 
Let $2k$ be the dimension of the subspace of $g$-invariant elements in $V$. By choosing suitable Darboux coordinates, we can identify the triple  $(V,\omega,g)$ with  $(\R^{2n},\omega,g)$, where
\begin{itemize}
\item $\R^{2n}=\R^{2k}\oplus \R^{2n-2k}$, with coordinates $y^1,\cdots,y^{2k}$ on $\R^{2k}$ and $z^{2k+1},\cdots,z^{2n}$ on $\R^{2n-2k}$
\item $\omega=\omega|_{\R^{2k}\otimes\R^{2k}}+\omega|_{\R^{2n-2k}\otimes \R^{2n-2k}}$ %\displaystyle \sum\limits_{i=1}^{k} y^i\wedge y^{i+k} + \displaystyle \sum\limits_{j=2k+1}^{n+k} z^j\wedge z^{j+n-k}$
\item %$g$ acts trivially on $\R^{2k}$, and preserves the subspace $\R^{2n-2k}$. 
$g|_{\R^{2k}}=1$, $g(\R^{2n-2k})\subseteq \R^{2n-2k}$. Denote $g_\perp:=g|_{\R^{2n-2k}}$. 
\end{itemize}
In fact, we can split $(V,\omega,g)=\oplus_{i=1}^n (V_i,\omega_i,g_i)$ with $\text{dim}V_i=2$, by making a complex coordinate change. 

The inverse of $\omega$ is a bivector 
\begin{align*}
\Pi&=
\frac{1}{2}\sum_{i=1}^{k}\bracket{\frac{\partial}{\partial y^i}\otimes\frac{\partial}{\partial y^{i+k}}-\frac{\partial}{\partial y^{i+k}}\otimes \frac{\partial}{\partial y^i}}
+\frac{1}{2}\sum_{j=2k+1}^{n+k}\bracket{\frac{\partial}{\partial z^j}\otimes\frac{\partial}{\partial z^{j+n-k}}-\frac{\partial}{\partial z^{j+n-k}}\otimes \frac{\partial}{\partial z^j}}\\
&=:\Pi_1+\Pi_2.
\end{align*}

\subsection{Twisted propagator}
In this section we calculate the  propagator in the $g$-twisted sector,  which represents the integral kernel of ``$d^{-1}$'' on $\mathcal E_g=\Omega^\bullet(S^1, \mathcal L_g)$ as indicated from the free action. The topological quantum mechanical model has a local gauge symmetry 
$$
\varphi \to \varphi + d\chi
$$
and we need to impose gauge fixing condition in order to invert $d$ appropriately. This can be done by using the standard flat metric on $S^1$, and the gauge fixed propagator becomes the integral kernel of the operator 
$$
d^*{\frac1  D}
$$  
where $d^*$ is the adjoint of $d$ and $D$ is the Laplacian. This propagator will be used to construct the Feynman diagram integrals in the quantum theory. 

Let $x$ denote the linear coordinate on $S^1$ where $x\sim x+1$ are identified. From the  Laplacian $D=-\frac{d^2}{dx^2}$, we can form the heat operator 
$$e^{-tD}:\,\Omega^\bullet(S^1, \mathcal L_g)\to \Omega^\bullet(S^1, \mathcal L_g).$$

The heat kernel $\mathbb{K}_t\in \Omega^\bullet(S^1\times S^1, \mathcal L_g\boxtimes \mathcal L_g)$ is defined by the integral relation
$$(e^{-tD}\alpha)(x)=\int_{y\in S^1}\left< \mathbb{K}_t(x,y),\alpha(y) \right>_y.$$
Here $\left< - ,-  \right>_y$ means we pair the second factor of $\mathbb{K}_t$ with $\alpha$.
Explicitly, the heat kernel is  given by
$$\mathbb{K}_t(x,y)=\frac{1}{\sqrt{4\pi t}}\sum_{m\in \mathbb{Z}}\,\mathrm{exp}\bracket{-\frac{(x-y+m)^2}{4t}} (dx\otimes 1-1\otimes dy)\cdot (1\otimes g^{-m})\Pi.$$
Here $\Pi$ is the inverse of $\omega$, i.e., the Poisson bi-vector.  

Let $d^*$ be the adjoint of $d$. The propagator $\mathbb{P}_0^\infty $ representing the integral kernel of $d^*{\frac 1 D}$ is
\iffalse
$$
\begin{aligned}
\mathbb{P}_0^\infty (x,y)
&=\displaystyle \lim_{\epsilon  \to 0  \atop   L\to \infty} \int_{\epsilon }^{L} (d^*\otimes 1)\mathbb{K}_t (x,y ) dt \\
&=\displaystyle \lim_{\epsilon  \to 0 \atop   L\to \infty} \int_{\epsilon }^{L} dt \frac{1}{\sqrt{4\pi t}}\sum_{m\in \mathbb{Z}}\bracket{-\frac{x-y+m}{2t}}\,\mathrm{exp}\bracket{-\frac{(x-y+m)^2}{4t}} \cdot (1\otimes g^{-m})\Pi \\
&=-\frac{1}{4\sqrt{\pi }} \sum_{k=0}^{\orderofg -1}\displaystyle \lim_{\epsilon  \to 0 \atop   L\to \infty} \int_{\epsilon }^{L} dt\, t^{-3/2} \sum_{m\in \orderofg \mathbb{Z}+k}(x-y+m)\,\mathrm{exp}\bracket{-\frac{(x-y+m)^2}{4t}} \cdot (1\otimes g^{-k})\Pi \\
&=-\frac{1}{4\sqrt{\pi }} \sum_{k=0}^{\orderofg -1}\displaystyle \lim_{\epsilon  \to 0 \atop   L\to \infty} \int_{\epsilon }^{L} dt\, t^{-3/2}  \sum_{m\in \orderofg \mathbb{Z}+k}\\
&\qquad \qquad \qquad \underset{\mathrm{absolutely\,\, convergent}}{\underbrace{\left ( (x-y+m)\,\mathrm{exp}\bracket{-\frac{(x-y+m)^2}{4t}}-\frac{1}{\orderofg }\int_{x-y+m-\orderofg /2}^{x-y+m+\orderofg /2}u\,e^{-u^2/4t}du \right )}} \cdot (1\otimes g^{-k})\Pi \\
&=-\frac{1}{4\sqrt{\pi }} \sum_{k=0}^{\orderofg -1} \sum_{m\in \orderofg \mathbb{Z}+k}\displaystyle \lim_{\epsilon  \to 0 \atop   L\to \infty} \int_{\epsilon }^{L} dt\, t^{-3/2} \\
&\qquad \qquad \qquad \left ( (x-y+m)\,\mathrm{exp}\bracket{-\frac{(x-y+m)^2}{4t}}-\frac{1}{\orderofg }\int_{x-y+m-\orderofg /2}^{x-y+m+\orderofg /2}u\,e^{-u^2/4t}du \right ) \cdot (1\otimes g^{-k})\Pi .
\end{aligned}
$$
\fi
$$
\begin{aligned}
\mathbb{P}_0^\infty (x,y)
&=\displaystyle \lim_{\substack{\epsilon  \to 0  \\   L\to \infty}} \int_{\epsilon }^{L} (d^*\otimes 1)\mathbb{K}_t (x,y ) dt \\
&=\displaystyle \lim_{\substack{\epsilon  \to 0 \\   L\to \infty}} \int_{\epsilon }^{L} dt \frac{1}{\sqrt{4\pi t}}\sum_{m\in \mathbb{Z}}\bracket{-\frac{x-y+m}{2t}}\,\mathrm{exp}\bracket{-\frac{(x-y+m)^2}{4t}} \cdot (1\otimes g^{-m})\Pi \\
&=-\frac{1}{4\sqrt{\pi }} \sum_{k=0}^{\orderofg -1}\displaystyle \lim_{\substack{\epsilon  \to 0 \\   L\to \infty}} \int_{\epsilon }^{L} dt\, t^{-3/2} \sum_{m\in \orderofg \mathbb{Z}+k}(x-y+m)\,\mathrm{exp}\bracket{-\frac{(x-y+m)^2}{4t}} \cdot (1\otimes g^{-k})\Pi \\
&=-\frac{1}{4\sqrt{\pi }} \sum_{k=0}^{\orderofg -1}\displaystyle \lim_{\substack{\epsilon  \to 0 \\   L\to \infty}} \int_{\epsilon }^{L} dt\, t^{-3/2}  \sum_{m\in \orderofg \mathbb{Z}+k}\\
&\qquad \qquad \qquad \underset{\mathrm{absolutely\,\, convergent}}{\underbrace{\left ( (x-y+m)\,\mathrm{exp}\bracket{-\frac{(x-y+m)^2}{4t}}-\frac{1}{\orderofg }\int_{x-y+m-\orderofg /2}^{x-y+m+\orderofg /2}u\,e^{-u^2/4t}du \right )}} \cdot (1\otimes g^{-k})\Pi \\
&=-\frac{1}{4\sqrt{\pi }} \sum_{k=0}^{\orderofg -1} \sum_{m\in \orderofg \mathbb{Z}+k}\displaystyle \lim_{\substack{\epsilon  \to 0 \\   L\to \infty}} \int_{\epsilon }^{L} dt\, t^{-3/2} \\
&\qquad \qquad \qquad \left ( (x-y+m)\,\mathrm{exp}\bracket{-\frac{(x-y+m)^2}{4t}}-\frac{1}{\orderofg }\int_{x-y+m-\orderofg /2}^{x-y+m+\orderofg /2}u\,e^{-u^2/4t}du \right ) \cdot (1\otimes g^{-k})\Pi .
\end{aligned}
$$

When $0<x-y<1$, 
$$\begin{aligned}
P_1:=\mathbb{P}_0^\infty (x,y)
&=-\frac{1}{4\sqrt{\pi }} \sum_{k=0}^{\orderofg -1} \sum_{m\in \orderofg \mathbb{Z}+k}
%\\ &\qquad \qquad \qquad 
\left ( (x-y+m)\frac{2\sqrt{\pi }}{|x-y+m|}-\frac{1}{\orderofg }\int_{x-y+m-\orderofg /2}^{x-y+m+\orderofg /2}u\frac{2\sqrt{\pi }}{|u|}du \right ) \cdot (1\otimes g^{-k})\Pi \\
&=-\frac{1}{2} \sum_{k=0}^{\orderofg -1} \sum_{m=k,k-\orderofg }
%\\ &\qquad \qquad \qquad 
\left ( (x-y+m)\frac{1}{|x-y+m|}-\frac{1}{\orderofg }\int_{x-y+m-\orderofg /2}^{x-y+m+\orderofg /2}u\frac{1}{|u|}du \right ) \cdot (1\otimes g^{-k})\Pi  \\
&=\sum_{k=0}^{\orderofg -1}\left ( \frac{x-y}{\orderofg }-\frac{1}{2}+\frac{k}{\orderofg } \right ) \cdot (1\otimes g^{-k})\Pi \\
&=\left (x-y-\frac{1}{2}\right )   \bracket{1\otimes \frac{1}{\orderofg }\sum_{k=0}^{\orderofg -1}g^{-k}}\Pi +\frac{1}{\orderofg } \sum_{k=0}^{\orderofg -1}\bracket{k-\frac{\orderofg -1}{2}} \cdot (1\otimes g^{-k})\Pi \\
&=\left (x-y-\frac{1}{2}\right )   \Pi_1 - \bracket{1\otimes \frac {1}{1-{g_\perp}^{-1}}}\Pi_2 \\
&=:P_{11}+P_{12}
%&=\left (x-y-\frac{1}{2}\right )   \Pi_1 - \bracket{1\otimes \frac {{g_\perp}^{-1}}{1-{g_\perp}^{-1}}}\Pi_2 -\Pi_2
%\\&=P_0+P-\Pi_2
.
\end{aligned}$$
The propagator $P_1$ represents a line that connects $x$ to $y$ in Feynman diagrams. When pulled back to configuration space $S^1[2]$, 
its non-constant part $P_{11}$ becomes $\left (u-\frac{1}{2} \right )\Pi_1$, which is the propagator used in \cites{trace,GLL,Localized}.

When $x=y$, 
$$\begin{aligned}
P_2:=\mathbb{P}_0^\infty (x,y)|_{x=y}
&=-\frac{1}{2} \sum_{k=1}^{\orderofg -1} \sum_{m=k,k-\orderofg }
%\\&\qquad \qquad \qquad
%\left ( (x-y+m)\frac{1}{|x-y+m|}-\frac{1}{\orderofg }\int_{x-y+m-\orderofg /2}^{x-y+m+\orderofg /2}u\frac{1}{|u|}du \right ) \cdot (1\otimes g^{-k})\Pi  \\
\left ( m\frac{1}{|m|}-\frac{1}{\orderofg }\int_{m-\orderofg /2}^{m+\orderofg /2}u\frac{1}{|u|}du \right ) \cdot (1\otimes g^{-k})\Pi  \\
&=\frac{1}{\orderofg } \sum_{k=1}^{\orderofg -1}\bracket{k-\frac{\orderofg }{2}} \cdot (1\otimes g^{-k})\Pi \\
&=- \bracket{1\otimes\frac {{g_\perp}^{-1}}{1-{g_\perp}^{-1}}}\Pi_2 -\frac12 \Pi_2
.
\end{aligned}$$
The propagator $P_2$ represents a loop from $x$ to itself.

We also denote 
$$
P_3=- \bracket{1\otimes\frac {{g_\perp}^{-1}}{1-{g_\perp}^{-1}}}\Pi_2=P_{12} +\Pi_2 = P_2 +\frac{1}{2}\Pi_2.
$$

\begin{rem}
Choose a $g$-invariant complex structure on $\R^{2n-2k}$ %(for example, one can choose a complex structure and do a $g$-average)
and identify it with $\mathbb C ^{n-k}$.  
Then the matrix of $g$ is  unitary and the matrix of $\frac {1+{g_\perp}^{-1}}{1-{g_\perp}^{-1}}$ is anti-Hermitian. Then
$$\begin{aligned}
P_2 
&=  -\frac12  \left ( 1\otimes \frac {1+{g_\perp}^{-1}}{1-{g_\perp}^{-1}} \right ) \sum_{i} \left ( \frac{\partial }{\partial x^i} \otimes \frac{\partial }{\partial y^i}-\frac{\partial }{\partial y^i} \otimes \frac{\partial }{\partial x^i}\right ) 
\\
%&=  -\frac12  \left ( 1\otimes \frac {1+{g_\perp}^{-1}}{1-{g_\perp}^{-1}} \right ) \sum_{i}2 \left ( \frac{\partial }{\partial z^i} \otimes \frac{\partial }{\partial \overline{z}^i}-\frac{\partial }{\partial \overline{z}^i} \otimes \frac{\partial }{\partial z^i}\right ) 
%\\&= - \sum_{i} \left ( \frac{\partial }{\partial z^i} \otimes \frac {1+{g_\perp}^{-1}}{1-{g_\perp}^{-1}} \frac{\partial }{\partial \overline{z}^i}-\frac{\partial }{\partial \overline{z}^i} \otimes \frac {1+{g_\perp}^{-1}}{1-{g_\perp}^{-1}} \frac{\partial }{\partial z^i}\right ) 
%\\&= - \sum_{i,j} \left ( \frac{\partial }{\partial z^i} \otimes \overline{\left ( \frac {1+{g_\perp}^{-1}}{1-{g_\perp}^{-1}} \right )}^{ji} \frac{\partial }{\partial \overline{z}^j}-\frac{\partial }{\partial \overline{z}^i} \otimes \left ( \frac {1+{g_\perp}^{-1}}{1-{g_\perp}^{-1}} \right )^{ji} \frac{\partial }{\partial z^j}\right ) \\
&= \sum_{i,j}\left ( \frac {1+{g_\perp}^{-1}}{1-{g_\perp}^{-1}} \right )^{ij} \left ( \frac{\partial }{\partial z ^i} \otimes  \frac{\partial }{\partial \overline{z}^j} + \frac{\partial }{\partial \overline{z} ^j} \otimes  \frac{\partial }{\partial z^i} \right ).
\end{aligned}$$
This is exactly the one used in \cite{fe:g-index} to construct a trace.
\end{rem}

\subsubsection{Configuration spaces}
The propagators are not smooth on $S^1\times S^1$, but can be lifted smoothly to compactified configuration spaces of $S^1$ (as shown for Chern-Simons type theory in \cites{Kontsevich-diagram,axelrod1994chern,GJ}). Correlation functions  in our model will be expressed via integrals on  these  configurations spaces.  %Boundary contributions of configuration spaces corresponds to when point collides.

We first introduce the geometric notation, following the presentation in \cite{Localized}.   Let
$$
\Conf_{S^1}[m+1]=\{(p_0,p_1, \cdots, p_m)\in (S^1)^{m+1}\,|\, p_i\neq p_j\} \subset (S^1)^{m+1}
$$
be the configuration space of $m+1$ ordered points on the circle $S^1$. Let
$$
 \Cyc_{S^1}[m+1]=\{(p_0,p_1,\cdots,p_m)\in \Conf_{S^1}[m+1]\,|\, p_0, \cdots, p_m\ \text{are anti-clockwise cyclic ordered} \}
$$
be the connected component of $\Conf_{S^1}[m+1]$ where points have the prescribed cyclic order.

$\Cyc_{S^1}[m+1]$ has a natural compactification as follows. Let us identify $S^1=\R/\Z$ so that the total length of $S^1$ is 1. Given a cyclic ordered points $(p_0,\cdots, p_m)$ on $S^1$, let $u_{i,i+1}$ denote the oriented distance by traveling  from $p_i$ to $p_{i+1}$ anti-clockwise ($p_{m+1}\equiv p_0$). Let $\Delta_m$ denote the standard simplex
$$
\Delta_m=\{(\lambda_0,\cdots, \lambda_m)\in \R^{n+1}\,|\, \lambda_i\geq 0,\ \  \sum_{i=0}^m \lambda_i=1\}.
$$
Let $\Delta_{m}^o$ be the interior of $\Delta_m$. Then we have a natural identification
\begin{align*}
\Cyc_{S^1}[m+1] & \iso S^1\times \Delta_{m}^o\\
(p_0,\cdots,p_m) & \mapsto \{p_0\}\times (u_{0,1},u_{1,2},\cdots, u_{m,0}).
\end{align*}
This allows us to compactify $\Cyc_{S^1}[m+1]$ by  $S^1\times \Delta_{m}$, which will be denoted by $S^1_{cyc}[m+1]$.

Similarly, we can compactify the whole space $\Conf_{S^1}[m]$, denoted by ${S^1}[m]$. ${S^1}[m]$ is a manifold with corners. Alternately, it could be constructed via successive real-oriented blow ups of diagonals in $(S^1)^m$. % as described in \cite{axelrod1993chern,kontsevich1994feynman,getzler1994operads}. 
In particular, it carries a natural blow-down map
$$\pi:S^1[m]\rightarrow (S^1)^m.
$$

For example,  $S^1[2]$ is parametrized as a cylinder
$$S^1[2]=\{(e^{2\pi i\theta},u)\,|\,0\leq\theta<1,0\leq u\leq 1\}.$$
With this parametrization, the blow down map is
$$\pi:S^1[2]\rightarrow (S^1)^2,  \quad (e^{2\pi i\theta},u)\mapsto (e^{2\pi i\theta},e^{2\pi i\theta+u}).$$
If we denote an element of $(S^1)^2$ by two ordered points $(p_0, p_1)$, then $u$ is the oriented distance by traveling from $p_0$ to $p_1$ anti-clockwise.

\subsection{Vacuum expectation}
In this section, we calculate the vacuum expectation which encodes the partition function of  free  quantum mechanics on an orbifold. We give a heuristic explanation of the factor $\det(1-{g_\perp}^{-1})^{-1}$ that will appear in the trace map and orbifold index later. This factor is precisely the vacuum expectation.

In the following discussions, we will identify
$$\Omega^\bullet(S^1, \mathcal L_g)\simeq \{\omega \in \Omega^\bullet(\R)\otimes V\,|\, \omega(x+1)=g\cdot\omega(x)\}$$
by pulling back $\mathcal L_g$ along the universal cover $\R\to S^1$.%, and identifying it with a trivial vector bundle with  on $\R$ with fiber $V$.
%We will always assume this identification in the following discussions. 

We decompose the triple $(V,\omega,g)$  into direct sum of triples $(V_i,\omega_i,g_i)$ 
with $\dim V_i=2$.   The vacuum expectation will be the product of the vacuum expectation associated to each subspace of this decomposition.  
Let us assume $W=V_i$ for some $i$ and compute the vacuum expectation for $W$. 

Firstly, consider the case $g\neq 1$. Choose a   basis $\{u,v\}$ of $W$ such that $\omega(\partial_u,\partial_v)=1$. Then the matrix of $g$ on $W$ is of the form 
$$
g=\begin{pmatrix}
\cos\bracket{\frac{2\pi l}{\orderofg }} &    \sin\bracket{\frac{2\pi l}{\orderofg }}\\
-\sin\bracket{\frac{2\pi l}{\orderofg }} &     \cos\bracket{\frac{2\pi l}{\orderofg }}\\
\end{pmatrix}.
$$
We can write the $g$-action as $g(u+iv)=e^{\frac{2\pi il}{\orderofg }}(u+iv)$ by 
using a complex coordinate. 

Only $0$-forms contribute to the action. Such forms have an orthonormal basis 
\begin{align*}
&\quad&& \text{Re} \left( e^{\frac{2\pi il}{\orderofg } t}e^{2\pi imt}(u+iv)\right), && \text{Im}\left( e^{\frac{2\pi il}{\orderofg } t}e^{2\pi imt}(u+iv)\right), &&m\in \Z.&&\quad& %\\
%&\quad&& \text{Re} (e^{\frac{2\pi il}{\orderofg }}e^{2\pi imt}(x+iy)) dt, && \text{Im}( e^{\frac{2\pi il}{\orderofg }}e^{2\pi imt}(x+iy) dt), &&m\in \Z.&&\quad&
\end{align*}
The vacuum expectation is computed via the standard Gaussian integral and is naively given by 
\begin{align*}
\frac{1}{\sqrt{\det(d)}}
=\prod_{m\in \Z} \bigg | 2\pi\bracket{m+\frac{l}{\orderofg }}\bigg |^{-1}
=\bbracket{\prod_{m\neq 0}  |2\pi m |\cdot \prod_{m\neq 0} \bracket{1+\frac{l}{m\orderofg }} \cdot \frac{2\pi l}{\orderofg }}^{-1}
=\prod_{m= 1}^{\infty} (2\pi m)^{-2}  \cdot\bbracket{ 2\sin\bracket{\frac{\pi l}{\orderofg }}}^{-1}.
\end{align*}

The factor $\displaystyle \prod\limits_{m= 1}^{\infty} (2\pi m)^{-2}$ is divergent and needs regularization. We use the standard zeta function regularization. Consider the function  
$$
\zeta_1(s)=\sum_{m=1}^\infty (2\pi m)^{-2s}
$$
which converges for $\text{Re}(s)>1$  and can be analytically continued to a neighborhood of $0$. Taking derivative at $0$, we get the formal expression $\zeta_1'(0)=\displaystyle \sum\limits_{m=1}^\infty \log((2\pi m)^{-2})$. So we can use analytic continuation to define the above divergent product by
$$
\prod\limits_{m= 1}^{\infty} (2\pi m)^{-2}:=\text{exp}\,( \zeta_1'(0)).
$$

The function $\zeta_1(s)$ is related to the Riemann zeta function $$
\zeta(s)=\sum\limits_{m=1}^\infty  m^{-s}$$ 
by $\zeta_1(s)=(2\pi)^{-2s}\zeta(2s)$. So 
$$
\zeta_1'(0)=-2\log(2\pi)\zeta(0)+2\zeta'(0)=-2\log(2\pi)\cdot(-\frac{1}{2})+2\cdot \bracket{-\frac{1}{2}\log(2\pi)}=0.
$$
Thus the factor $\displaystyle \prod\limits_{m= 1}^{\infty} (2\pi m)^{-2}$ is regularized by $\text{exp}\, (\zeta_1'(0))=1$.

The regularized vacuum expectation is
$$
\frac{1}{\sqrt{\det(d)}}
=\bbracket{2\sin\bracket{\frac{\pi l}{\orderofg }}}^{-1}
=\det\begin{pmatrix}
1-\cos\bracket{-\frac{2\pi l}{\orderofg }} &  -\sin\bracket{-\frac{2\pi l}{\orderofg }}\\
\sin\bracket{-\frac{2\pi l}{\orderofg }} & 1- \cos\bracket{-\frac{2\pi l}{\orderofg }}\\
\end{pmatrix}^{-1}
=\det(1-g^{-1})^{-1}.
$$

Secondly, if $g=1$, the   vacuum expectation is similarly regularized   
$$
\frac{1}{\sqrt{\det(d)}}=
\prod_{m= 1}^{\infty} (2\pi m)^{-2}:=1.
$$

We have calculated the vacuum expectation for a special $2$-dimensional subspace $W\subset V$. We see the vacuum expectation associated to   $V$ is 
$$
\det(1-{g_\perp}^{-1})^{-1}.
$$

%In this paper, we will abuse the notation by writing $1-{g|_{\R^{2n-2k}}}^{-1}$ as $1-g^{-1}$.

\subsection{Correlation map}

\subsubsection{Twisted chains}
Let $A$ be a    $\C$-algebra with unit, 
and $G$ be a finite group acting on $A$. 
For   $g\in G$ and $a_0\otimes a_1\otimes\cdots\otimes a_m \in A^{\otimes m+1}$, define the $g$-twisted Hochschild differential 
\begin{align*}
b_g(a_0\otimes a_1\otimes\cdots\otimes a_m )=&
(-1)^ma_ma_0\otimes a_1\otimes \cdot\cdot\cdot\otimes a_{m-1}+a_0 g(a_1)\otimes a_2\otimes\cdot\cdot\cdot\otimes a_{m} \\&
+\sum_{i=1}^{m-1}(-1)^ia_0\otimes \cdot\cdot\cdot \otimes a_i a_{i+1}\otimes\cdot\cdot\cdot \otimes a_m
\end{align*}
and the $g$-twisted  Connes operator
\begin{align*}
B_g(a_0 \otimes a_1\otimes\cdot\cdot\cdot\otimes a_m)=&
1\otimes a_0\otimes a_1\otimes \cdot\cdot\cdot \otimes a_{m} \\&
+\sum_{i=1}^{m}(-1)^{mi} 1\otimes g(a_{m-i+1})\otimes\cdot\cdot\cdot \otimes g(a_m)\otimes a_0\otimes   \cdot\cdot\cdot \otimes a_{m-i}.
\end{align*}
Let  $$C_{-m}(A):=A\otimes (A/\C 1)^{\otimes m}$$
be the space of  cyclic $m$-chains. $b_g$ descends to a map $b_g: C_{-m}(A)\to C_{-m+1}(A)$ which satisfies $b_g^2=0$. 

\begin{defn}\label{defn-g-Hochschild}
The $g$-twisted Hochschild complex is
$$
C_{-\bullet}(A,g)=(C_{-\bullet}(A),b_g).
$$
\end{defn}

An element $h\in G$ acts on  $C_{-m}(A)$ diagonally by 
$$
h(a_0\otimes a_1\otimes\cdots\otimes a_m)= h(a_0)\otimes h(a_1)\otimes\cdots\otimes h(a_m).
$$
Denote $C_{-m}(A)^g$ to be the $g$-invariant subspace of $C_{-m}(A)$.   $B_g$ descends to a map $B_g: C_{-m}(A)^g\to C_{-m-1}(A)^g$ which satisfies $B_g^2=0$ and $B_gb_g+b_gB_g=0$.

\begin{defn}
The $g$-twisted periodic cyclic complex is
$$
CC_{-\bullet}^{per}(A,g)=(C_{-\bullet}(A)^g[u,u^{-1}], b_g+uB_g).
$$
\end{defn}

%\Yang \begin{rem}For different representatives in a conjugate class $[g]\in\Conj(G)$, there are natural conjugate actions linking $g$-twisted chains which are isomorphisms.\end{rem}

\iffalse
\begin{rem}
The algebra of formal functions on the singular space $V/G$ has a smooth refinement $\mathcal O(V)\rtimes G$, which is a twisted group algebra.  The Hochschild/cyclic chain complex of a twisted group algebra naturally splits \cite{brodzki2016periodic} ({\color{red} check about the literature again})
$$
C_{-\bullet} (A\rtimes G) 
\simeq \bigoplus_{[g]\in \Conj(G)} C_{-\bullet} (A,g) ^{C(g)},
$$
$$
CC_{-\bullet}^{per} (A\rtimes G) 
\simeq \bigoplus_{[g]\in \Conj(G)} CC_{-\bullet}^{per} (A,g) ^{C(g)}.
$$
Each term in the right hand side  corresponds to a sector of the inertia orbifold. 
\end{rem}
\fi

Let us consider matrix-valued observables, where the matrix coefficients come from  a vector bundle $E$  coupled to the system. In the plain case \cite{Localized}, taking trace provides a natural quasi-isomorphism $C_{-\bullet}(\gl_r(A))\to C_{-\bullet}(A)$.  In the orbifold case, we need a $g$-twisted trace for its matrix coefficients. Let  $M_0,\cdots,M_m\in \gl_r$ be matrices corresponding to $m+1$ ordered point on a circle with a $g$-twist between $M_0$ and $M_1$ (See picture \ref{trace}).  Then we define the $g$-twisted trace 
$$ 
\mathrm{tr}_g (M_0\otimes M_1 \otimes \cdots \otimes M_m) := \mathrm{tr}(M_0  g M_1\cdots  M_m).
$$
%by first multiplying $M_0$, $g$, $M_1,\cdots M_m$ and then taking the matrix trace. 
Here $g$ is the cyclic structure of $E$  represented by a matrix, so we can insert $g$ and take trace. 
$\mathrm{tr}_g$ satisfies
\begin{equation}\label{Trace}
\mathrm{tr}_g (M_0\otimes M_1 \otimes \cdots \otimes M_m) = \mathrm{tr}(gM_1g^{-1}\otimes M_2 \otimes \cdots \otimes M_m\otimes M_0). \tag{$\mathrm{tr}_g$}
\end{equation}

\begin{center}
\begin{figure}
\tikzset{every picture/.style={line width=0.75pt}} %set default line width to 0.75pt        

\begin{tikzpicture}[x=0.75pt,y=0.75pt,yscale=-1,xscale=1]

%Shape: Circle [id:dp4605099287172436] 
\draw   (71,144.62) .. controls (71,101.75) and (105.75,67) .. (148.62,67) .. controls (191.48,67) and (226.23,101.75) .. (226.23,144.62) .. controls (226.23,187.48) and (191.48,222.23) .. (148.62,222.23) .. controls (105.75,222.23) and (71,187.48) .. (71,144.62) -- cycle ;
%Shape: Circle [id:dp8675283395126523] 
\draw  [fill={rgb, 255:red, 0; green, 0; blue, 0 }  ,fill opacity=1 ] (198.3,87) .. controls (198.3,85.76) and (199.31,84.75) .. (200.55,84.75) .. controls (201.79,84.75) and (202.8,85.76) .. (202.8,87) .. controls (202.8,88.24) and (201.79,89.25) .. (200.55,89.25) .. controls (199.31,89.25) and (198.3,88.24) .. (198.3,87) -- cycle ;
%Shape: Circle [id:dp33349790193549667] 
\draw  [fill={rgb, 255:red, 0; green, 0; blue, 0 }  ,fill opacity=1 ] (131.8,68.5) .. controls (131.8,67.26) and (132.81,66.25) .. (134.05,66.25) .. controls (135.29,66.25) and (136.3,67.26) .. (136.3,68.5) .. controls (136.3,69.74) and (135.29,70.75) .. (134.05,70.75) .. controls (132.81,70.75) and (131.8,69.74) .. (131.8,68.5) -- cycle ;
%Shape: Circle [id:dp2757944691103186] 
\draw  [fill={rgb, 255:red, 0; green, 0; blue, 0 }  ,fill opacity=1 ] (73.8,117.5) .. controls (73.8,116.26) and (74.81,115.25) .. (76.05,115.25) .. controls (77.29,115.25) and (78.3,116.26) .. (78.3,117.5) .. controls (78.3,118.74) and (77.29,119.75) .. (76.05,119.75) .. controls (74.81,119.75) and (73.8,118.74) .. (73.8,117.5) -- cycle ;
%Shape: Circle [id:dp028776516894871373] 
\draw  [fill={rgb, 255:red, 0; green, 0; blue, 0 }  ,fill opacity=1 ] (128.3,220) .. controls (128.3,218.76) and (129.31,217.75) .. (130.55,217.75) .. controls (131.79,217.75) and (132.8,218.76) .. (132.8,220) .. controls (132.8,221.24) and (131.79,222.25) .. (130.55,222.25) .. controls (129.31,222.25) and (128.3,221.24) .. (128.3,220) -- cycle ;
%Shape: Circle [id:dp8544287400628858] 
\draw  [fill={rgb, 255:red, 0; green, 0; blue, 0 }  ,fill opacity=1 ] (202.8,198.4) .. controls (202.8,197.16) and (203.81,196.15) .. (205.05,196.15) .. controls (206.29,196.15) and (207.3,197.16) .. (207.3,198.4) .. controls (207.3,199.64) and (206.29,200.65) .. (205.05,200.65) .. controls (203.81,200.65) and (202.8,199.64) .. (202.8,198.4) -- cycle ;
%Straight Lines [id:da6634265005224809] 
\draw    (218.62,144.62) -- (232.25,144.62) ;
%Shape: Circle [id:dp5841754942400138] 
\draw   (406,143.62) .. controls (406,100.75) and (440.75,66) .. (483.62,66) .. controls (526.48,66) and (561.23,100.75) .. (561.23,143.62) .. controls (561.23,186.48) and (526.48,221.23) .. (483.62,221.23) .. controls (440.75,221.23) and (406,186.48) .. (406,143.62) -- cycle ;
%Shape: Circle [id:dp8940539323938441] 
\draw  [fill={rgb, 255:red, 0; green, 0; blue, 0 }  ,fill opacity=1 ] (554.7,170.4) .. controls (554.7,169.16) and (555.71,168.15) .. (556.95,168.15) .. controls (558.19,168.15) and (559.2,169.16) .. (559.2,170.4) .. controls (559.2,171.64) and (558.19,172.65) .. (556.95,172.65) .. controls (555.71,172.65) and (554.7,171.64) .. (554.7,170.4) -- cycle ;
%Shape: Circle [id:dp3982684139081867] 
\draw  [fill={rgb, 255:red, 0; green, 0; blue, 0 }  ,fill opacity=1 ] (466.8,67.5) .. controls (466.8,66.26) and (467.81,65.25) .. (469.05,65.25) .. controls (470.29,65.25) and (471.3,66.26) .. (471.3,67.5) .. controls (471.3,68.74) and (470.29,69.75) .. (469.05,69.75) .. controls (467.81,69.75) and (466.8,68.74) .. (466.8,67.5) -- cycle ;
%Shape: Circle [id:dp3486223496020515] 
\draw  [fill={rgb, 255:red, 0; green, 0; blue, 0 }  ,fill opacity=1 ] (408.8,116.5) .. controls (408.8,115.26) and (409.81,114.25) .. (411.05,114.25) .. controls (412.29,114.25) and (413.3,115.26) .. (413.3,116.5) .. controls (413.3,117.74) and (412.29,118.75) .. (411.05,118.75) .. controls (409.81,118.75) and (408.8,117.74) .. (408.8,116.5) -- cycle ;
%Shape: Circle [id:dp9690903462642281] 
\draw  [fill={rgb, 255:red, 0; green, 0; blue, 0 }  ,fill opacity=1 ] (463.3,219) .. controls (463.3,217.76) and (464.31,216.75) .. (465.55,216.75) .. controls (466.79,216.75) and (467.8,217.76) .. (467.8,219) .. controls (467.8,220.24) and (466.79,221.25) .. (465.55,221.25) .. controls (464.31,221.25) and (463.3,220.24) .. (463.3,219) -- cycle ;
%Shape: Circle [id:dp4480956517013791] 
\draw  [fill={rgb, 255:red, 0; green, 0; blue, 0 }  ,fill opacity=1 ] (528.6,205) .. controls (528.6,203.76) and (529.61,202.75) .. (530.85,202.75) .. controls (532.09,202.75) and (533.1,203.76) .. (533.1,205) .. controls (533.1,206.24) and (532.09,207.25) .. (530.85,207.25) .. controls (529.61,207.25) and (528.6,206.24) .. (528.6,205) -- cycle ;
%Straight Lines [id:da3821194293417409] 
\draw    (553.62,143.62) -- (567.25,143.62) ;
%Straight Lines [id:da565648715959082] 
\draw    (308.75,144.62) -- (366.75,144.62) ;
\draw [shift={(368.75,144.62)}, rotate = 180] [color={rgb, 255:red, 0; green, 0; blue, 0 }  ][line width=0.75]    (10.93,-3.29) .. controls (6.95,-1.4) and (3.31,-0.3) .. (0,0) .. controls (3.31,0.3) and (6.95,1.4) .. (10.93,3.29)   ;

% Text Node
\draw (209.3,201.4) node [anchor=north west][inner sep=0.75pt]   [align=left] {$M_0$};
% Text Node
\draw (211.67,70) node [anchor=north west][inner sep=0.75pt]   [align=left] {$M_1$};
% Text Node
\draw (114.67,44.5) node [anchor=north west][inner sep=0.75pt]   [align=left] {$M_2$};
% Text Node
\draw (238.17,135.5) node [anchor=north west][inner sep=0.75pt]   [align=left] {$g$-twist};
% Text Node
\draw (121.19,227.12) node [anchor=north west][inner sep=0.75pt]   [align=left] {$M_m$};
% Text Node
\draw (532.85,208) node [anchor=north west][inner sep=0.75pt]   [align=left] {$M_0$};
% Text Node
\draw (449.67,43.5) node [anchor=north west][inner sep=0.75pt]   [align=left] {$M_2$};
% Text Node
\draw (562,169.4) node [anchor=north west][inner sep=0.75pt]   [align=left] {$g M_1 g^{-1}$};
% Text Node
\draw (573.17,134.5) node [anchor=north west][inner sep=0.75pt]   [align=left] {$g$-twist};
% Text Node
\draw (456.19,226.12) node [anchor=north west][inner sep=0.75pt]   [align=left] {$M_m$};
% Text Node
\draw (48,104) node [anchor=north west][inner sep=0.75pt]   [align=left] {$M_3$};
% Text Node
\draw (380,104) node [anchor=north west][inner sep=0.75pt]   [align=left] {$M_3$};
% Text Node
\draw (52,166) node [anchor=north west][inner sep=0.75pt]   [align=left] {$\cdots$};
% Text Node
\draw (386,166) node [anchor=north west][inner sep=0.75pt]   [align=left] {$\cdots$};

\end{tikzpicture}
\caption{A matrix $M_1$ passing through the $g$-twist}\label{trace}
\end{figure}
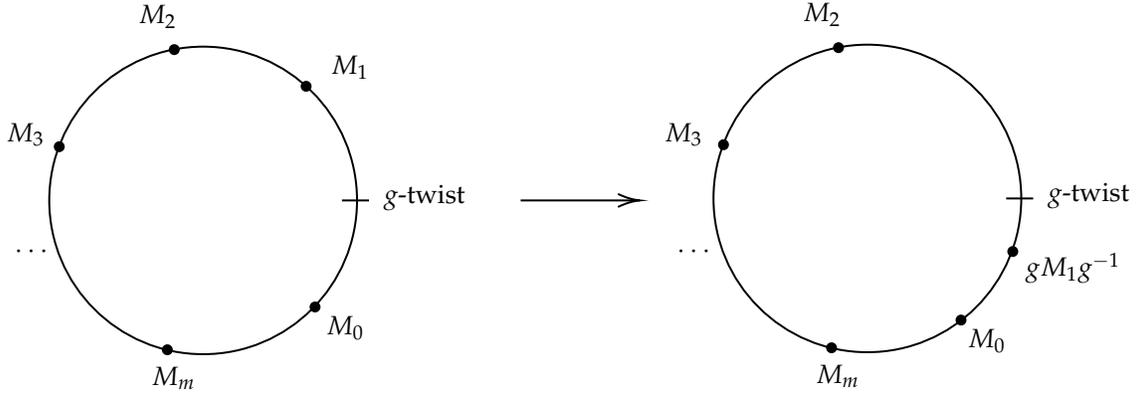
\end{center}

\subsubsection{Some operations on chains}

Define three operators (Einstein summation convention will be always assumed)
\begin{itemize}
    \item the de Rham differential at tangent direction of $\R^{2k}$
    $$\d:\hat{\Omega}^{-\*}_{{2n}}\rightarrow \hat{\Omega}^{-\*-1}_{{2n}},\quad \d =dy^i\partial_{y^i}$$
    \item  contraction with the bivector $\Pi_1$ (see the notation in Section \ref{TQM})
    $$\iota_{\Pi_1}:\hat{\Omega}^{-\*}_{{2n}}\rightarrow \hat{\Omega}^{-\*+2}_{{2n}}, \quad \iota_{\Pi_1}=\frac{1}{2}\omega^{ij}\iota_{\partial_{y^i}}\iota_{\partial_{y^j}}$$
    \item the BV operator 
    $$\Delta:=\mathcal{L}_{\Pi_1}=[\d,\iota_{\Pi_1}]:\hat{\Omega}^{-\*}_{{2n}}\rightarrow \hat{\Omega}^{-\*+1}_{{2n}}, \quad \Delta=\omega^{ij}\mathcal{L}_{\partial_{y^i}}\iota_{\partial_{y^j}}$$
\end{itemize}

These operators naturally extend to the 
 m-th tensor product of $\hat{\Omega}^{-\*}_{2n}$ 
  $$(\hat{\Omega}^{-\*}_{{2n}})^{\otimes m}:=\hat{\Omega}^{-\*}_{{2n}}\otimes_{\C\LS{\hbar}}\cdot\cdot\cdot\otimes_{\mathbb{C}\LS{\hbar}}\hat{\Omega}^{-\*}_{{2n}}$$
  by
   \begin{align*}\d   (a_1\otimes\cdot\cdot\cdot\otimes a_m)&:=\sum_{1\leq \alpha\leq m}\pm a_1\otimes\cdot\cdot\cdot\otimes \d   a_{\alpha}\otimes\cdot\cdot\cdot\otimes a_m\\
    \iota_{\Pi_1}(a_1\otimes\cdot\cdot\cdot\otimes a_m)&:=\frac{1}{2}\sum_{1\leq\alpha,\beta\leq m}\pm \omega^{ij}a_1\otimes\cdot\cdot\cdot\otimes\iota_{\partial_{y^i}}a_{\alpha}\otimes\cdot\cdot\cdot \otimes \iota_{\partial_{y^j}}a_{\beta}\otimes\cdot\cdot\cdot\otimes a_m \\
    \Delta(a_1\otimes\cdot\cdot\cdot\otimes a_m)&:=\sum_{1\leq\alpha,\beta\leq m}\pm\omega^{ij}a_1\otimes\cdot\cdot\cdot\otimes\mathcal{L}_{\partial_{y^i}}a_{\alpha}\otimes\cdot\cdot\cdot\otimes\iota_{\partial_{y^j}}a_{\beta}\otimes\cdot\cdot\cdot\otimes a_m.
    \end{align*}
    Here $\pm$ are Koszul signs by passing odd operators through the graded objects.

We then define contraction with propagators. This will be used to describe Feynman rules. 
\begin{defn} Let $\Omega^\*_{S^1[m]}$ be smooth differential forms on $S^1[m]$.
Let 
$$
P=a^{ij} f\otimes  {\partial_{y^i}}\otimes {\partial_{y^j}} \in C^{\infty}({S^1[2]})\otimes \R^{2n}\otimes \R^{2n},\quad 
P'= b^{ij}  {\partial_{y^i}}\otimes {\partial_{y^j}} \in   \R^{2n}\otimes \R^{2n}
$$
be propagators   representing a line and a self loop respectively. 
Define   $\Omega^\*_{S^1[m]}$-linear operators
  $$\partial_P,\partial_{P'}:\Omega^\*_{S^1[m]}\otimes (\hat{\Omega}^{-\*}_{{2n}})^{\otimes m}\rightarrow \Omega^\*_{S^1[m]}\otimes (\hat{\Omega}^{-\*}_{{2n}})^{\otimes m}$$
  by
  \begin{align*}
  \partial_{P}(a_1\otimes\cdots\otimes a_m)&:=\frac{1}{2}\sum_{1\leq\alpha\neq\beta\leq m}\pi^*_{\alpha\beta}(f)\otimes \bracket{a^{ij}a_1\otimes\cdots\otimes\mathcal{L}_{\partial_{y^i}}a_{\alpha}\otimes\cdots\otimes\mathcal{L}_{\partial_{y^j}}a_{\beta}\otimes\cdots \otimes a_m}.
%  D(a_1\otimes\cdots\otimes a_k)&:=\sum_{\alpha}\pm d\theta_{\alpha}\otimes \bracket{ a_1\otimes\cdots\otimes \d a_{\alpha}\otimes\cdots\otimes a_k}.
  \end{align*}
\begin{align*}
  \partial_{P'}(a_1\otimes\cdots\otimes a_m)&:=\frac{1}{2}\sum_{1\leq\alpha\leq m}   \bracket{b^{ij}a_1\otimes\cdots\otimes\mathcal{L}_{\partial_{y^i}}\mathcal{L}_{\partial_{y^j}}a_{\alpha}\otimes\cdots\otimes a_m}.
  \end{align*}
Here $a_i\in \hat\Omega^{-\*}_{{2n}}$ and $\pi_{\alpha\beta}:S^1[m]\rightarrow S^1[2]$ is the forgetful map to the two points indexed by $\alpha,\beta$. 
%$\theta_{\alpha}\in [0,1)$ is the parameter on the $S^1$ indexed by $\alpha$ and $d\theta_\alpha$ is viewed as a 1-form in $S^1[m]$ via the pull-back $\pi_{\alpha}:S^1[m]\rightarrow S^1$.
\end{defn}

We also denote
$$
\int_{S^1[m]}: \Omega^\*_{S^1[m]}\otimes (\hat{\Omega}^{-\*}_{{2n}})^{\otimes m}\to (\hat{\Omega}^{-\*}_{{2n}})^{\otimes m}
$$
by integrating out differential forms on $S^1[m]$.

\subsubsection{Free correlation map}

%Now we can formulate Feynman diagram integrals on $S^1[k]$ as in \cite{axelrod1993chern,kontsevich1994feynman,getzler1994operads}.

We generalize the construction of correlation maps in \cite{Localized} to the orbifold case.

\begin{defn} \label{defn-correlation}
Define the $g$-twisted free correlation map 
$$
\langle-\rangle_{free}^g:
C_{-\bullet}(\gl_r(\mathcal{W}_{2n}),g)^{C(g)}
\rightarrow
 \bracket{\hat{\Omega }^{-\bullet }_{2k}}^{C(g)}
$$
by
$$
\langle \hat{\mathcal{O}}_0\otimes \hat{\mathcal{O}}_1\otimes \cdots\otimes \hat{\mathcal{O}}_m\rangle_{free}^g=
\mathrm{det}(1-{g_\perp}^{-1})^{-1}\cdot
\sigma_z \bracket{
\mathrm{tr}_g \int_{S^1_{cyc}[m+1]} d\theta_0 \cdots d\theta_{m}e^{\hbar\partial_{P_1}}e^{\hbar\partial_{P_2}}(\hat{\mathcal{O}}_0\otimes \d  \hat{\mathcal{O}}_1\otimes \cdots \otimes \d  \hat{\mathcal{O}}_m)
}.
$$
Here $\partial_{P_1},\partial_{P_2}$ are applied to all $m+1$ factors,
and $\sigma_z$ is the symbol map which sets $z^j,dz^j$ to $0$.

\end{defn}

\begin{lem}\label{lem-g-intertwine}
The following properties hold
\begin{itemize}
\item[(1)] The $g$-twisted free correlation map intertwines the Hochschild boundary operator $b_g$ with the operator $\hbar\Delta$  
$$\hbar\Delta\langle  \hat{\mathcal{O}}\rangle_{free}^g=\langle b_g( \hat{\mathcal{O}})\rangle_{free}^g
.$$

\item[(2)] The $g$-twisted free correlation map intertwines the Connes operator $B_g$ with the de Rham differential $\d  $ 
$$\d   \langle  \hat{\mathcal{O}}\rangle_{free}^g=\langle B_g( \hat{\mathcal{O}})\rangle_{free}^g
.$$
\end{itemize}
Here $\hat{\mathcal{O}}=\hat{\mathcal{O}}_0 \otimes \hat{\mathcal{O}}_1\otimes\cdots\otimes \hat{\mathcal{O}}_m$.
\end{lem}

\begin{proof}
Split $\gl_r(\mathcal{W}_{2n})$ into the tensor product $\mathcal{W}_{2k}\otimes\mathcal{W}_{2n-2k}\otimes\gl_r$, and let $\hat{\mathcal{O}}_i=a_i\otimes b_i \otimes M_i$ with respect to the decomposition. 
The correlation map is built of three parts
\begin{itemize}
\item[(1)] $\tau_0: (\mathcal{W}_{2k})^{\otimes m+1}\to \hat{\Omega }^{-m }_{2k},$
$$\tau_0(a_0\otimes \cdots \otimes a_m ) =\mult \bracket{
\int_{S^1_{cyc}[m+1]} d\theta_0 \cdots d\theta_{m}e^{\hbar\partial_{P_{11}}}(a_0\otimes \d a_1\otimes \cdots \otimes \d a_m)}.$$
Here $\mult$ is the map $a_0\otimes \cdots \otimes a_m \mapsto a_0 \cdots a_m$. %$\tau_0$ is the free expectation map in \cite{Localized}.
\item[(2)]  $\tau_1: (\mathcal{W}_{2n-2k})^{\otimes m+1}\to \C(\!(\hbar)\!),$
$$\tau_1(b_0\otimes \cdots \otimes b_m )
= 
\mathrm{det}(1-{g_\perp}^{-1})^{-1}\cdot
\sigma_z(\mult(e^{\hbar\partial_{P_{12}}}e^{\hbar\partial_{P_2}}(b_0\otimes \cdots \otimes b_m))).$$
Here $\partial_{P_{12}}$ represents a line and $\partial_{P_2}$ represents a self loop. 
Recall $P_{3}=P_{12}+ \Pi_2$, we have
\begin{align*}
\tau_1(b_0\otimes \cdots \otimes b_m )
&= 
\mathrm{det}(1-{g_\perp}^{-1})^{-1}\cdot
\sigma_z(e^{\hbar\partial_{P_{12}}}\mult(b_0\otimes \cdots \otimes b_m)) \\
&= 
\mathrm{det}(1-{g_\perp}^{-1})^{-1}\cdot
\sigma_z(e^{\hbar\partial_{P_{3}}}\mult(e^{- \hbar\partial_{\Pi_2}}(b_0\otimes \cdots \otimes b_m))) \\
&=\mathrm{det}(1-{g_\perp}^{-1})^{-1}\cdot 
\sigma_z(e^{\hbar\partial_{P_3}} (b_0\hat\star \cdots \hat\star b_m)).
\end{align*}
%Here we trait $\partial_{\Pi_2}$ as a propagator representing a self loop. 
Here the associative product $\hat\star$ is given by 
$$b_1\hat \star b_2=\mult(e^{- \hbar \cdot  \partial_{\Pi_2}}(b_1\otimes b_2 )).$$
We also denote 
$$
\tau_1'(b_0\otimes \cdots \otimes b_m )
=\mathrm{det}(1-{g_\perp}^{-1})^{-1}\cdot 
\sigma_z(e^{\hbar\partial_{P_3}} (b_0 \star \cdots  \star b_m))
$$
by using the star product
$$b_1  \star b_2=\mult(e^{  \hbar \cdot  \partial_{\Pi_2}}(b_1\otimes b_2 )).$$
\item[(3)] $\mathrm{tr}_{g}: (\gl_r)^{\otimes m+1}\to \C,$
$$\mathrm{tr}_{g}(M_0\otimes M_1\otimes \cdots \otimes M_m) = 
\mathrm{tr}(M_0 g M_1\cdots M_m).$$
\end{itemize}
Then 
$$\langle \hat{\mathcal{O}}_0 \otimes \hat{\mathcal{O}}_1\otimes \cdots\otimes \hat{\mathcal{O}}_m\rangle_{free}^g=
\tau_0(a_0\otimes \cdots \otimes a_m)
\tau_1(b_0\otimes \cdots \otimes b_m)
\mathrm{tr}_{g}(M_0\otimes \cdots \otimes M_m).$$
$\tau_0$ is the free expectation map in \cite{Localized}. 
The proof follows from Lemma 3.5, Lemma 3.6 in \cite{Localized}, 
%a calculation of $\tau_0$, 
Equation (1.9) in \cite{fe:g-index}   and  \eqref{Trace}.   
\end{proof}

\subsubsection{Interactive correlation map}\label{interactive-correlation}
We now extend the trace to certain Lie algebra cochains following \cite{Localized} . %\Yang See Section \ref{Liealg} for the definition of relative Lie algebra cochains.

Recall the Weyl algebra $\mathcal{W}_{2n}=\mathbb{C}[[y^i,z^j]]\LS{\hbar}$. Denote $\mathcal{W}^+_{2n}=\mathbb{C}[[y^i,z^j]]\PS{\hbar}$,   
and its $g$-invariant part    by $(\mathcal{W}^+_{2n})^g$. 
The Moyal-Weyl product $\star$  induces a Lie algebra structure with Lie bracket defined by
$$
[f,g]:=[f,g]_\star:=\frac{1}{\hbar}(f\star g-g\star f).
$$

Let $\Sp_{2n}^g$ be the group of symplectic linear transformations on $\R^{2n}$ that commute with $g$. It acts on Weyl algebras naturally by inner automorphisms. 
Its Lie algebra $\mathfrak{sp}_{2n}^g$ can be identified with $g$-invariant quadratic polynomials in $\C[y^i,z^j]^g = \C[y^i] \oplus \C[z^j]^g$, 
%$\C[y^1,\cdots,y^{2k}]\oplus \C[z^{2k+1},\cdots,z^{2n}]^g$. 
which defines a natural  embedding 
$$
\mathfrak{sp}_{2n}^g
\into (\mathcal{W}^+_{2n})^g.
$$
%It acts on $(\mathcal{W}^+_{2n})^g$ by the corresponding inner derivations.

Define a Lie algebra $\mathfrak{g}$
$$\mathfrak{g}:=(\mathcal{W}^+_{2n})^g \cdot\text{Id}+\hbar \gl_r((\mathcal{W}^+_{2n})^g)$$
and a subalgebra $\mathfrak{h}\subset \mathfrak{g}$
$$
\mathfrak{h}
%=\sp_{2n}^g\oplus \hbar\gl_r\oplus Z(\mathfrak{g})
=\sp_{2k}\oplus\sp_{2n-2k}^g\oplus \hbar\gl_r\oplus \mathbb{C}\oplus\oplus_{i>1} \hbar^i\mathbb{C}.
$$
%$$\mathfrak{h}=\sp_{2k}\oplus\sp_{2n-2k}^g\oplus \hbar\gl_r\oplus Z(\mathfrak{g})=\sp_{2k}\oplus\sp_{2n-2k}^g\oplus \hbar\gl_r\oplus \mathbb{C}\oplus\oplus_{i>1} \hbar^i\mathbb{C}.$$
%\Yang Here $Z(\mathfrak{g})$ is the center of $\mathfrak{g}$.  
%Note that these Lie algebras relies on $\mathfrak{g}$.

Let us recall the shuffle product on tensors, which we denote by $\times_{sh}$.  The shuffle product of a $p$-tensor with a $q$-tensor is a $(p+q)$-tensor defined by
\begin{align*}
(v_1\otimes \cdots \otimes v_p)\times_{sh} (v_{p+1}\otimes \cdots \otimes v_{p+q})
=\sum_{\sigma\in Sh(p,q)}v_{\sigma^{-1}(1)}\otimes \cdots \otimes v_{\sigma^{-1}(p+q)}.
\end{align*}
Here $Sh(p,q)$ is the subset of the permutation group $S_{p+q}$ given by $(p,q)$-shuffles
$$
Sh(p,q)=\{\sigma \in S_{p+q}\,|\, \sigma(1)<\cdots <\sigma (p)\ \text{and}\ \sigma(p+1)<\cdots<\sigma(p+q)\}.
$$

%({\color{red} define $\Theta$ first and also cite \cite{Localized}})

%Let us recall the   universal flat connection $\widehat{\Theta}$ in \cite{Localized}, which is used as an interaction term of topological quantum mechanics. 

In \cite{Localized}, the interaction term of topological quantum mechanics is encoded in the universal flat connection  $\widehat{\Theta}$ at Lie algebra level.
Let $\Id:\mathfrak{g}\rightarrow \mathfrak{g}$ be the identity map, which is viewed as a 1-chain in $C^1_{\Lie}(\mathfrak{g};\mathfrak{g})$ denoted by $\widehat{\Theta}$. It solves the Maurer-Cartan equation
$$
    \pa \widehat{\Theta}+{\frac1 2}[\widehat{\Theta},\widehat{\Theta}]=0
$$
in the dg Lie algebra $(C^\*_{\Lie}(\mathfrak{g})\otimes \g, \pa, [-,-])$, where $\pa$ is  the \CE differential on the first factor and the Lie bracket is taken in the second factor. When we identify 
$$
C^\*_{\Lie}(\mathfrak{g};\mathfrak{g})\iso C^\*_{\Lie}(\mathfrak{g})\otimes \g
$$
as a vector space, the usual \CE differential $\pa_{\Lie}$ on $C^\*_{\Lie}(\mathfrak{g};\mathfrak{g})$ becomes
$$\partial_{\Lie}=\partial+[\widehat{\Theta},-].$$

\iffalse
Let $\Id:\mathfrak{g}\rightarrow \mathfrak{g}$ be the identity map, which is viewed as a 1-chain in $C^1_{\Lie}(\mathfrak{g};\mathfrak{g})$ denoted by $\widehat{\Theta}$.

Let $(C^\*_{\Lie}(\mathfrak{g}),\pa)$ be the dg algebra with $\pa$ the \CE differential with respect to the trivial representation. We have an induced dg Lie algebra
$$
(C^\*_{\Lie}(\mathfrak{g})\otimes \g, \pa, [-,-])
$$
where $[-,-]$ is the Lie bracket on the $\g$-factor. Then $\widehat{\Theta}$ solves the Maurer-Cartan equation
$$
    \pa \widehat{\Theta}+{\frac1 2}[\widehat{\Theta},\widehat{\Theta}]=0.
$$
As vector spaces, we can identify
$$
C^\*_{\Lie}(\mathfrak{g};\mathfrak{g})\iso C^\*_{\Lie}(\mathfrak{g})\otimes \g.
$$
Under this identification, the usual \CE differential $\pa_{\Lie}$ on $C^\*_{\Lie}(\mathfrak{g};\mathfrak{g})$ is
$$\partial_{\Lie}=\partial+[\widehat{\Theta},-].$$
\fi

\begin{defn}\label{defn:interaction}
Define the $g$-twisted interactive correlation map 
$$\langle-\rangle_{int}^g\in
C_{\mathrm{Lie}}^\bullet \bracket{\mathfrak{g};
\Hom_{\mathbb{C}(\!(\hbar)\!)}\bracket{
CC_{-\bullet}^{per}(\gl_r(\mathcal{W}_{2n}),g)^{C(g)}
,
 \bracket{\hat{\Omega }^{-\bullet }_{2k}}^{C(g)}
}}$$
by
\begin{align*}
\langle \hat{\mathcal{O}}_0  \otimes \hat{\mathcal{O}}_1\otimes \cdots \otimes \hat{\mathcal{O}}_m\rangle_{int}^g
&:=\abracket{\sum_{j\geq 0} \hat{\mathcal{O}}_0  \otimes \bracket{\bracket{\hat{\mathcal{O}}_1\otimes \cdots \otimes \hat{\mathcal{O}}_m}  \times_{sh} (\widehat\Theta/\hbar)^{\otimes j} }}_{free}^g\\
&=\sum_{j\geq 0}\sum_{ \substack{i_0,\cdots,i_m\geq 0\\ i_0+\cdots+i_m=j}} \abracket{\widehat{\mathcal{O}}_0   \otimes (\widehat\Theta/\hbar)^{\otimes i_0}\otimes \hat{\mathcal{O}}_1\otimes (\widehat\Theta/\hbar)^{\otimes i_1}\otimes \cdots \otimes \hat{\mathcal{O}}_m\otimes(\widehat\Theta/\hbar)^{\otimes i_m}}_{free}^g.
\end{align*} 
\end{defn}

In the above expression, we view $\hat \Theta$ as a Lie algebra 1-cochain valued in $\gl_r((\mathcal{W}_{2n})^g)$, i.e., an element in $C^1(\mathfrak{g}, \gl_r((\mathcal{W}_{2n})^g))$, using the natural embedding $\mathfrak{g}\subset \gl_r((\mathcal{W}_{2n})^g)$.   Therefore the following component of $\langle-\rangle_{int}^g$
$$
\hat{\mathcal{O}}_0 \otimes \hat{\mathcal{O}}_1\otimes \cdots \otimes \hat{\mathcal{O}}_m\mapsto 
\abracket{ \hat{\mathcal{O}}_0 \otimes \bracket{\bracket{\hat{\mathcal{O}}_1\otimes \cdots \otimes \hat{\mathcal{O}}_m}  \times_{sh} (\widehat\Theta/\hbar)^{\otimes j} }}_{free}^g
$$
provides an element in 
$$\langle-\rangle_{int}^g\in
C_{\mathrm{Lie}}^j \bracket{\mathfrak{g};
\Hom_{\mathbb{C}(\!(\hbar)\!)}\bracket{
CC_{-\bullet}^{per}(\gl_r(\mathcal{W}_{2n}),g)^{C(g)}
,
 \bracket{\hat{\Omega }^{-\bullet }_{2k}}^{C(g)}
}}.
$$
Koszul sign convention is always assumed to organize such a map into a Lie algebra cochain.

\begin{thm}
$\langle-\rangle_{int}^g$ is closed under the differential $\partial_{\mathrm{Lie}}+b_g-\hbar\Delta$ and $B_g-\d $.
\end{thm}
\begin{proof}
See the proof of Theorem 3.8 in \cite{Localized}.
\end{proof}

We $\mathbb{C} [u,u^{-1}]$-linearly  extend $\abracket{-}_{int}^g$ to an element
$$\langle-\rangle_{int}^g\in
C_{\mathrm{Lie}}^\bullet \bracket{\mathfrak{g};
\Hom_{\mathbb{C}(\!(\hbar)\!)}\bracket{
CC_{-\bullet}^{per}(\gl_r(\mathcal{W}_{2n}),g)^{C(g)}
,
 \bracket{\hat{\Omega }^{-\bullet }_{2k}}^{C(g)}
[u,u^{-1}]
}}.$$
It follows that 

\begin{thm}\label{thm-interaction-S1}
$\abracket{-}^g_{int}$ is closed under the differential
$(\partial_{\mathrm{Lie}}+b_g-\hbar\Delta)+u(B_g-\d ).$
\end{thm}

\section{Orbifold algebraic index}
As an application, we show how to establish the orbifold algebraic index theorem %by Pflaum-Posthuma-Tang 
\cites{Pflaum2007An,fst} in terms of the topological quantum mechanical model. This generalizes the work \cite{Localized} to the orbifold case. We follow the strategy and presentation  in  \cite{Localized}, and illustrate the extensions to orbifolds.

\subsection{Twisted universal index}\label{sec: GM}

\subsubsection{Twisted universal trace}
Denote $\mathbb{K}:=\mathbb{C}\LS{\hbar}[u,u^{-1}].$
\begin{defn}\label{BV-integration} We define the following Berezin integration map by
\begin{align*}
  \int_{BV}: \hat{\Omega }^{-\bullet }_{2k}[u,u^{-1}] & \to \mathbb{K}%=\C\LS{\hbar}[u,u^{-1}]
  \\
          a     &\mapsto u^k \sigma_y \bracket{ e^{\hbar \iota_{\Pi_1}/u}a}.
\end{align*}
Here $\sigma_y$ is the symbol map which sets $y^i,dy^i$'s to zero. 
\end{defn}

 $ \int_{BV}$ gives a  map of chain complexes
$$
\int_{BV}: \bracket{
 \bracket{\hat{\Omega }^{-\bullet }_{2k}}^{C(g)}
[u,u^{-1}]
,
\hbar \Delta+ u\d 
} \to (\mathbb{K},0)
$$
with degree $2k$.

\begin{defn}\label{defn-universal-trace}
We define the $g$-twisted universal trace map to be the element
$$
\widehat{\Tr}_g := \int_{BV}\abracket{-}_{int}^g
\in
C_{\mathrm{Lie}}^\bullet \bracket{\mathfrak{g};
\Hom_{\mathbb{C}(\!(\hbar)\!)}\bracket{
CC_{-\bullet}^{per}(\gl_r(\mathcal{W}_{2n}),g)^{C(g)}
,
\mathbb{K}
}}.$$

Given $a_1, \cdots,a_k\in \g$, we will write
$$
\widehat{\Tr}_g[a_1,\cdots, a_k](-)\in \Hom_{\mathbb{C}(\!(\hbar)\!)}\bracket{
CC_{-\bullet}^{per}(\gl_r(\mathcal{W}_{2n}),g)^{C(g)}
,
\mathbb{K}
}
$$
for the corresponding evaluation map on cyclic tensors.

\end{defn}

\begin{thm}\label{TRACE}
The $g$-twisted universal trace lies in the $(\g, \mathfrak{h})$-Lie algebra cochain complex
$$
\widehat{\Tr}_g
\in C_{\mathrm{Lie}}^\bullet \bracket{\mathfrak{g},\mathfrak{h};
\Hom_{\mathbb{C}(\!(\hbar)\!)}\bracket{
CC_{-\bullet}^{per}(\gl_r(\mathcal{W}_{2n}),g)^{C(g)}
,
\mathbb{K}
}}$$
and is closed under the differential
$$
    (\partial_{\mathrm{Lie}}+b_g+uB_g) \widehat{\Tr}_g=0.
$$
\end{thm}
\begin{proof} The equation
$(\partial_{\mathrm{Lie}}+b_g+uB_g) \widehat{\Tr}_g=0$
follows from Theorem \ref{thm-interaction-S1} and the fact that $\int_{BV}$ is a cochain map.  We only need to show that $\widehat{\Tr}_g$ is $\mathfrak{h}$-invariant and vanishes when there is some argument taking value in $\mathfrak{h}$.

Note that for $f\in\mathfrak{h}$, $\d  f$ is linear in $y^i$ so $\partial_{P_1}$ can be only applied once. It will contribute to the expectation value map $\abracket{\cdots, f ,\cdots}_{free}^g$ a factor of
   $$\int^{1}_{0}\LR{u-\frac{1}{2}}du=0.$$
 Thus
$$\widehat{\Tr}_g[\cdots,f,\cdots](-)=\int_{BV} \langle \cdots, f, \cdots\rangle_{free}^g=0.$$

  Finally, $\widehat{\Tr}_g$ is $\mathfrak{h}$-invariant since the operations $\Tr_g,\iota_{{\Pi_1}},\partial_{P_{1}},\partial_{P_{2}},\sigma_y,\sigma_z$ are all $\mathfrak{h}$-invariant.  
\end{proof}

\subsubsection{Lie algebraic characteristic classes}\label{sec:Char}

We first review the Chern-Weil construction of characteristic classes in Lie algebra cohomology.  They will descent to the
usual characteristic forms via the Gelfand-Fuks map.

Let us first recall the construction of curvature forms in Lie algebra cohomology. Let $\g$ be a Lie algebra and $\mathfrak{h}$ be a Lie subalgebra. The curveture form 
$R\in \Hom(\wedge^2\g,\mathfrak{h})$
is defined by 
$$
R(\alpha,\beta):= [\pr(\alpha),\pr(\beta)]_{\mathfrak{h}}-\pr[\alpha,\beta]_{\g}, \quad \alpha, \beta\in \g
$$
with a choice of a $\mathfrak{h}$-equivariant splitting $\text{pr}:\g\to\mathfrak{h}$ of the inclusion $\mathfrak{h} \into\g$ as a vector space. $R$ defines a cohomology class $[R]$ in $H^2(\g,\mathfrak{h})$ which does not depend on the choice of $\pr$. Moreover, given an invariant polynomial $P\in \Sym^m(\mathfrak{h}^\vee)^{\mathfrak{h}}$ on $\mathfrak{h}$ of degree $m$, the cochain 
$$
P(R)\in C^{2m}(\g,\mathfrak{h};\C), \quad P(R):  \wedge^{2m}\g\stackrel{\wedge^m R}{\to} \Sym^{m}(\mathfrak{h})\stackrel{P}{\to} \C
$$
defines a cohomology class $[P(R)]$ in $H^{2m}(\g,\mathfrak{h};\C)$ which again does not depend on the choice of $\pr$. Therefore we have the analogue of Chern-Weil characteristic map
\begin{align*}
\chi: \Sym^\bullet(\mathfrak{h}^\vee)^{\mathfrak{h}}&\to H^{\*}(\g,\mathfrak{h};\C)\\
    P &\mapsto \chi(P):=[P(R)].
\end{align*}

\iffalse
Let us first recall the construction of curvature form in Lie algebra cohomology. Let $\g$ be a Lie algebra and $\mathfrak{h}$ be a Lie subalgebra. Suppose there exists an $\mathfrak{h}$-equivariant splitting
$$
\pr: \g\to \mathfrak{h}
$$
of the embedding $\mathfrak{h}\into \g$. In general $\pr$ is not a Lie algebra homomorphism from $\g$ to $\mathfrak{h}$. The failure defines a linear map
$
R\in \Hom(\wedge^2\g,\mathfrak{h})
$
by
$$
R(\alpha,\beta):= [\pr(\alpha),\pr(\beta)]_{\mathfrak{h}}-\pr[\alpha,\beta]_{\g}, \quad \alpha, \beta\in \g.
$$
$\mathfrak{h}$-equivariance implies that $R(\alpha,-)=0$ if $\alpha\in \mathfrak{h}$ and $R$ is $\mathfrak{h}$ equivariant. $R$ is called the curvature form. Let $\Sym^m(\mathfrak{h}^\vee)^{\mathfrak{h}}$ be invariant polynomials on $\mathfrak{h}$ of degree $m$. Given $P\in \Sym^m(\mathfrak{h}^\vee)^{\mathfrak{h}}$, we can associate a cochain $P(R)$ in $C^{2m}(\g, \mathfrak{h};\C)$ given by the composition
$$
 P(R):  \wedge^{2m}\g\stackrel{\wedge^m R}{\to} \Sym^{m}(\mathfrak{h})\stackrel{P}{\to} \C.
$$
It can be checked that $\pa_{\Lie}P(R)=0$, defining a cohomology class $[P(R)]$ in $H^{2m}(\g,\mathfrak{h};\C)$ which does not depend on the choice of $\pr$. Therefore we have the analogue of Chern-Weil characteristic map
\begin{align*}
\chi: \Sym^\bullet(\mathfrak{h}^\vee)^{\mathfrak{h}}&\to H^{\*}(\g,\mathfrak{h};\C)\\
    P &\mapsto \chi(P):=[P(R)].
\end{align*}
\fi

Now we apply the above construction to our situation where 
$$\mathfrak{g}:=(\mathcal{W}^+_{2n})^g \cdot\text{Id}+\hbar \gl_r((\mathcal{W}^+_{2n})^g), \quad 
\mathfrak{h}=\sp_{2k}\oplus\sp_{2n-2k}^g\oplus \hbar\gl_r\oplus \mathbb{C} \oplus\oplus_{i>1} \hbar^i\mathbb{C},$$
Any element in $\mathfrak{g}$ can be uniquely written as
$$
f\cdot \text{Id}+\hbar A, \quad f\in \mathbb{C}[[y^1,\cdots,y^{2k},z^{2k+1},\cdots,z^{2n}]]^g, \quad A\in \gl_r((\mathcal{W}^+_{2n})^g).
$$

Define a $\mathfrak{h}-$equivariant projection $\pr:\mathfrak{g}\rightarrow \mathfrak{h}$ as follows
$$\pr(f\cdot\text{Id}+\hbar A):=
\bracket{
%\sum_{i,j}
\frac{1}{2}\partial_{y^i}\partial_{y^j}f(0)y^i y^j,
%\sum_{k,l}
\frac{1}{2}\partial_{z^k}\partial_{z^l}f(0)z^k z^l,
\hbar A_1(0),f(0),\bigoplus\limits_{i>1}\frac{1}{r}\trace(\hbar^i A_i(0))}.
 $$
 Here we write
  $\hbar A=\hbar A_1+\hbar^2 A_2+\cdots$ . We can write $\pr=\pr_1+\pr_2+\pr_3+\pr_4$ where
\begin{align*}
  \pr_1(f\cdot\text{Id}+\hbar A)&=\frac{1}{2}\partial_{y^i}\partial_{y^j}f(0)y^i y^j&&\in \mathfrak{sp}_{2k}\\
  \pr_2(f\cdot\text{Id}+\hbar A)&=\frac{1}{2}\partial_{z^k}\partial_{z^l}f(0)z^k z^l&&\in \mathfrak{sp}_{2n-2k}^g\\
  \pr_3(f\cdot\text{Id}+\hbar A)&=\hbar A_1(0)&&\in\hbar\gl_r\\
  \pr_4(f\cdot\text{Id}+\hbar A)&=f(0)+\sum_{i>1}\frac{1}{r}\trace(\hbar^iA_i(0))&&\in \mathbb{C}\oplus\oplus_{i>1} \hbar^i\mathbb{C}
\end{align*}

  The corresponding curvature is given by
 $$R:=[\pr(-),\pr(-)]-\pr([(-),(-)])\in \Hom_{\C}(\wedge^2 \g,\mathfrak{h}).
 $$
It can be decomposed into three terms $R=R_1+R_2+R_3+R_4,$ where
\begin{align*}
R_1&:=[\pr_1(-),\pr_1(-)]-\pr_1[-,-]&&\in \Hom_{\C}(\wedge^2 \g, \sp_{2k})\\
R_2&:=[\pr_2(-),\pr_2(-)]-\pr_2[-,-]&&\in \Hom_{\C}(\wedge^2 \g, \sp_{2n-2k}^g)\\
R_3&:=
-\pr_3[-,-]&&\in \Hom_{\C}(\wedge^2 \g, \gl_r)\\
R_4&:=-\pr_4[-,-]&&\in \Hom_{\C}(\wedge^2 \g, \C\oplus \oplus_{i>1}\hbar^i \C).
\end{align*}
%\begin{rem} It is worthwhile to point out that all the $\Hom$'s here are only $\C$-linear map, but not $\C\PS{\hbar}$-linear, although $\g$ is a $\C\PS{\hbar}$-module. \end{rem}
Explicitly,
\begin{align*}
R_{1}((f+\hbar A),(g+\hbar B))&=-\frac{1}{2} 
%\hat
\omega^{ij}(\partial_{y^i}f(0)\partial_{y^j}\partial_{y^p}\partial_{y^q}g(0)+\partial_{y^i}g(0)\partial_{y^j}\partial_{y^p}\partial_{y^q}f(0))y^py^q\\
R_{2}((f+\hbar A),(g+\hbar B))&=-\frac{1}{2} 
%\hat
\omega^{ij}(\partial_{y^i}f(0)\partial_{y^j}\partial_{z^p}\partial_{z^q}g(0)+\partial_{y^i}g(0)\partial_{y^j}\partial_{z^p}\partial_{z^q}f(0))z^pz^q\\
R_3((f+\hbar A),(g+\hbar B))&=-\hbar 
%\hat
\omega^{ij} (\partial_{y^i}f(0) \partial_{y^j}B_1(0)- \partial_{y^i}A_1(0)\partial_{y^j}g(0))\\
R_4((f+\hbar A),(g+\hbar B))&=-
%\widehat
\omega^{ij}\partial_{y^i}f(0)\partial_{y^j} g(0)+O(\hbar^2).
\end{align*}

Let us define three characteristic classes in the Lie algebra cohomology $H^\bullet(\mathfrak{g},\mathfrak{h};\mathbb{C}) $:
\begin{itemize}
    \item the $\hat A$ genus
$$\hat{A}(\mathfrak{sp}_{2k}):=\bbracket{\det\bracket{\frac{R_1/2}{\sinh(R_1/2)}}^{1/2}}%\in H^\bullet(\mathfrak{g},\mathfrak{h};\mathbb{C}) ,
$$
    \item  the equivariant Chern Character
$$
\Ch_g^\star(\sp_{2n-2k}^g)
=\bbracket{\sum_{m=0}^\infty \frac{1}{m!}\tau_1 \underset{m \, \mathrm{times}}{\underbrace{\bracket{-\frac{R_2}{\hbar}, \cdots,-\frac{R_2}{\hbar}}}}}
=\bbracket{\sum_{m=0}^\infty \frac{1}{m!}\tau_1' \underset{m \, \mathrm{times}}{\underbrace{\bracket{\frac{R_2}{\hbar}, \cdots,\frac{R_2}{\hbar}}}}}%\in H^\bullet(\mathfrak{g},\mathfrak{h};\mathbb{C}),
$$
(See proof of Lemma \ref{lem-g-intertwine} for the definition of $\tau_1,\tau_1'$.)
    \item the equivariant Chern Character
    $$\Ch_g(\gl_r):=[\Tr(g \cdot e^{R_3})]%\in H^\bullet(\mathfrak{g},\mathfrak{h};\mathbb{C}).
$$
%Here $\exp_\star(x)=\sum_{m=0}^\infty \frac{1}{m!}\underset{m \text{ times}}{\underbrace{x\star \cdots \star x}}$. 
\end{itemize}
%The above cohomology classes does not depend on the choice of projection $\pr$.

\subsubsection{Gauss-Manin connection and index}
In this section, we calculate the $\hbar$-variation of periodic cyclic homologies in terms of   Getzler's  Gauss-Manin connection \cite{getzler1993cartan} as in \cite{Localized}.   The   Gauss-Manin connection is shown in \cite{Localized} to have the form 
$$
\nabla_{\hbar\pa_\hbar}^{GM}=\hbar\pa_\hbar+\mathcal L_{\E} + \text{$(b_g+uB_g)$-homotopy},
$$
where $\mathcal L_{\E}$ is the Lie derivative with the vector field 
%$\E=\frac{1}{2}\bracket{\sum_i y^i\frac{\partial}{\partial y^i}}$.
$\E=\frac{1}{2}\bracket{\sum_i y^i\frac{\partial}{\partial y^i}+ \sum_jz^j\frac{\partial}{\partial z^j}}$.
So we define

\begin{defn} Define the $\hbar$-connection $\nabla$ on the $\C\LS{\hbar}$-module $\W_{2n}$ and  $\hat{\Omega }^{-\bullet}_{2k}$ by
$$
  \nabla_{\hbar\pa_\hbar}={\hbar\pa_\hbar}+\mathcal L_{\E}.
$$
\end{defn}

$\nabla_{\hbar\pa_\hbar}$ is invariant under $g$ and induces $\hbar$-connections on   tensors of $\gl_r(\mathcal{W}_{2n})$, $C_{-\bullet}(\gl_r(\mathcal{W}_{2n}),g)$, 
\iffalse\\
$\Hom_{\C\LS{\hbar}} \bracket{C_{-\bullet}(\gl_r(\mathcal{W}_{2n}),g)^{C(g)}
, 
 \bracket{\hat{\Omega }^{-\bullet }_{2k}}^{C(g)}}$, 
$\Hom_{\mathbb{C}(\!(\hbar)\!)}\bracket{
CC_{-\bullet}^{per}(\gl_r(\mathcal{W}_{2n}),g)^{C(g)}
,
 \bracket{\hat{\Omega }^{-\bullet }_{2k}}^{C(g)}
}$, 
\fi
 etc.
All of them will still be denoted by $\nabla_{\hbar\partial_\hbar}$.

\begin{defn} Let $V$ be a $\C\LS{\hbar}$-module where the connection $\nabla_{\hbar\partial_\hbar}$ is defined. Assume $V$ carries a structure of a $\g$-module. We define the linear map
$$
\nabla_{\hbar\pa_\hbar}:  C^\*(\mathfrak{g};V)\to C^\*(\mathfrak{g};V)
$$
by extending that on $V$. 
\iffalse
In other words, given $\varphi \in C^k(\mathfrak{g};V)$,  the cochain $\nabla_{\hbar\pa_\hbar} \varphi\in C^k(\mathfrak{g};V)$ is
$$
(\nabla_{\hbar\pa_\hbar} \varphi) (a_1,\cdots, a_k):= \nabla_{\hbar\pa_\hbar} \bracket{\varphi (a_1,\cdots, a_k)}, \quad a_i \in \g.
$$
\fi
Similarly we define $\nabla_{\hbar\pa_\hbar}$ on $C^\*(\mathfrak{g}, \mathfrak{h};V)$.
\end{defn}

\begin{rem} We emphasize the following point as emphasized in \cite{Localized}. Although $\g$ is  a Lie algebra over $\C\PS{\hbar}$, we do not allow $\nabla_{\hbar\pa_\hbar}$ to act on the $\g$-factor. The reason is that cochains in $C^\bullet(\g; V)$ is only $\C$-linear, but not $\C\PS{\hbar}$-linear. For examples, the projections and curvatures defined in Section \ref{sec:Char} are only $\C$-linear maps.

\end{rem}

In particular,  we have now a well-defined operator
\begin{align*}
\nabla_{\hbar\pa_\hbar}: 
C_{\mathrm{Lie}}^\bullet \bracket{\mathfrak{g},\mathfrak{h};
\Hom_{\mathbb{C}(\!(\hbar)\!)}\bracket{
CC_{-\bullet}^{per}(\gl_r(\mathcal{W}_{2n}),g)^{C(g)}[u,u^{-1}]
,
\mathbb{K}
}}
\to \\
C_{\mathrm{Lie}}^\bullet \bracket{\mathfrak{g},\mathfrak{h};
\Hom_{\mathbb{C}(\!(\hbar)\!)}\bracket{
CC_{-\bullet}^{per}(\gl_r(\mathcal{W}_{2n}),g)^{C(g)}[u,u^{-1}]
,
\mathbb{K}
}}.
\end{align*}

\begin{lem}
The free correlation  map
$$
\langle-\rangle_{free}^g\in \Hom_{\C\LS{\hbar}}
\bracket{
C_{-\bullet}(\gl_r(\mathcal{W}_{2n}),g)^{C(g)}
,
 \bracket{\hat{\Omega }^{-\bullet }_{2k}}^{C(g)}}
$$
is flat with respect to $\nabla_{\hbar\pa_\hbar}$. In other words, for any 
$\hat{\OO}\in 
C_{-\bullet}(\gl_r(\mathcal{W}_{2n}),g)^{C(g)}$,
$$
\nabla_{\hbar\pa_\hbar}\abracket{ \hat{\OO}}_{free}^g =\abracket{ \mathcal \nabla_{\hbar\pa_\hbar} (\hat{\OO})}_{free}^g.
$$
Similarly, $\int_{BV}$ is flat with respect to $\nabla_{\hbar\pa_\hbar}$.
\end{lem}
This lemma follows from a direct check. This also implies that $\abracket{-}_{free}^g$ is homotopic flat with respect to Getzler's Gauss-Manin connection.

Note that $\nabla_{\hbar\pa_\hbar}$ does not commute with $\partial_{\Lie}$. In fact, using $\partial_{\Lie}=\partial+[\hat\Theta,-]$, we find
$$
[\nabla_{\hbar\pa_\hbar}, \pa_{\Lie}]=[\nabla_{\hbar\pa_\hbar}, \pa+{\frac 1 \hbar }[\hat\Theta,-]_\star]=\bbracket{\nabla_{\hbar\pa_\hbar}(\hat\Theta/\hbar),-}_\star.
$$
Here $\hat\Theta$ is viewed as an element in $C^1(\mathfrak{g}, \gl_r((\mathcal{W}_{2n})^g))$ so $\nabla_{\hbar\pa_\hbar}$ can be applied.   To derive the last equality, we have used the fact that the Moyal commutator $[-,-]_\star$   is compatible with $\nabla_{\hbar\pa_\hbar}$. %The above commutator relation will play an important role in deriving the algebra index below.

\begin{prop}\label{con:EACT}
In the cochain complex $C^\bullet (\mathfrak{g},\mathfrak{h};\mathbb{K})$,
$$\nabla_{\hbar\partial_{\hbar}}
(e^{R_4/u\hbar}\widehat{\mathrm{Tr}}_g(1))=
\partial_{\Lie}\text{-exact term}.$$
\end{prop}
\begin{proof}
See the proof of Proposition 3.18 in \cite{Localized}.
\end{proof}

Given an even cochain $A=\sum_{p\ even}A_p,  A_p\in C_{Lie}^p(\g,\mathfrak{h};\mathbb{K})$, we denote $A_u:=\sum_{p}u^{-p/2}A_p$.

\begin{prop}\label{ONELOOP}
In the cochain complex $C^\bullet (\mathfrak{g},\mathfrak{h};\mathbb{K})$,
$$\widehat{\mathrm{Tr}}_g(1)=
u^ke^{-R_4/u\hbar}(\hat{A}(\mathfrak{sp}_{2k})_u\cdot Ch_g^\star(\mathfrak{sp}_{2n-2k})_u\cdot Ch_g(\gl_r)_{u}+O(\hbar)).$$

\end{prop}

See appendix \ref{Feynman-diagram} for a proof.

\iffalse
\begin{proof}
This comes form a Feynman diagram computation. See appendix \ref{Feynman-diagram}. 
\end{proof}
\fi

\begin{prop}\label{index thm}
As a cohomology class in $H^\bullet (\mathfrak{g},\mathfrak{h};\mathbb{K})$,
$$\widehat{\mathrm{Tr}}_g(1)=
u^ke^{-R_4/u\hbar}\cdot\hat{A}(\mathfrak{sp}_{2k})_u\cdot Ch_g^\star(\mathfrak{sp}^g_{2n-2k)})_u\cdot Ch_g(\gl_r)_{u}.$$
\end{prop}
\begin{proof} By Proposition \ref{ONELOOP}, we can expand
$$
e^{R_4/u\hbar}\uTr_g(1)=u^k(\hat{A}(\mathfrak{sp}_{2k})_u\cdot Ch_g^\star(\mathfrak{sp}^g_{2n-2k)})_u\cdot Ch(\gl_r)_{u}+O(\hbar)).
$$
By Proposition \ref{con:EACT}, terms in $O(\hbar)$ have nontrivial weights in $\nabla_{\hbar\pa_\hbar}$ and hence are $\pa_{\Lie}$-exact. The theorem follows by passing to cohomology.
\end{proof}

\begin{rem}
The Lie algebra cohomology class is computed in \cite{Pflaum2007An} using large N method.
\end{rem}

\iffalse
The full algebraic index is
$$
\uTr(1) 
= \frac{1}{|G|}\sum_{g\in G}  \uTr_g(1)
= \sum_{[g]\in\Conj(G)} \frac{1}{|C(g)|} \uTr_g(1)
$$
which lies in the direct sum of $H^\bullet (\mathfrak{g},\mathfrak{h};\mathbb{K})$ for $[g]\in \Conj(G)$. %Here $C(g)=\{h\in G \,|\, gh=hg \}$. 
\fi

\subsection{Descent on orbifolds}
In this section, we explain how to descend Lie algebraic constructions to smooth constructions on orbifolds by using a flat structure, following the presentation in  \cite{Localized}. 

\subsubsection{The idea of descent} 

Let $G$ be a Lie group with Lie algebra $\g$, and $P\to X$ be a $G$-principal bundle on $X$ with a $G$-invariant flat connection $\gamma\in\Omega^1(P,\g)$. $\gamma$ can be seen as a map
\iffalse
\begin{align*}
C^1_{\Lie}(\g) &\to \Omega^1(P) \\
\alpha &\mapsto \alpha(\gamma) 
\end{align*}
\fi
from $g^*=C^1_{\Lie}(\g)$ to $G$-invariant $1$-forms on $\Omega^1(P)$, and  choosing  components along $X$, we get a map $g^*\to\Omega^1(X)$. This map extends naturally to 
\begin{align*}
\desc: (C^\*_{\Lie}(\mathfrak{g}),\partial_{\Lie})&\rightarrow (\Omega^\*(X),d)\\
	\alpha &\mapsto \alpha(\gamma,\cdots,\gamma).
\end{align*}	
The Maurer-Cartan equation 
\begin{equation}\label{Maurer-Cartan}
d\gamma+\frac{1}{2}[\gamma,\gamma]=0\tag{$*$}
\end{equation}
says $\desc$ is a chain map, so descends to cohomologies. 

Let $V$ be a vector space with a $G$-action $G\to \Aut V$. Then the associated bundle $V_P:=P\times_{G}V$ is a vector bundle on $X$ with a flat connection, denoted by $\nabla^{\gamma}$. On the other hand, $V$ carries a $\g$-action $\g\to \End V$, by differentiating the $G$-action. 
Then there is a   chain map
\begin{align*}
\desc: (C^\*_{\Lie}(\mathfrak{g};V),\partial_{\Lie})&\rightarrow (\Omega^\*(X;{V_P}),\nabla^{\gamma})\\
	\alpha &\mapsto \alpha(\gamma,\cdots,\gamma).
\end{align*}	

In section \ref{sec: Fed-conn}, we will see
formal geometry of a symplectic orbifold produces an infinite-dimensional Lie algebra $\g$, which is not the Lie algebra of any Lie group $G$. On the other hand, the geometry of $X$ is encoded in a $K$-principal bundle, where $K$ can be seen as a subgroup of the non-existent group $G$. Such a pair $(\g,K)$ is called a Harish-Chandra pair and can be used to descent. See \cite{Localized} for precise definitions.

Let $(\g, K)$ be a \HC pair. Given a flat $(\mathfrak{g},K)$-bundle $P\rightarrow X$ and a $(\mathfrak{g},K)$-module $V$, we have similarly a chain map
\begin{align*}
\desc: (C^\*_{\Lie}(\mathfrak{g},K;V),\partial_{\Lie})&\rightarrow (\Omega^\*(X;{V_P}),\nabla^{\gamma})\\
	\alpha &\mapsto \alpha(\gamma,\cdots,\gamma)
\end{align*}
and by passing to cohomology, we obtain a descent map (again denoted by $\desc$)
 $$
  \desc: H^\*_{\Lie}(\mathfrak{g},K;V) \rightarrow H^\*(X;{V_P}).
 $$

In section \ref{sec: Fed-conn}, the Maurer-Cartan equation (\ref{Maurer-Cartan}) holds only up to a $Z(\g)$-valued 2-form on $X$. We say $\gamma$ is projective flat. Here $Z(\g)$ is the center of $\g$. Let $H$ be a subgroup of the center of $K$ such that $\Lie(H)=Z(\g)\cap \Lie(K)$, then $\gamma$ induces a flat $(\g/Z(\g), K/H)$ -principal bundle structure on $P/H$.

\subsubsection{Fedosov's connection on symplectic orbifolds} \label{sec: Fed-conn}

Let $X$ be a symplectic orbifold with a symplectic connection $\nabla$, and $E$ be a  rank $r$ orbifold vector bundle  on $X$ with a connection $\nabla^E$. The Weyl bundle
$$
\mathcal{W}_X^+:=F_{\Sp}(X)\times_{\Sp_{2n}}\mathcal{W}^+_{2n}%, \quad \mathcal{W}_X:=F_{\Sp}(X)\times_{\Sp_{2n}}\mathcal{W}_{2n}
.
$$
is a bundle of algebras with a natural connection $\nabla^{\mc W}$.

Given a symplectic connection $\nabla$ and any sequences $\{\omega_k\}_{k\geq 1}$ of closed 2-forms on $X$,  Fedosov showed \cites{fst,connection-orbifold}   the connection $\nabla^{\mc W}\otimes 1 + 1\otimes \nabla^E$ on $\mc W^+_X \otimes \End(E)$ can be modified by an element 
$$\gamma\in \Omega^1_X(  \mc W^+_X \otimes \End(E))$$
to be a flat connection $D$, which is called Fedosov's flat connection
$$
D:=\nabla^{\mc W}\otimes 1 + 1\otimes \nabla^E+\frac{1}{\hbar}[\gamma,-]_{\star},
\qquad D^2={\frac1  \hbar}[\omega_{\hbar},-]_\star=0.
$$
Here $\omega_{\hbar} =-\omega+\sum_{k\geq 1}\hbar^k\omega_k$ is the characteristic class of the deformation quantization, which is an central element in $\Omega^2_X(  \mc W^+_X \otimes \End(E))$.
\iffalse
This is equivalent to solve the equation 
$$
(\nabla^{\mc W}\otimes 1 + 1\otimes \nabla^E)\gamma
+\frac{1}{2\hbar}[\gamma,\gamma]_{\star}
+R_\nabla^{\mc W}+R_{\nabla^E}^{\mc W}
=\omega_{\hbar}. \label{FEDOSOV}
$$
Here $R_\nabla^{\mc W}, R_{\nabla^E}^{\mc W} %\in \Omega^2_X(\mathcal{W}^+\otimes \End(E))
$ are curvature forms of $\mathcal{W}^+$ and $\End(E)$.
\fi

The symbol map  \cite{Fedosov-DQ} gives a natural isomorphism of the space of flat sections of $\mc W^+_X\otimes \End(E)$ with $\Gamma(X,\End(E))[[\hbar]]$, leading to a deformation quantization of the Poisson algebra $\Gamma(X,\End(E))$.

Next, we explain Fedosov's flat connection on orbifolds as a projective flat structure.

Let $i:\mathcal O \into X$ be a $g$-sector in $\wedge X$. The tangent bundle $T_X$ of $X$ pulls back  to an orbifold  vector bundle  $i^*T_X$ on $\mathcal O$, whose   fiber    at a point $x\in \mathcal O$ is described by the triple $(T_xX,\omega_x,g)$.
At any point $x\in\mathcal O$, the set of $g$-invariant symplectic linear isomorphisms $(\R^{2n},\omega,g)\simeq (T_xX,\omega_x,g)$ is a $\Sp_{2n}^g$-torsor. These torsors glue to a principal $\Sp_{2n}^g$-bundle $F_{\Sp^g_{2n}}(\mathcal O)$ on $\mathcal O$, which is  called the $g$-invariant symplectic frame bundle of $i^*T_X$.
%is  a principal $\Sp_{2n}^g$-bundle $F_{\Sp^g_{2n}}(\mathcal O)$ on $\mathcal O$, whose fiber at $x\in \mathcal O$ is the set of $g$-invariant symplectic linear isomorphisms $(\R^{2n},\omega,g)\simeq (T_xX,\omega_x,g)$.

The pull-back connection $i^*\nabla +i^*\nabla^E$ on $i^*T_X\times_{\OO}i^*E$ can be equivalently described by a $g$-invariant connection 1-form 
$$
A\in \Omega^1(F_{\Sp^g_{2n}}(\mathcal O)\times_{\mathcal{O}} i^*Fr(E),\sp_{2n}^g\oplus\gl_{r})
$$
on the  principal $\Sp_{2n}^g\times \GL_r$-bundle $F_{\Sp^g_{2n}}(\mathcal O)\times_{\mathcal{O}} i^*Fr(E)$. $A+i^*\gamma$ defines a projective flat $(\g, \Sp_{2n}^g\times \GL_r)$-bundle structure on $F_{\Sp^g_{2n}}(\mathcal O)\times_{\mathcal{O}} i^*Fr(E)$
\begin{equation}\label{FEDOSOV}
d(A+i^*\gamma)+{\frac1 2}[A+i^*\gamma, A+i^*\gamma]=i^*\omega_\hbar. \tag{F}
\end{equation}
We will slightly abuse the notation to use $\omega_{\hbar}$ to denote its pull-back $i^*\omega_{\hbar}$ to  $\mathcal O$.
 
We can apply the descent construction to the  \HC pair $(\mathfrak{g}/Z(\mathfrak{g}),\Sp_{2n}^g\times \PGL_r)$. Here  $Z(\g)=\C\PS{\hbar}$ and $Z(\g)\cap \sp_{2n}^g=0$.  
There is a natural isomorphism
$$
C^\bullet_{Lie}(\mathfrak{g}, \sp_{2k}\oplus \sp_{2n-2k}^g+ \hbar\gl_r+Z(\mathfrak{g});\mathbb{C}\LS{\hbar})\iso C^\bullet_{Lie}(\mathfrak{g}/Z(\g),\Sp_{2n}^g\times \text{PGL}_r ;\mathbb{C}\LS{\hbar}).
$$
As a corollary, we obtain the Gelfand-Fuks map of cochain complexes by descent
\begin{align*}
\desc:
\bracket{C^\bullet_{Lie}(\mathfrak{g}, \sp_{2k}\oplus \sp_{2n-2k}^g+ \hbar\gl_r+Z(\mathfrak{g});\mathbb{C}\LS{\hbar}),\partial_{Lie}}&\rightarrow \bracket{\Omega^\bullet ({\mathcal O})\LS{\hbar},d}\\
\alpha\quad \quad &\mapsto \alpha(i^*\gamma, \cdots, i^*\gamma).
\end{align*}
Here we do not insert $A$ since it lies in $\sp_{2n}^g+\gl_r$. 
%This is the Gelfand-Fuks map discussed in \cite{feigin2005hochschild,Pflaum2007An,gorokhovsky2017equivariant,feigin1989riemann}.

\subsubsection{Descent of Lie algebraic  characteristic classes}

\begin{prop}\label{prop-char-desc}
Under the descent map $\desc: H^\*(\g,\mathfrak{h};\C)\to H^\*(\mathcal O)\LS{\hbar}$ via the Fedosov connection,
\begin{align*}
  &\desc(\hat{A}(\mathfrak{sp}_{2k}))=\hat{A}(\mathcal O)\in H^\bullet(\mathcal O,\mathbb{C})\\
&\desc(\Ch_g^\star(\sp_{2n-2k}^g))=(\det (1-{g_\perp}^{-1}e^{-R^\perp}))^{-1}\in H^\bullet(\mathcal O,\mathbb{C})  \\
&  \desc(\Ch_g(\gl_r))=e^{\omega_1}\Ch_g(E)\in H^\bullet(\mathcal O,\mathbb{C})\\
&  \desc(R_4)=\omega_{\hbar}-\hbar \omega_1\in H^\*(\mathcal O,\mathbb{C})\PS{\hbar}.
\end{align*}
Here  $\omega_1$ is the $\hbar^1$ term of $\omega_{\hbar}$, $\hat A(\mathcal O)$ is the $\hat A$-genus of   $\mathcal O$, $R^\perp$ is the curvature  of the normal bundle of the embedding $\mathcal O\hookrightarrow X$, and $\Ch(E)$ is the Chern character of the bundle $E$.

\end{prop}
\begin{proof}
We write down the first few terms of the %$g$-invariant 
Fedosov connection $\gamma$ \cite{fedosov-bundle}:  
  $$i^*\gamma=\omega_{ij}y^idx^j+\frac{1}{8}(R_{ijkl}y^iy^jy^k+R_{pqkl}z^pz^qy^k)dx^l+\hbar((R_E)_{ij}y^idx^j+(\omega_1)_{ij}y^idx^j)+\cdots$$
Here we omit the terms with weight bigger than 3.  Those term will not appear in our computation of curvature, the Lie bracket of higher weight term will be projected out by $\pr$.

The descent of $R_1$ is given by
$$
R_1(i^*\gamma,i^*\gamma)=\frac{1}{4}
%\sum_{\substack{i,j,\\k,l\leq 2k}}
R_{ijkl}
y^iy^jdx^k\wedge dx^l.
$$
This is the curvature $2$-form of $\mathcal O$. Applying the invariant $\hat{A}$ polynomial, we get $\hat{A}(\mathcal O)$. 
\iffalse
The descent of $R_2$ is
$$
R_2(i^*\gamma,i^*\gamma)=\frac{1}{4}
%\sum_{\substack{i,j,\\k,l\leq 2k}}
R_{pqkl}z^pz^qdx^k\wedge dx^l.
$$
This is the curvature  of the normal bundle.
\fi

The second equation is proved in \cite{fe:g-index}.

The descent of $R_3$ is
$$
R_2(\gamma,\gamma)=\hbar(R^E_{kl}dx^k\wedge dx^l+\omega_1).
$$
Applying the invariant polynomial $\Tr(g\cdot e^{R_3})$, we get $e^{\omega_1}\Ch_g(E).$

For the last identity we use %the Fedosov equation
\eqref{FEDOSOV} and apply $\pr_4$ to both side
$$
\pr_4(d(A+i^*\gamma)+\frac{1}{2}[A+i^*\gamma,A+i^*\gamma])=\pr_4(\omega_{\hbar}),
$$
we get $\desc(R_4)=-\hbar \omega_1+\omega_{\hbar}$.
\end{proof}

\subsection{The algebraic index theorem for orbifolds}

%Now we discuss how to descent the above index theorem to geometric situation.

Let $W_D$ be the algebra of quantum observables on $X$ ($D$ is the Fedosov connection)
$$
W_D:=\fbracket{s\in \Gamma(X,   \mc W^+_X \otimes \End(E))\,|\,  Ds=0 }.
$$
Restricted to a sector $i:\mathcal O\into X$, we get quantum observables on $\mathcal O$ 
$$
i^*W_D:=\fbracket{s\in \Gamma(\OO,   i^*\mc W^+_X \otimes i^*\End(E))\,|\,  Ds=0 }.
$$

We apply the Gelfand-Fuks descent to the universal trace map:
$$\desc(\widehat{\mathrm{Tr}}_g(-))\in \Omega^\*(\mathcal{O},E^{per}).$$
Here
$$
E^{per}:=((F_{\Sp^g_{2n}}(\mathcal O)\times_{\mathcal{O}} i^*Fr(E))\times_{(\Sp_{2n}^g\times \GL_r)} 
\Hom_{\mathbb{K}}\bracket{
CC_{-\bullet}^{per}(\gl_r(\mathcal{W}_{2n}),g)^{C(g)}
,
\mathbb{K}
}
.$$
$\widehat{\mathrm{Tr}}_g$ is $D+b_g+uB_g$-closed by construction. Restrict to flat sections of $D$, we get a cochain map
$$\desc(\widehat{\mathrm{Tr}}_g):(CC_{-\*}^{per}(i^*W_D),b_g+uB_g)\rightarrow (\Omega^\*(\mathcal{O})\LS{\hbar}[u,u^{-1}],d_{dR})$$
with degree $\dim \mathcal{O}$. 
For a quantum observable $f\in  W_D$, we can restrict it to $\mathcal O$, apply $\desc(\widehat{\mathrm{Tr}}_g)$ to it and then integrate over   $\mathcal{O}$ to get a $g$-twisted trace
$$
  \Tr_g(f)= \int_{\mathcal{O}} \desc(\widehat{\mathrm{Tr}}_g(f|_\mathcal{O}))\in \C\LS{\hbar}.
 $$
The value does not depend on $u$, by the degree reason.

The trace map has the following normalization property
\begin{prop}
 $$ \Tr_g(f)=\frac{(-1)^k}{\hbar^k} \left(\int_{\mathcal O}\mathrm{tr}(f|_\mathcal{O})\frac{\omega^k}{k!}+O(\hbar)\right),\quad \forall f\in \Gamma(X,\End(E)).$$
Here $k=\frac12 \dim \mathcal{O}$.
%Here $k$ is half the dimension of $\mathcal O$. 
\end{prop}
\begin{proof}
  It follows from Feynman diagram computations in appendix \ref{Feynman-diagram}. The $\hbar$-leading term of $\Tr_g(f)$ comes from the tree diagrams which only involves the vertex $d_{2k}(\omega_{ij}y^i dx^j)$, and this gives
  $$
  \int_{\mathcal O} \int_{BV}\text{tr}(f|_\mathcal{O})e^{d_{2k}( \omega_{ij}y^i dx^j)/\hbar}=\frac{(-1)^k}{\hbar^k}\int_{\mathcal O}\text{tr}(f|_\mathcal{O})\frac{\omega^k}{k!}.
  $$
\end{proof}

The index for the quantum algebra is given by
$$
\mathrm{Tr}(1):=\int_{\wedge {X}}\mathrm{desc}(\widehat{\mathrm{Tr}}_g(1)).
$$
%To define the right hand side, we use a partition of unity with respect to an orbifold atlas $\{(\tilde U_i,G_i,U_i,\pi_i)\}_{i\in I}$ of $X$, and for each chart $\{(\tilde U_i,G_i,U_i,\pi_i)\}$
%is defined locally by an orbifold atlas $U$ of $X$ and integration on $\wedge U$ and use a partition of unity, but does not depend on these choices. 
Then Theorem \ref{index thm} and Proposition \ref{prop-char-desc} imply the following algebraic index theorem (see \cites{Pflaum2007An,fst})
\begin{thm}
Let $X$ be a compact smooth symplectic orbifold of dimension $2n$, and $E$ be a complex orbifold vector bundle over $X$. Let $W_D$ be the quantum algebra associated to the deformation quantization class $\omega_{\hbar}$. %Let $\Tr$ be the  trace map on $W_D$  obtained from the universal trace $\widehat{\mathrm{Tr}}_g$ by descent construction. 
Then the algebraic index is 
$$ \Tr(1)=\int_{\wedge X}e^{-\omega_{\hbar}/\hbar}\frac{\hat{A}(\wedge X)\Ch_g(E)}{m \det (1-{g_\perp}^{-1}e^{-R^\perp})}.$$
Here $m$ is the locally constant function on $\wedge X$ which coincides for every sector $\mathcal O$ with the multiplicity of $\mathcal O$.
\end{thm}

\appendix
\section{A brief review of orbifolds}
In this section we introduce some notions of orbifolds and refer to \cite{fe:refof} for more details. %\Yang can also use groupoids/stacks to describe orbifolds

\begin{defn}
Let $X$ be a topological space. An orbifold chart on $X$ is a $4$-tuple $(\tilde U,G,U,\pi)$, where 
\begin{itemize}
    \item $U$ is an open subset of $X$,
    \item $\tilde U$ is a   contractible open subset of $\R^n$, 
    \item $G$ is a finite group of diffeomorphisms of $\tilde U$,
    \item $\pi:\tilde U\to U$ is a map which can be  factored as $\pi=\bar \pi\circ p$, where $p:\tilde U\to \tilde U/G$ is the quotient map %orbit map
    , and $\bar \pi:\tilde U/G\to U$ is a homeomorphism. 
\end{itemize}
\end{defn}

Two orbifold charts $(\tilde U_i,G_i,U_i,\pi_i)$, $(\tilde U_j,G_j,U_j,\pi_j)$
are called {compatible},  if for any points $\tilde x_i\in \tilde U_i$, $\tilde x_j\in \tilde U_j$ with $\pi_i(\tilde x_i)=\pi_j(\tilde x_j)$, there is a chart $(\tilde U_k,G_k,U_k,\pi_k)$ with $U_k\subset U_i\cap U_j$, open embeddings $\tilde U_k\into \tilde U_i$, $\tilde U_k\into \tilde U_j$, and group embeddings $G_k\into G_i$, $G_k\into G_j$, such that the embeddings of $\tilde U_k$ into $\tilde U_i,\tilde U_j$ are equivariant with the group embeddings and compatible with the embeddings $U_k$ into $U_i, U_j$.

\begin{defn}
An orbifold is a paracompact Hausdorff space $X$ with an equivalence class of  orbifold atlases. Here an orbifold atlas $\{(\tilde U_i,G_i,U_i,\pi_i)\}_{i\in I}$ is a collection  of compatible orbifold charts, such that $\{U_i\}_{i\in I}$ is  a cover of  $X$. 
\end{defn}

%Differential geometrical constructions on  orbifolds  can be defined to be equivariant objects on local charts which are  compatible with change of charts. 
%For example, we have the notion of orbifold vector bundles and principal bundles. 

\begin{defn}
An orbifold vector (principal) bundle $F$ on $X$ is an object which assigns to an orbifold chart $(\tilde U_i,G_i,U_i,\pi_i)$  a $G$-equivariant vector (principal) bundle $\tilde F_i$ on $\tilde U_i$. These local equivariant bundles should be compatible with respect to change of orbifold charts. %  explicitly write the local group action

%A section of an orbifold  bundle $F\to X$ is a map $X\to E$ such that the composition $X\to F\to X$ equals the identity map on $X$. Writing in charts, it is equivalently a $G$-equivariant section $\tilde U\to \tilde F_i$ compatible with respect to change of  charts.
\end{defn}

Note that an orbifold bundle is not a  fiber bundle on the base space. On an orbifold bundle, we have the notion of sections and connections. 
%An orbifold has its tangent and cotangent bundle. 
Smooth functions, vector fields and differential forms are defined as sections of corresponding orbifold vector bundles. %Connections  on these bundles are similarly defined.

\begin{defn}
A symplectic orbifold is an orbifold $X$ with a non-degenerate closed $2$-form.
\end{defn}
The Linearization Theorem says that locally we can find Darboux coordinates  such that the orbifold group $G$ acts linearly. In other words, a symplectic orbifold is locally modeled by $(V/G,\omega)$, where $(V,\omega)$ is a symplectic vector space, and $G$ is a finite group of symplectic linear transformations. %See \cite{fe:g-density}.

%we can find a chart $(\tilde U,G,U,\pi)$ in which $G$ acts linearly and the symplectic form has constant coefficient, with the center of $\tilde U\subset \R^{2n}$ being preimage of any point in the orbifold. So we still have Darboux coordinates, as on a symplectic manifold.

\begin{defn}
  A \textit{symplectic connection} on an orbifold $X$ is a torsion free connection $\nabla$ on the tangent bundle $T_X$ that is compatible with the symplectic form: $\nabla \omega=0$.
\end{defn}
Symplectic connections exist on any symplectic orbifold and are not unique.% \cite{Fedosov:1996fu}. 

Singularities of an orbifold have the form of fixed points of  finite group actions. The inertia orbifold $\wedge X$ of  an  orbifold $X$ collects all kinds of fixed point subsets. 
Let we explain this in an orbifold chart $(\tilde U,G,U,\pi)$. 
We can assume $G$ acts on $\tilde U$ linearly by the linearization theorem. %, and $G$ acts effectively. 
For a conjugate class $[g]\in\Conj(G)$, the fixed point set of $g$ in $\tilde U$ is $\tilde U^{g}:=\{x\in \tilde U \,|\, gx=x \}$. When projected to $U$, it becomes $\pi(\tilde U^{g})=\tilde U^{g}/\!\!/C(g)$, where $C(g)=\{h\in G \,|\, gh=hg\}$. $\tilde U^{g}/\!\!/C(g)$ is an orbifold with a chart $(\tilde U^{g},C(g),\tilde U^{g}/\!\!/C(g),\pi|_{\tilde U^{g}})$. Thus, the disjoint union of all   kinds of singularities is an orbifold
$$
\wedge U=\coprod_{[g]\in\Conj (G)} \tilde U^{g}/\!\!/C(g)
%=U \coprod \text{low dimensional sectors}
.$$
\begin{defn}
The inertia orbifold $\wedge X$ is the orbifold with underlying set $\{(x,[g]) \,|\, x\in X, [g] \in \Conj(G_x)  \}$ and orbifold charts $(\tilde U^{g},C(g),\tilde U^{g}/\!\!/C(g),\pi|_{\tilde U^{g}})$ described above. 
\end{defn}
The orbifold structure of $\wedge X$ does not rely on the choice of orbifold charts of $X$.

Let $\mathcal O$ be a sector   of $\wedge X$, i.e., a minimal connected component of $\wedge X$. By construction, there is a natural inclusion $\mathcal O \hookrightarrow X$.
  For any bundle $E$ on $\mathcal O$,   pointwise conjugate classes $[g]$    glue to a smooth section  $g\in \Gamma(\mathcal O,\Aut(E))$, which is called a cyclic structure in \cite{crainic}.

Integration on orbifolds are defined by using a partition of unity and integration in local charts.  For an orbifold chart $(\tilde U,G,U,\pi)$, let $m$ denote  the order of the subgroup of $G$ of elements  which acts trivially on $\tilde U$, which is called the multiplicity of the chart. An orbifold chart of $\wedge X$ may have a multiplicity $m>1$, since the action of $C(g)$ may not be effective. 
We require that  compatible charts have the same multiplicity. Then on an orbifold $X$ the multiplicity  $m$ is a locally constant function $m:X\to \Z_+$.  
%By definition $m$ is independent of the choice of atlases, hence is a locally constant function on $X$, called the multiplicity of $X$. 
\begin{defn}
Let $\alpha$ be  a compactly supported differential form on an orbifold $X$. Take a partition of unity $\{\rho_i(x)\}_{i\in I}$ such that $\supp(\rho_i)\subset U_i$.  Lift $\rho_i(x)$ and $\alpha$ to the  chart $(\tilde U_i,G_i,U_i,\pi_i)$, we get a $G_i$-invariant function $\rho_i(\tilde x)$ and a top form $\tilde\alpha_i$ on $\tilde U_i$. Denote the multiplicity of the chart by by $m_i$. Define
$$ 
\int_X \alpha:=\sum_{i\in I} \frac{m_i}{|  G_i|} \int_{\tilde U_i} \rho_i(\tilde x) \tilde\alpha_i.
$$
\end{defn}
%The integration is independent of the choice of the partition of unity.

\section{Feynman diagrams and  proof of Theorem  \ref{ONELOOP}}\label{Feynman-diagram}

By the standard technique of Feynman diagrams (see e.g. \cite{bessis1980quantum}),
$
\widehat{\mathrm{Tr}}_g(1)
$
can be expressed as 
$$
\widehat{\mathrm{Tr}}_g(1)=u^k \bracket{\sum_{\Gamma} \frac{\hbar^{l(\Gamma)-k(\Gamma)}W_{\Gamma}}{|\text{Aut}(\Gamma)|}}.
$$
Here a diagram $\Gamma$ has loop number  $l(\Gamma)$  and $k(\Gamma)$ connected components.   
The sum of $\Gamma$ is over all diagrams without external edges, possibly disconnected or with self-loops. 

Recall the $\mathfrak{h}-$equivariant projection $\pr$ in Section \ref{sec:Char}. Let
$$\widehat{\gamma}:=\widehat\Theta-\pr\in C^1_{\Lie}(\mathfrak{g},\mathfrak{h};\mathfrak{g}).$$
By Theorem \ref{TRACE},  we can replace $\widehat{\Theta}$ by $\widehat{\gamma}$ in the expression of $\widehat{\Tr}_g$ since $\pr$ lies in $\mathfrak{h}$. 
%We always assume this replacement in the following discussions. 
As in Definition \ref{defn:interaction}, when we insert $\widehat \gamma$ to define $\uTr_g$, it is viewed as a Lie algebra 1-cochain valued in $\gl_r((\W_{2n})^g)$, i.e. an element of $C^1(\g,\mathfrak{h}; \gl_r((\W_{2n})^g)$, using the natural embedding $\g\subset \gl_r((\W_{2n})^g)$.

Thus, vertices of a diagram are given by $\hat \gamma\in C^1(\g,\mathfrak{h}; \gl_r((\W_{2n}^g)))$ which we decompose as  
$$
\hat \gamma=\sum_{i\geq 0}\hat \gamma_i \hbar^i, \quad  \hat\gamma_i=\sum_{m\geq 0} \hat\gamma_i^{(m)}.
$$
$\hat \gamma_i^{(m)}$ collects those terms in $\hat \gamma$ whose value in $\gl_r(\W_{2n})$ is homogeneous of degree $m$ in $y$'s and degree $i$ in $\hbar$. Then in the above Feynman diagram expansion, $d_{2k}(\widehat \gamma_i^{(m)})$ contributes a $m$-valency vertex of loop number $i$ with coefficient valued in $C^1(\g, \mathfrak{h}; \gl_r(\C))$. 

There are three types of edges
\begin{itemize}
    \item A red edge connects a $dy_i$ with a $dy_j$
    \item A blue edge connects a $y_i$ with a $y_j$
    \item A green edge connects  a $z_i$ with a $z_j$. It may appear as  a loop.
\end{itemize}
We associate propagators
$$
\iota_{\Pi_1/u}, \qquad \partial_{\Pi_1}, \qquad \partial_{P_3}
$$
with them, see figure \ref{edge}.  Any vertex $d_{2k}(\widehat \gamma_i^{(m)})$ connects with exactly 1 red edge.

\begin{center}
\begin{figure}
\tikzset{every picture/.style={line width=0.75pt}} %set default line width to 0.75pt        

\begin{tikzpicture}[x=0.75pt,y=0.75pt,yscale=-1,xscale=1]

%Straight Lines [id:da3476933884550959] 
\draw [color={rgb, 255:red, 255; green, 0; blue, 0 }  ,draw opacity=1 ]   (91.75,170.5) -- (161.75,170.5) ;
%Shape: Circle [id:dp8751353732118314] 
\draw  [fill={rgb, 255:red, 0; green, 0; blue, 0 }  ,fill opacity=1 ] (89.5,170.5) .. controls (89.5,169.26) and (90.51,168.25) .. (91.75,168.25) .. controls (92.99,168.25) and (94,169.26) .. (94,170.5) .. controls (94,171.74) and (92.99,172.75) .. (91.75,172.75) .. controls (90.51,172.75) and (89.5,171.74) .. (89.5,170.5) -- cycle ;
%Shape: Circle [id:dp9760748476745267] 
\draw  [fill={rgb, 255:red, 0; green, 0; blue, 0 }  ,fill opacity=1 ] (159.5,170.5) .. controls (159.5,169.26) and (160.51,168.25) .. (161.75,168.25) .. controls (162.99,168.25) and (164,169.26) .. (164,170.5) .. controls (164,171.74) and (162.99,172.75) .. (161.75,172.75) .. controls (160.51,172.75) and (159.5,171.74) .. (159.5,170.5) -- cycle ;
%Straight Lines [id:da5721694043568092] 
\draw [color={rgb, 255:red, 0; green, 0; blue, 255 }  ,draw opacity=1 ]   (231.75,170.5) -- (301.75,170.5) ;
%Shape: Circle [id:dp33414199961087243] 
\draw  [fill={rgb, 255:red, 0; green, 0; blue, 0 }  ,fill opacity=1 ] (229.5,170.5) .. controls (229.5,169.26) and (230.51,168.25) .. (231.75,168.25) .. controls (232.99,168.25) and (234,169.26) .. (234,170.5) .. controls (234,171.74) and (232.99,172.75) .. (231.75,172.75) .. controls (230.51,172.75) and (229.5,171.74) .. (229.5,170.5) -- cycle ;
%Shape: Circle [id:dp6043797573830922] 
\draw  [fill={rgb, 255:red, 0; green, 0; blue, 0 }  ,fill opacity=1 ] (299.5,170.5) .. controls (299.5,169.26) and (300.51,168.25) .. (301.75,168.25) .. controls (302.99,168.25) and (304,169.26) .. (304,170.5) .. controls (304,171.74) and (302.99,172.75) .. (301.75,172.75) .. controls (300.51,172.75) and (299.5,171.74) .. (299.5,170.5) -- cycle ;
%Straight Lines [id:da026506320466228472] 
\draw [color={rgb, 255:red, 0; green, 255; blue, 0 }  ,draw opacity=1 ]   (371.75,170.5) -- (441.75,170.5) ;
%Shape: Circle [id:dp49787130182477535] 
\draw  [fill={rgb, 255:red, 0; green, 0; blue, 0 }  ,fill opacity=1 ] (369.5,170.5) .. controls (369.5,169.26) and (370.51,168.25) .. (371.75,168.25) .. controls (372.99,168.25) and (374,169.26) .. (374,170.5) .. controls (374,171.74) and (372.99,172.75) .. (371.75,172.75) .. controls (370.51,172.75) and (369.5,171.74) .. (369.5,170.5) -- cycle ;
%Shape: Circle [id:dp3373768553294072] 
\draw  [fill={rgb, 255:red, 0; green, 0; blue, 0 }  ,fill opacity=1 ] (439.5,170.5) .. controls (439.5,169.26) and (440.51,168.25) .. (441.75,168.25) .. controls (442.99,168.25) and (444,169.26) .. (444,170.5) .. controls (444,171.74) and (442.99,172.75) .. (441.75,172.75) .. controls (440.51,172.75) and (439.5,171.74) .. (439.5,170.5) -- cycle ;
%Curve Lines [id:da7848589980888199] 
\draw [color={rgb, 255:red, 0; green, 255; blue, 0 }  ,draw opacity=1 ]   (511.75,170.5) .. controls (490.25,148.5) and (494.81,118.5) .. (511.28,118.56) .. controls (527.75,118.62) and (532.75,146.5) .. (511.75,170.5) -- cycle ;
%Shape: Circle [id:dp5249697596090368] 
\draw  [fill={rgb, 255:red, 0; green, 0; blue, 0 }  ,fill opacity=1 ] (509.5,170.5) .. controls (509.5,169.26) and (510.51,168.25) .. (511.75,168.25) .. controls (512.99,168.25) and (514,169.26) .. (514,170.5) .. controls (514,171.74) and (512.99,172.75) .. (511.75,172.75) .. controls (510.51,172.75) and (509.5,171.74) .. (509.5,170.5) -- cycle ;

% Text Node
\draw (80.51,179) node [anchor=north west][inner sep=0.75pt]   [align=left] {$dy^i$};
% Text Node
\draw (151.21,179) node [anchor=north west][inner sep=0.75pt]   [align=left] {$dy^j$};
% Text Node
\draw (226.11,179) node [anchor=north west][inner sep=0.75pt]   [align=left] {$y^i$};
% Text Node
\draw (294.51,179) node [anchor=north west][inner sep=0.75pt]   [align=left] {$y^j$};
% Text Node
\draw (365.71,179) node [anchor=north west][inner sep=0.75pt]   [align=left] {$z^i$};
% Text Node
\draw (436.11,179) node [anchor=north west][inner sep=0.75pt]   [align=left] {$z^j$};
% Text Node
\draw (495.31,179) node [anchor=north west][inner sep=0.75pt]   [align=left] {$z^i$};
% Text Node
\draw (515,179) node [anchor=north west][inner sep=0.75pt]   [align=left] {$z^j$};
% Text Node
\draw (110,140) node [anchor=north west][inner sep=0.75pt]   [align=left] {$\iota_{\Pi_1/u}$};
% Text Node
\draw (256,140) node [anchor=north west][inner sep=0.75pt]   [align=left] {$\partial_{\Pi_1}$};
% Text Node
\draw (396,140) node [anchor=north west][inner sep=0.75pt]   [align=left] {$\partial_{P_3}$};
% Text Node
\draw (502,97) node [anchor=north west][inner sep=0.75pt]   [align=left] {$\partial_{P_3}$};

\end{tikzpicture}
\caption{Three types of edges.}\label{edge}
\end{figure}
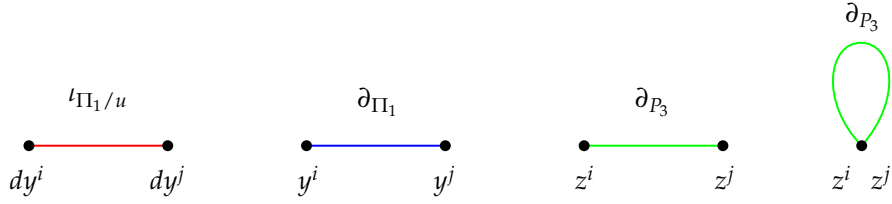
\end{center}

Note that our vertices are matrix valued, therefore it is important to keep their orders. This is a subtle difference with the usual Feynman diagram technique. For each diagram $\Gamma$,  let $V(\Gamma)$ be the set of ordered vertices of $\Gamma$. We use a bijection
 $$
 \chi:\{1,\cdots ,|V(\Gamma)|\}\rightarrow V(\Gamma),
 $$
 to label the order. The Feynman integral $W_{\Gamma}\in C^{|V(\Gamma)|}(\g, \mathfrak{h};\C)$ is defined as follows. Given $\xi_i\in \g/\mathfrak{h}$,

 \begin{align*}
  W_{\Gamma}(\xi_1\wedge\cdots\wedge\xi_{m})
  :=\sum_{\varepsilon\in S_{m}\ \text{and}\ \chi}\text{sign}(\varepsilon)W_{\Gamma^{\chi}}( \xi_{\varepsilon(1)}\otimes\cdots\otimes \xi_{\varepsilon(m)}), \quad m=|V(\Gamma)|.
\end{align*}
Here $\Gamma^{\chi}$ means the graph $\Gamma$ equipped with an ordering map $\chi$, and
$$
W_{\Gamma^{\chi}}(\xi_1\otimes\cdots\otimes \xi_{m}):=$$
$$
\sigma\bracket{\tr\int_{S^1_{cyc}[m+1]}d\theta_0\cdots d\theta_{m}\bracket{\prod_{e\in E(\Gamma)}P_e}(1\otimes d_{2k}{\widehat{\gamma}_{\upsilon(\chi(1))}}(\xi_{1})\otimes\cdots\otimes d_{2k}{\widehat{\gamma}_{\upsilon(\chi(m))}}(\xi_{m}))}.
$$
Here $P_e$ means applying the propagator of the edge $e$ to the two vertices indexed by two ends of $e$. 

\iffalse
This gives the precise meaning of the Feynman diagram expansion
$$
\uTr_g(1)=u^k \bracket{\sum_{\substack{\Gamma\ \\}} {\frac{\hbar^{g(\Gamma)-k(\Gamma)}W_{\Gamma}} {|\text{Aut}(\Gamma)|}}}.
$$
Note that by the symbol map $\sigma$ in the last step, all diagrams  do not contain external edges. In other words, all edges of vertices should be contracted by propagators.
\fi

The proposition follows by computing the tree diagrams ($g=0$) and 1-loop diagrams ($g=1$).

We first collect some details on the vertices and the curvatures. Recall the curvature% projection map
\begin{align*}
R_1&:=[\pr_1(-),\pr_1(-)]-\pr_1[-,-]&&\in \Hom_{\C}(\wedge^2 \g, \sp_{2k})\\
R_2&:=[\pr_2(-),\pr_2(-)]-\pr_2[-,-]&&\in \Hom_{\C}(\wedge^2 \g, \sp_{2n-2k}^g)\\
R_3&:=
%[\pr_3(-),\pr_3(-)]
-\pr_3[-,-]&&\in \Hom_{\C}(\wedge^2 \g, \gl_r)\\  
R_4&:=-\pr_4[-,-]&&\in \Hom_{\C}(\wedge^2 \g, \C\oplus \oplus_{i>1}\hbar^i \C).
\end{align*}
%and the curvature
\begin{align*}
R_{1}((f+\hbar A),(g+\hbar B))&=-\frac{1}{2} 
%\hat
\omega^{ij}(\partial_{y^i}f(0)\partial_{y^j}\partial_{y^p}\partial_{y^q}g(0)+\partial_{y^i}g(0)\partial_{y^j}\partial_{y^p}\partial_{y^q}f(0))y^py^q\\
R_{2}((f+\hbar A),(g+\hbar B))&=-\frac{1}{2} 
%\hat
\omega^{ij}(\partial_{y^i}f(0)\partial_{y^j}\partial_{z^p}\partial_{z^q}g(0)+\partial_{y^i}g(0)\partial_{y^j}\partial_{z^p}\partial_{z^q}f(0))z^pz^q\\
R_3((f+\hbar A),(g+\hbar B))&=-\hbar 
%\hat
\omega^{ij} (\partial_{y^i}f(0) \partial_{y^j}B_1(0)- \partial_{y^i}A_1(0)\partial_{y^j}g(0))\\
R_4((f+\hbar A),(g+\hbar B))&=-
%\widehat
\omega^{ij}\partial_{y^i}f(0)\partial_{y^j} g(0)+O(\hbar^2).
\end{align*}

Let us also write down  a few leading terms in the expansion $d_{2k}\widehat{\gamma}=\sum_{i\geq 0}d_{2k}\widehat{\gamma}_i$:
\begin{align*}
d_{2k}\widehat\gamma_0((f+\hbar A))&
%=d_{2k}\widehat{\gamma}(f)
=\partial_{y^i} f(0)dy^i+\frac{1}{2} \partial_{y^i}\partial_{y^p}\partial_{y^q} f(0)y^p y^q dy^i+\frac{1}{2} \partial_{y^i}\partial_{z^p}\partial_{z^q} f(0)z^p z^q dy^i+\mathrm{higher \,\,terms}%O(y^3dy)
\\
d_{2k}\widehat\gamma_1((f+\hbar A))&
%=\hbar d_{2k}A_1
=\hbar\partial_{y^i} A_1(0)dy^i+\hbar O(ydy) \\
d_{2k}\widehat{\gamma}_2(f+\hbar A)&=\hbar^2 d_{2k}A_2\\
&\cdots\\
d_{2k}\widehat{\gamma}_l(f+\hbar A)&=\hbar^l d_{2k}A_l,\quad l>0.
\end{align*}

%Now we are ready to compute tree and 1-loop diagrams. %See figure \ref{graph}.

\begin{itemize}
\item Let $\Gamma_1$ be the graph with  two vertices labeled by $d_{2k}(\widehat{\gamma}^{(1)}_0)$ connected by a red propagator.  It is the only connected tree  without external edges. This term contributes $\widehat{\omega}_0/u$. Here $\widehat{\omega}_0 \in C^2(\g, \mathfrak{h};\C)$ is the 2-cocycle
$$
  \widehat{\omega}_0((f+\hbar A),(g+\hbar B))= \omega^{ij}\partial_{y^i}f(0)\partial_{y^j} g(0).
$$
In other words, $\hat \omega_0= -R_4|_{\hbar=0}$. % which descends to the symplectic 2-form under the Gelfand-Fuks map.
  See figure \ref{graph1}.
\begin{center}
\begin{figure}
\tikzset{every picture/.style={line width=0.75pt}} %set default line width to 0.75pt        

\begin{tikzpicture}[x=0.75pt,y=0.75pt,yscale=-1,xscale=1]
%uncomment if require: \path (0,824); %set diagram left start at 0, and has height of 824

%Straight Lines [id:da6821782992632561] 
\draw [color={rgb, 255:red, 255; green, 0; blue, 0 }  ,draw opacity=1 ]   (146.05,39) -- (216.05,39) ;
%Shape: Circle [id:dp27563855006913485] 
\draw  [fill={rgb, 255:red, 0; green, 0; blue, 0 }  ,fill opacity=1 ] (143.8,39) .. controls (143.8,37.76) and (144.81,36.75) .. (146.05,36.75) .. controls (147.29,36.75) and (148.3,37.76) .. (148.3,39) .. controls (148.3,40.24) and (147.29,41.25) .. (146.05,41.25) .. controls (144.81,41.25) and (143.8,40.24) .. (143.8,39) -- cycle ;
%Shape: Circle [id:dp8071724437868996] 
\draw  [fill={rgb, 255:red, 0; green, 0; blue, 0 }  ,fill opacity=1 ] (213.8,39) .. controls (213.8,37.76) and (214.81,36.75) .. (216.05,36.75) .. controls (217.29,36.75) and (218.3,37.76) .. (218.3,39) .. controls (218.3,40.24) and (217.29,41.25) .. (216.05,41.25) .. controls (214.81,41.25) and (213.8,40.24) .. (213.8,39) -- cycle ;

% G1
% Text Node
\draw (64,31) node [anchor=north west][inner sep=0.75pt]   [align=left] {$\Gamma_1$};
% Text Node
\draw (120,47.5) node [anchor=north west][inner sep=0.75pt]   [align=left] {$d_{2k}(\widehat{\gamma}^{(1)}_0)$};
% Text Node
\draw (190,47.5) node [anchor=north west][inner sep=0.75pt]   [align=left] {$d_{2k}(\widehat{\gamma}^{(1)}_0)$};

\end{tikzpicture}
\caption{The contribution of $\Gamma_1$ is $\widehat{\omega}_0/u$.}\label{graph1}
\end{figure}
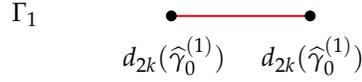
\end{center}

\item Let $\Gamma_2$ be the graph with two vertices labeled by $d_{2k}(\widehat{\gamma}^{(1)}_0)$ and $d_{2k}(\widehat{\gamma}^{(1)}_1)$, connected by a red propagator.   This diagram gives $R_3/u$. See figure \ref{graph2}.
\begin{center}
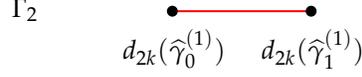
\begin{figure}
\tikzset{every picture/.style={line width=0.75pt}} %set default line width to 0.75pt        

\begin{tikzpicture}[x=0.75pt,y=0.75pt,yscale=-1,xscale=1]
%uncomment if require: \path (0,824); %set diagram left start at 0, and has height of 824

%Straight Lines [id:da5036839911156225] 
\draw [color={rgb, 255:red, 255; green, 0; blue, 0 }  ,draw opacity=1 ]   (146.55,109) -- (216.55,109) ;
%Shape: Circle [id:dp48986134906117385] 
\draw  [fill={rgb, 255:red, 0; green, 0; blue, 0 }  ,fill opacity=1 ] (144.3,109) .. controls (144.3,107.76) and (145.31,106.75) .. (146.55,106.75) .. controls (147.79,106.75) and (148.8,107.76) .. (148.8,109) .. controls (148.8,110.24) and (147.79,111.25) .. (146.55,111.25) .. controls (145.31,111.25) and (144.3,110.24) .. (144.3,109) -- cycle ;
%Shape: Circle [id:dp4167473144706567] 
\draw  [fill={rgb, 255:red, 0; green, 0; blue, 0 }  ,fill opacity=1 ] (214.3,109) .. controls (214.3,107.76) and (215.31,106.75) .. (216.55,106.75) .. controls (217.79,106.75) and (218.8,107.76) .. (218.8,109) .. controls (218.8,110.24) and (217.79,111.25) .. (216.55,111.25) .. controls (215.31,111.25) and (214.3,110.24) .. (214.3,109) -- cycle ;

% Text Node
\draw (64,101.07) node [anchor=north west][inner sep=0.75pt]   [align=left] {$\Gamma_2$};
% Text Node
\draw (120,116.5) node [anchor=north west][inner sep=0.75pt]   [align=left] {$d_{2k}(\widehat{\gamma}^{(1)}_0)$};
% Text Node
\draw (190,116.5) node [anchor=north west][inner sep=0.75pt]   [align=left] {$d_{2k}(\widehat{\gamma}^{(1)}_1)$};

\end{tikzpicture}
\caption{The diagram $\Gamma_2$ gives $R_3/u$.}\label{graph2}
\end{figure}
\end{center}

Let $\mathfrak{G}_2$ be the set of graphs such that each element is the disjoint of some $\Gamma_2$'s. The contribution of $\mathfrak{G}_2$  is $\tr(g\cdot e^{R_3/u})$. See figure \ref{graph2+}.
\begin{center}
\begin{figure}

\tikzset{every picture/.style={line width=0.75pt}} %set default line width to 0.75pt        

\begin{tikzpicture}[x=0.75pt,y=0.75pt,yscale=-1,xscale=1]
%uncomment if require: \path (0,824); %set diagram left start at 0, and has height of 824

%Straight Lines [id:da7074560553657564] 
\draw [color={rgb, 255:red, 255; green, 0; blue, 0 }  ,draw opacity=1 ]   (184.19,70.05) -- (209.63,76.94) ;
%Shape: Ellipse [id:dp9034987590188189] 
\draw  [fill={rgb, 255:red, 0; green, 0; blue, 0 }  ,fill opacity=1 ] (182.77,67.61) .. controls (184.11,66.82) and (185.82,67.28) .. (186.6,68.62) .. controls (187.39,69.97) and (186.94,71.7) .. (185.61,72.49) .. controls (184.27,73.28) and (182.55,72.83) .. (181.77,71.48) .. controls (180.99,70.14) and (181.44,68.4) .. (182.77,67.61) -- cycle ;
%Shape: Ellipse [id:dp73607591233332] 
\draw  [fill={rgb, 255:red, 0; green, 0; blue, 0 }  ,fill opacity=1 ] (157.33,74.5) .. controls (158.66,73.71) and (160.38,74.16) .. (161.16,75.5) .. controls (161.94,76.85) and (161.5,78.59) .. (160.16,79.38) .. controls (158.83,80.17) and (157.11,79.71) .. (156.33,78.37) .. controls (155.55,77.02) and (155.99,75.29) .. (157.33,74.5) -- cycle ;
%Shape: Ellipse [id:dp8807683876211211] 
\draw  [color={rgb, 255:red, 0; green, 0; blue, 0 }  ,draw opacity=0.11 ] (133.3,121.42) .. controls (133.3,93.21) and (155.96,70.33) .. (183.91,70.33) .. controls (211.86,70.33) and (234.52,93.21) .. (234.52,121.42) .. controls (234.52,149.64) and (211.86,172.52) .. (183.91,172.52) .. controls (155.96,172.52) and (133.3,149.64) .. (133.3,121.42) -- cycle ;
%Shape: Ellipse [id:dp3642686770912036] 
\draw  [fill={rgb, 255:red, 0; green, 0; blue, 0 }  ,fill opacity=1 ] (137.49,94.76) .. controls (138.03,93.29) and (139.64,92.54) .. (141.09,93.09) .. controls (142.54,93.63) and (143.28,95.25) .. (142.75,96.72) .. controls (142.21,98.18) and (140.6,98.93) .. (139.15,98.39) .. controls (137.7,97.85) and (136.96,96.22) .. (137.49,94.76) -- cycle ;
%Shape: Ellipse [id:dp04330242769474646] 
\draw  [fill={rgb, 255:red, 0; green, 0; blue, 0 }  ,fill opacity=1 ] (130.67,120.44) .. controls (131.21,118.98) and (132.82,118.23) .. (134.27,118.77) .. controls (135.72,119.31) and (136.47,120.94) .. (135.93,122.4) .. controls (135.39,123.87) and (133.78,124.62) .. (132.33,124.08) .. controls (130.88,123.53) and (130.14,121.91) .. (130.67,120.44) -- cycle ;
%Shape: Ellipse [id:dp7880602418586143] 
\draw  [fill={rgb, 255:red, 0; green, 0; blue, 0 }  ,fill opacity=1 ] (137.49,146.13) .. controls (138.03,144.66) and (139.64,143.91) .. (141.09,144.46) .. controls (142.54,145) and (143.28,146.62) .. (142.75,148.09) .. controls (142.21,149.56) and (140.6,150.3) .. (139.15,149.76) .. controls (137.7,149.22) and (136.96,147.59) .. (137.49,146.13) -- cycle ;
%Shape: Ellipse [id:dp8502870303591257] 
\draw  [fill={rgb, 255:red, 0; green, 0; blue, 0 }  ,fill opacity=1 ] (156.12,164.93) .. controls (156.65,163.47) and (158.27,162.72) .. (159.72,163.26) .. controls (161.17,163.8) and (161.91,165.43) .. (161.37,166.89) .. controls (160.84,168.36) and (159.23,169.11) .. (157.77,168.56) .. controls (156.32,168.02) and (155.58,166.4) .. (156.12,164.93) -- cycle ;
%Shape: Ellipse [id:dp5475825676546497] 
\draw  [fill={rgb, 255:red, 0; green, 0; blue, 0 }  ,fill opacity=1 ] (181.56,171.26) .. controls (182.1,169.79) and (183.71,169.04) .. (185.16,169.58) .. controls (186.61,170.13) and (187.35,171.75) .. (186.82,173.22) .. controls (186.28,174.68) and (184.67,175.43) .. (183.22,174.89) .. controls (181.77,174.35) and (181.02,172.72) .. (181.56,171.26) -- cycle ;
%Shape: Ellipse [id:dp039478800609058395] 
\draw  [fill={rgb, 255:red, 0; green, 0; blue, 0 }  ,fill opacity=1 ] (207,164.93) .. controls (207.54,163.47) and (209.15,162.72) .. (210.6,163.26) .. controls (212.05,163.8) and (212.8,165.43) .. (212.26,166.89) .. controls (211.72,168.36) and (210.11,169.11) .. (208.66,168.56) .. controls (207.21,168.02) and (206.47,166.4) .. (207,164.93) -- cycle ;
%Shape: Ellipse [id:dp008917252647358098] 
\draw  [fill={rgb, 255:red, 0; green, 0; blue, 0 }  ,fill opacity=1 ] (225.63,146.13) .. controls (226.17,144.66) and (227.78,143.91) .. (229.23,144.46) .. controls (230.68,145) and (231.42,146.62) .. (230.88,148.09) .. controls (230.35,149.56) and (228.74,150.3) .. (227.29,149.76) .. controls (225.83,149.22) and (225.09,147.59) .. (225.63,146.13) -- cycle ;
%Shape: Ellipse [id:dp7750592987464049] 
\draw  [fill={rgb, 255:red, 0; green, 0; blue, 0 }  ,fill opacity=1 ] (207,75.95) .. controls (207.54,74.49) and (209.15,73.74) .. (210.6,74.28) .. controls (212.05,74.82) and (212.8,76.45) .. (212.26,77.92) .. controls (211.72,79.38) and (210.11,80.13) .. (208.66,79.59) .. controls (207.21,79.05) and (206.47,77.42) .. (207,75.95) -- cycle ;
%Straight Lines [id:da3658916566698218] 
\draw [color={rgb, 255:red, 255; green, 0; blue, 0 }  ,draw opacity=1 ]   (158.75,76.94) -- (140.12,95.74) ;
%Straight Lines [id:da20747829188492328] 
\draw [color={rgb, 255:red, 255; green, 0; blue, 0 }  ,draw opacity=1 ]   (140.12,147.11) -- (133.3,121.42) ;
%Straight Lines [id:da7246102982887674] 
\draw [color={rgb, 255:red, 255; green, 0; blue, 0 }  ,draw opacity=1 ]   (184.19,172.24) -- (158.75,165.91) ;
%Straight Lines [id:da4881699101341176] 
\draw [color={rgb, 255:red, 255; green, 0; blue, 0 }  ,draw opacity=1 ]   (228.26,147.11) -- (209.63,165.91) ;
%Shape: Ellipse [id:dp023895767189291983] 
\draw  [fill={rgb, 255:red, 0; green, 0; blue, 0 }  ,fill opacity=1 ] (156.12,75.95) .. controls (156.65,74.49) and (158.27,73.74) .. (159.72,74.28) .. controls (161.17,74.82) and (161.91,76.45) .. (161.37,77.92) .. controls (160.84,79.38) and (159.23,80.13) .. (157.77,79.59) .. controls (156.32,79.05) and (155.58,77.42) .. (156.12,75.95) -- cycle ;
%Shape: Ellipse [id:dp5140040496747655] 
\draw  [fill={rgb, 255:red, 0; green, 0; blue, 0 }  ,fill opacity=1 ] (137.49,94.76) .. controls (138.03,93.29) and (139.64,92.54) .. (141.09,93.09) .. controls (142.54,93.63) and (143.28,95.25) .. (142.75,96.72) .. controls (142.21,98.18) and (140.6,98.93) .. (139.15,98.39) .. controls (137.7,97.85) and (136.96,96.22) .. (137.49,94.76) -- cycle ;
%Shape: Ellipse [id:dp21652076625571115] 
\draw  [fill={rgb, 255:red, 0; green, 0; blue, 0 }  ,fill opacity=1 ] (130.67,120.44) .. controls (131.21,118.98) and (132.82,118.23) .. (134.27,118.77) .. controls (135.72,119.31) and (136.47,120.94) .. (135.93,122.4) .. controls (135.39,123.87) and (133.78,124.62) .. (132.33,124.08) .. controls (130.88,123.53) and (130.14,121.91) .. (130.67,120.44) -- cycle ;
%Shape: Ellipse [id:dp15459068771439866] 
\draw  [fill={rgb, 255:red, 0; green, 0; blue, 0 }  ,fill opacity=1 ] (137.49,146.13) .. controls (138.03,144.66) and (139.64,143.91) .. (141.09,144.46) .. controls (142.54,145) and (143.28,146.62) .. (142.75,148.09) .. controls (142.21,149.56) and (140.6,150.3) .. (139.15,149.76) .. controls (137.7,149.22) and (136.96,147.59) .. (137.49,146.13) -- cycle ;
%Shape: Ellipse [id:dp7303669095727521] 
\draw  [fill={rgb, 255:red, 0; green, 0; blue, 0 }  ,fill opacity=1 ] (156.12,164.93) .. controls (156.65,163.47) and (158.27,162.72) .. (159.72,163.26) .. controls (161.17,163.8) and (161.91,165.43) .. (161.37,166.89) .. controls (160.84,168.36) and (159.23,169.11) .. (157.77,168.56) .. controls (156.32,168.02) and (155.58,166.4) .. (156.12,164.93) -- cycle ;
%Shape: Ellipse [id:dp3858617650982882] 
\draw  [fill={rgb, 255:red, 0; green, 0; blue, 0 }  ,fill opacity=1 ] (181.56,171.26) .. controls (182.1,169.79) and (183.71,169.04) .. (185.16,169.58) .. controls (186.61,170.13) and (187.35,171.75) .. (186.82,173.22) .. controls (186.28,174.68) and (184.67,175.43) .. (183.22,174.89) .. controls (181.77,174.35) and (181.02,172.72) .. (181.56,171.26) -- cycle ;
%Shape: Ellipse [id:dp3564985976094175] 
\draw  [fill={rgb, 255:red, 0; green, 0; blue, 0 }  ,fill opacity=1 ] (207,164.93) .. controls (207.54,163.47) and (209.15,162.72) .. (210.6,163.26) .. controls (212.05,163.8) and (212.8,165.43) .. (212.26,166.89) .. controls (211.72,168.36) and (210.11,169.11) .. (208.66,168.56) .. controls (207.21,168.02) and (206.47,166.4) .. (207,164.93) -- cycle ;
%Shape: Ellipse [id:dp11044396349394348] 
\draw  [fill={rgb, 255:red, 0; green, 0; blue, 0 }  ,fill opacity=1 ] (225.63,146.13) .. controls (226.17,144.66) and (227.78,143.91) .. (229.23,144.46) .. controls (230.68,145) and (231.42,146.62) .. (230.88,148.09) .. controls (230.35,149.56) and (228.74,150.3) .. (227.29,149.76) .. controls (225.83,149.22) and (225.09,147.59) .. (225.63,146.13) -- cycle ;
%Shape: Brace [id:dp6294805562206087] 
\draw   (213.02,67.97) .. controls (214.11,64.26) and (212.81,61.86) .. (209.1,60.76) -- (209.1,60.76) .. controls (203.81,59.19) and (201.71,56.56) .. (202.81,52.85) .. controls (201.71,56.56) and (198.51,57.63) .. (193.22,56.06)(195.6,56.76) -- (193.22,56.06) .. controls (189.51,54.96) and (187.11,56.26) .. (186.02,59.97) ;

% Text Node
\draw (64.66,111.07) node [anchor=north west][inner sep=0.75pt]   [align=left] {$\mathfrak{G}_2$};
% Text Node
\draw (238,108) node [anchor=north west][inner sep=0.75pt]   [align=left] {$\cdots$};
% Text Node
\draw (197.5,37.5) node [anchor=north west][inner sep=0.75pt]  [font=\small] [align=left] {$R_3/u$};
% Text Node
\draw (164,75) node [anchor=north west][inner sep=0.75pt]  [font=\tiny] [align=left] {$d_{2k}(\widehat{\gamma}^{(1)}_0)$};
% Text Node
\draw (200,81) node [anchor=north west][inner sep=0.75pt]  [font=\tiny] [align=left] {$d_{2k}(\widehat{\gamma}^{(1)}_1)$};

\end{tikzpicture}
\caption{The contribution of $\mathfrak{G}_2$ is $\tr(g\cdot e^{R_3/u})$.}\label{graph2+}
\end{figure}
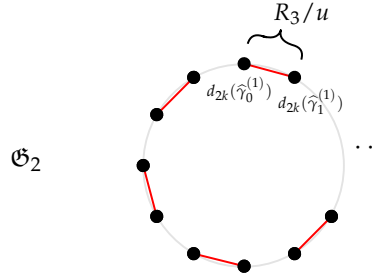
\end{center}

\item Let $\mathfrak{G}_3$ be the set of graphs with  $k$ vertices of valency 3 that are connected by blue propagators to form a wheel, and they are further connected by red propagators to $k$ vertices of valency 1. The 3-valency vertex on the wheel is represented by $d_{2k}(\widehat{\gamma}^{(3)}_0)$ and the 1-valency vertex is represented by $d_{2k}(\widehat{\gamma}^{(1)}_0)$.

 The  sum of these connected diagrams    contributes
$$
%\sum\limits^{\text{con. 1-loop}}_{\substack{\Gamma_1}: \text{type I}} {{W_{\Gamma_1}}\over |\text{Aut}(\Gamma_1)|}
=\sum_{k\geq 0}C(k)\cdot\frac{1}{k!}\sum'_{\epsilon}\text{sign}(\epsilon)F_k(R_{1}(\xi_{\epsilon(1)},\xi_{\epsilon(2)}),\cdots ,R_{1}(\xi_{\epsilon(2k-1)},\xi_{\epsilon(2k)}))$$
where $F_k=k!\text{ch}_k=\text{tr}(X^k)\in (\Sym^k\mathfrak{sp}_{2n}^{\vee})^{\mathfrak{sp}_{2n}},$ and the prime means sum is over $\epsilon\in S_{2k}$ such that $\epsilon(2i-1)<\epsilon(2i)$. The coefficient $C(k)$ of this factor can be compute by integrals over $S^1[k]$
$$C(k)=\frac{u^{-k}}{k}\int_{S^1[k]}\pi^*_{12}(P^{S^1})\pi^*_{23}(P^{S^1})\cdot\cdot\cdot \pi^*_{k1}.(P^{S^1})d\theta^1\wedge\cdots \wedge d\theta^k$$
When k is odd it is 0 and when k is even it is $\frac{u^{-k}}{k}\cdot\frac{2\zeta(k)}{{(2\pi i)}^k}$. Here $\zeta(k)$ is Riemann's zeta function. The total contribution from $\mathfrak{G}_3$ is

$$\sum_{k\geq 0}\frac{u^{-2k}}{2k}\frac{2\zeta(2k)}{(2\pi i)^{2k}}(2k)!ch_{2k}=\sum_{k\geq 0}\frac{2(2k-1)!\zeta(2k)}{(2\pi i)^{2k}}u^{-2k}ch_{2k}.$$
This is precisely $\log \hat{A}(\mathfrak{sp}_{2n})_u$ \cite{Hirzebruch1994Manifolds}. See figure \ref{graph3}.
\begin{center}
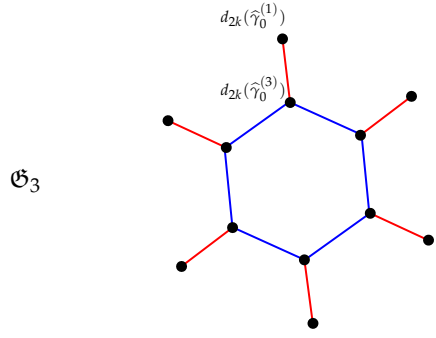
\begin{figure}

\tikzset{every picture/.style={line width=0.75pt}} %set default line width to 0.75pt        

\begin{tikzpicture}[x=0.75pt,y=0.75pt,yscale=-1,xscale=1]
%uncomment if require: \path (0,824); %set diagram left start at 0, and has height of 824

%Shape: Regular Polygon [id:dp1074386109108959] 
\draw  [color={rgb, 255:red, 0; green, 0; blue, 255 }  ,draw opacity=1 ] (247.56,290.31) -- (214.99,313.65) -- (178.5,297.12) -- (174.57,257.25) -- (207.13,233.91) -- (243.63,250.44) -- cycle ;
%Straight Lines [id:da4168942170184329] 
\draw [color={rgb, 255:red, 255; green, 0; blue, 0 }  ,draw opacity=1 ]   (203.76,202.09) -- (207.85,236.44) ;
%Shape: Circle [id:dp7034219140624827] 
\draw  [fill={rgb, 255:red, 0; green, 0; blue, 0 }  ,fill opacity=1 ] (203.58,199.84) .. controls (204.81,199.74) and (205.9,200.66) .. (206,201.9) .. controls (206.11,203.14) and (205.19,204.22) .. (203.95,204.33) .. controls (202.71,204.43) and (201.62,203.51) .. (201.52,202.27) .. controls (201.42,201.03) and (202.34,199.95) .. (203.58,199.84) -- cycle ;
%Shape: Circle [id:dp3872291692192843] 
\draw  [fill={rgb, 255:red, 0; green, 0; blue, 0 }  ,fill opacity=1 ] (207.48,231.96) .. controls (208.71,231.85) and (209.8,232.77) .. (209.9,234.01) .. controls (210.01,235.25) and (209.09,236.34) .. (207.85,236.44) .. controls (206.61,236.54) and (205.52,235.62) .. (205.42,234.38) .. controls (205.32,233.15) and (206.24,232.06) .. (207.48,231.96) -- cycle ;
%Straight Lines [id:da5091547799836247] 
\draw [color={rgb, 255:red, 255; green, 0; blue, 0 }  ,draw opacity=1 ]   (268.82,231.08) -- (241.26,252) ;
%Shape: Circle [id:dp40063895413287676] 
\draw  [fill={rgb, 255:red, 0; green, 0; blue, 0 }  ,fill opacity=1 ] (270.66,229.79) .. controls (271.38,230.8) and (271.13,232.21) .. (270.12,232.92) .. controls (269.1,233.64) and (267.7,233.39) .. (266.98,232.38) .. controls (266.27,231.36) and (266.51,229.96) .. (267.53,229.24) .. controls (268.54,228.53) and (269.95,228.77) .. (270.66,229.79) -- cycle ;
%Shape: Circle [id:dp2853772043531746] 
\draw  [fill={rgb, 255:red, 0; green, 0; blue, 0 }  ,fill opacity=1 ] (244.94,249.4) .. controls (245.66,250.42) and (245.41,251.82) .. (244.4,252.54) .. controls (243.38,253.26) and (241.98,253.01) .. (241.26,252) .. controls (240.55,250.98) and (240.79,249.58) .. (241.81,248.86) .. controls (242.82,248.15) and (244.23,248.39) .. (244.94,249.4) -- cycle ;
%Straight Lines [id:da7571149813585787] 
\draw [color={rgb, 255:red, 255; green, 0; blue, 0 }  ,draw opacity=1 ]   (178.57,297.33) -- (151.01,318.25) ;
%Shape: Circle [id:dp9704364738666633] 
\draw  [fill={rgb, 255:red, 0; green, 0; blue, 0 }  ,fill opacity=1 ] (180.41,296.04) .. controls (181.13,297.05) and (180.88,298.46) .. (179.87,299.17) .. controls (178.85,299.89) and (177.45,299.64) .. (176.73,298.63) .. controls (176.02,297.61) and (176.26,296.21) .. (177.28,295.49) .. controls (178.29,294.78) and (179.7,295.02) .. (180.41,296.04) -- cycle ;
%Shape: Circle [id:dp4286335206541302] 
\draw  [fill={rgb, 255:red, 0; green, 0; blue, 0 }  ,fill opacity=1 ] (154.69,315.65) .. controls (155.41,316.67) and (155.16,318.07) .. (154.15,318.79) .. controls (153.13,319.51) and (151.73,319.26) .. (151.01,318.25) .. controls (150.3,317.23) and (150.54,315.83) .. (151.56,315.11) .. controls (152.57,314.4) and (153.98,314.64) .. (154.69,315.65) -- cycle ;
%Straight Lines [id:da6365677132962042] 
\draw [color={rgb, 255:red, 255; green, 0; blue, 0 }  ,draw opacity=1 ]   (214.8,313.61) -- (219.34,347.91) ;
%Shape: Circle [id:dp8826982217236734] 
\draw  [fill={rgb, 255:red, 0; green, 0; blue, 0 }  ,fill opacity=1 ] (214.58,311.37) .. controls (215.82,311.25) and (216.92,312.16) .. (217.04,313.4) .. controls (217.16,314.63) and (216.25,315.73) .. (215.02,315.85) .. controls (213.78,315.97) and (212.68,315.07) .. (212.56,313.83) .. controls (212.44,312.59) and (213.35,311.49) .. (214.58,311.37) -- cycle ;
%Shape: Circle [id:dp23051427859356644] 
\draw  [fill={rgb, 255:red, 0; green, 0; blue, 0 }  ,fill opacity=1 ] (218.91,343.43) .. controls (220.15,343.31) and (221.25,344.22) .. (221.37,345.45) .. controls (221.48,346.69) and (220.58,347.79) .. (219.34,347.91) .. controls (218.11,348.03) and (217.01,347.12) .. (216.89,345.89) .. controls (216.77,344.65) and (217.67,343.55) .. (218.91,343.43) -- cycle ;
%Straight Lines [id:da9809834018100391] 
\draw [color={rgb, 255:red, 255; green, 0; blue, 0 }  ,draw opacity=1 ]   (248.02,290.12) -- (279.41,304.68) ;
%Shape: Circle [id:dp7961149552499174] 
\draw  [fill={rgb, 255:red, 0; green, 0; blue, 0 }  ,fill opacity=1 ] (246.02,289.1) .. controls (246.58,287.99) and (247.93,287.55) .. (249.04,288.11) .. controls (250.15,288.68) and (250.59,290.03) .. (250.03,291.14) .. controls (249.47,292.25) and (248.11,292.69) .. (247.01,292.13) .. controls (245.9,291.56) and (245.45,290.21) .. (246.02,289.1) -- cycle ;
%Shape: Circle [id:dp006454233883563587] 
\draw  [fill={rgb, 255:red, 0; green, 0; blue, 0 }  ,fill opacity=1 ] (275.39,302.64) .. controls (275.96,301.53) and (277.31,301.09) .. (278.42,301.65) .. controls (279.53,302.22) and (279.97,303.57) .. (279.41,304.68) .. controls (278.85,305.79) and (277.49,306.23) .. (276.38,305.67) .. controls (275.28,305.1) and (274.83,303.75) .. (275.39,302.64) -- cycle ;
%Straight Lines [id:da9347447643706948] 
\draw [color={rgb, 255:red, 255; green, 0; blue, 0 }  ,draw opacity=1 ]   (146.02,243.37) -- (177.41,257.93) ;
%Shape: Circle [id:dp0013428519443521303] 
\draw  [fill={rgb, 255:red, 0; green, 0; blue, 0 }  ,fill opacity=1 ] (144.02,242.35) .. controls (144.58,241.24) and (145.93,240.8) .. (147.04,241.36) .. controls (148.15,241.93) and (148.59,243.28) .. (148.03,244.39) .. controls (147.47,245.5) and (146.11,245.94) .. (145.01,245.38) .. controls (143.9,244.81) and (143.45,243.46) .. (144.02,242.35) -- cycle ;
%Shape: Circle [id:dp658340672827697] 
\draw  [fill={rgb, 255:red, 0; green, 0; blue, 0 }  ,fill opacity=1 ] (173.39,255.89) .. controls (173.96,254.78) and (175.31,254.34) .. (176.42,254.9) .. controls (177.53,255.47) and (177.97,256.82) .. (177.41,257.93) .. controls (176.85,259.04) and (175.49,259.48) .. (174.38,258.92) .. controls (173.28,258.35) and (172.83,257) .. (173.39,255.89) -- cycle ;

% Text Node
\draw (64.5,266.32) node [anchor=north west][inner sep=0.75pt]   [align=left] {$\mathfrak{G}_3$};
% Text Node
\draw (171,182.5) node [anchor=north west][inner sep=0.75pt] [font=\tiny]  [align=left] {$d_{2k}(\widehat{\gamma}^{(1)}_0)$};
% Text Node
\draw (171,219) node [anchor=north west][inner sep=0.75pt] [font=\tiny]  [align=left] {$d_{2k}(\widehat{\gamma}^{(3)}_0)$};

\end{tikzpicture}
\caption{The contribution of $\mathfrak{G}_3$ is $\log \hat{A}(\mathfrak{sp}_{2n})_u$.}\label{graph3}
\end{figure}
\end{center}

\item Let $\mathfrak{G}_4$ be the  set of graphs with  $k$ vertices of valency 3 that are connected by green  propagators to form a wheel, and they are further connected by red propagators to $k$ vertices of valency 1. The 3-valency vertex on the wheel is represented by $d_{2k}(\widehat{\gamma}^{(3)}_0)$ and the 1-valency vertex is represented by $d_{2k}(\widehat{\gamma}^{(1)}_0)$.

Let $\Gamma_4$ be the graph with a vertex  of valency $3$ labeled by $d_{2k}(\widehat{\gamma}^{(3)}_0)$, and a vertex of valency $1$  labeled by $d_{2k}(\widehat{\gamma}^{(1)}_0)$. The two vertexes are connected by a red propagator, and there is a green propagator connecting the $3$-valent vertex with itself.  

The contribution of $\mathfrak{G}_4$ and $\Gamma_4$ is $\log(\Ch_g^\star(\sp_{2n-2k}^g))$. See figure \ref{graph4}.
\begin{center}
\begin{figure}

\tikzset{every picture/.style={line width=0.75pt}} %set default line width to 0.75pt        

\begin{tikzpicture}[x=0.75pt,y=0.75pt,yscale=-1,xscale=1]
%uncomment if require: \path (0,824); %set diagram left start at 0, and has height of 824

%Straight Lines [id:da026978501794988063] 
\draw [color={rgb, 255:red, 0; green, 255; blue, 0 }  ,draw opacity=1 ]   (209.69,406.83) -- (258.19,435.83) ;
%Straight Lines [id:da32782219778345256] 
\draw [color={rgb, 255:red, 0; green, 255; blue, 0 }  ,draw opacity=1 ]   (194.69,490.83) -- (246.69,483.33) ;
%Straight Lines [id:da5090500330861777] 
\draw [color={rgb, 255:red, 0; green, 255; blue, 0 }  ,draw opacity=1 ]   (209.69,406.83) -- (174.88,441.87) ;
%Straight Lines [id:da2777490447913602] 
\draw [color={rgb, 255:red, 0; green, 255; blue, 0 }  ,draw opacity=1 ]   (174.88,441.87) -- (184.39,465.37) -- (194.69,490.83) ;
%Straight Lines [id:da06466794806599163] 
\draw [color={rgb, 255:red, 0; green, 255; blue, 0 }  ,draw opacity=1 ]   (258.19,435.83) -- (246.69,483.33) ;
%Straight Lines [id:da6899866313511351] 
\draw [color={rgb, 255:red, 255; green, 0; blue, 0 }  ,draw opacity=1 ]   (258.13,435.45) -- (287.97,430.21) ;
%Shape: Circle [id:dp9329222217019166] 
\draw  [fill={rgb, 255:red, 0; green, 0; blue, 0 }  ,fill opacity=1 ] (255.98,434.78) .. controls (256.35,433.59) and (257.62,432.93) .. (258.8,433.31) .. controls (259.99,433.68) and (260.65,434.94) .. (260.27,436.13) .. controls (259.9,437.31) and (258.64,437.97) .. (257.45,437.6) .. controls (256.27,437.23) and (255.61,435.96) .. (255.98,434.78) -- cycle ;
%Shape: Circle [id:dp8570372538273829] 
\draw  [fill={rgb, 255:red, 0; green, 0; blue, 0 }  ,fill opacity=1 ] (285.83,429.54) .. controls (286.2,428.35) and (287.46,427.69) .. (288.65,428.07) .. controls (289.83,428.44) and (290.49,429.7) .. (290.12,430.89) .. controls (289.75,432.07) and (288.48,432.73) .. (287.3,432.36) .. controls (286.11,431.99) and (285.45,430.72) .. (285.83,429.54) -- cycle ;
%Straight Lines [id:da6260629902601773] 
\draw [color={rgb, 255:red, 255; green, 0; blue, 0 }  ,draw opacity=1 ]   (246.46,483.31) -- (266.14,506.35) ;
%Shape: Circle [id:dp21140180193904834] 
\draw  [fill={rgb, 255:red, 0; green, 0; blue, 0 }  ,fill opacity=1 ] (245.95,481.12) .. controls (247.16,480.84) and (248.37,481.6) .. (248.65,482.81) .. controls (248.93,484.02) and (248.18,485.23) .. (246.97,485.51) .. controls (245.76,485.79) and (244.55,485.03) .. (244.27,483.82) .. controls (243.98,482.61) and (244.74,481.4) .. (245.95,481.12) -- cycle ;
%Shape: Circle [id:dp573592136735623] 
\draw  [fill={rgb, 255:red, 0; green, 0; blue, 0 }  ,fill opacity=1 ] (265.63,504.16) .. controls (266.84,503.88) and (268.05,504.63) .. (268.33,505.84) .. controls (268.62,507.05) and (267.86,508.26) .. (266.65,508.54) .. controls (265.44,508.82) and (264.23,508.07) .. (263.95,506.86) .. controls (263.67,505.65) and (264.42,504.44) .. (265.63,504.16) -- cycle ;
%Straight Lines [id:da2194339344226235] 
\draw [color={rgb, 255:red, 255; green, 0; blue, 0 }  ,draw opacity=1 ]   (145.46,434.75) -- (174.97,441.66) ;
%Shape: Circle [id:dp941762674387639] 
\draw  [fill={rgb, 255:red, 0; green, 0; blue, 0 }  ,fill opacity=1 ] (143.76,433.28) .. controls (144.57,432.34) and (145.99,432.23) .. (146.93,433.04) .. controls (147.87,433.85) and (147.98,435.27) .. (147.17,436.21) .. controls (146.36,437.16) and (144.94,437.27) .. (144,436.46) .. controls (143.06,435.65) and (142.95,434.23) .. (143.76,433.28) -- cycle ;
%Shape: Circle [id:dp46252051396539273] 
\draw  [fill={rgb, 255:red, 0; green, 0; blue, 0 }  ,fill opacity=1 ] (173.42,440.15) .. controls (174.37,439.35) and (175.79,439.47) .. (176.6,440.42) .. controls (177.4,441.37) and (177.28,442.79) .. (176.33,443.59) .. controls (175.38,444.39) and (173.96,444.27) .. (173.15,443.32) .. controls (172.35,442.37) and (172.47,440.95) .. (173.42,440.15) -- cycle ;
%Straight Lines [id:da6796403763465932] 
\draw [color={rgb, 255:red, 255; green, 0; blue, 0 }  ,draw opacity=1 ]   (195.15,490.44) -- (179.45,516.36) ;
%Shape: Circle [id:dp940237625828022] 
\draw  [fill={rgb, 255:red, 0; green, 0; blue, 0 }  ,fill opacity=1 ] (197.07,489.26) .. controls (197.71,490.33) and (197.38,491.71) .. (196.32,492.36) .. controls (195.26,493) and (193.87,492.67) .. (193.22,491.61) .. controls (192.58,490.55) and (192.91,489.16) .. (193.97,488.51) .. controls (195.03,487.87) and (196.42,488.2) .. (197.07,489.26) -- cycle ;
%Shape: Circle [id:dp9510806964918008] 
\draw  [fill={rgb, 255:red, 0; green, 0; blue, 0 }  ,fill opacity=1 ] (181.38,515.19) .. controls (182.02,516.25) and (181.69,517.63) .. (180.63,518.28) .. controls (179.57,518.93) and (178.18,518.59) .. (177.53,517.53) .. controls (176.89,516.47) and (177.22,515.09) .. (178.28,514.44) .. controls (179.34,513.79) and (180.73,514.13) .. (181.38,515.19) -- cycle ;
%Straight Lines [id:da9040599420680385] 
\draw [color={rgb, 255:red, 255; green, 0; blue, 0 }  ,draw opacity=1 ]   (206.94,377.09) -- (210.16,407.22) ;
%Shape: Circle [id:dp10088729979291033] 
\draw  [fill={rgb, 255:red, 0; green, 0; blue, 0 }  ,fill opacity=1 ] (207.76,374.99) .. controls (208.92,375.45) and (209.49,376.75) .. (209.04,377.91) .. controls (208.58,379.07) and (207.28,379.64) .. (206.12,379.19) .. controls (204.96,378.73) and (204.39,377.43) .. (204.85,376.27) .. controls (205.3,375.11) and (206.6,374.54) .. (207.76,374.99) -- cycle ;
%Shape: Circle [id:dp17214798974331347] 
\draw  [fill={rgb, 255:red, 0; green, 0; blue, 0 }  ,fill opacity=1 ] (210.98,405.12) .. controls (212.14,405.58) and (212.71,406.88) .. (212.25,408.04) .. controls (211.8,409.2) and (210.5,409.77) .. (209.34,409.32) .. controls (208.18,408.86) and (207.61,407.56) .. (208.06,406.4) .. controls (208.52,405.24) and (209.82,404.67) .. (210.98,405.12) -- cycle ;
%Curve Lines [id:da09449072685788795] 
\draw [color={rgb, 255:red, 0; green, 255; blue, 0 }  ,draw opacity=1 ]   (214.91,554.15) .. controls (237.36,533.12) and (260.82,538.15) .. (260.41,554.62) .. controls (259.99,571.08) and (238.45,575.65) .. (214.91,554.15) -- cycle ;
%Straight Lines [id:da9654485422722555] 
\draw [color={rgb, 255:red, 255; green, 0; blue, 0 }  ,draw opacity=1 ]   (184.61,554.09) -- (214.91,554.15) ;
%Shape: Circle [id:dp5725177609337122] 
\draw  [fill={rgb, 255:red, 0; green, 0; blue, 0 }  ,fill opacity=1 ] (182.61,553.05) .. controls (183.19,551.95) and (184.54,551.52) .. (185.65,552.1) .. controls (186.75,552.67) and (187.17,554.03) .. (186.6,555.13) .. controls (186.03,556.24) and (184.67,556.66) .. (183.56,556.09) .. controls (182.46,555.51) and (182.04,554.15) .. (182.61,553.05) -- cycle ;
%Shape: Circle [id:dp21203606183701773] 
\draw  [fill={rgb, 255:red, 0; green, 0; blue, 0 }  ,fill opacity=1 ] (212.91,553.11) .. controls (213.49,552.01) and (214.85,551.58) .. (215.95,552.15) .. controls (217.05,552.73) and (217.48,554.09) .. (216.9,555.19) .. controls (216.33,556.29) and (214.97,556.72) .. (213.87,556.14) .. controls (212.77,555.57) and (212.34,554.21) .. (212.91,553.11) -- cycle ;

% Text Node
\draw (67.46,439.87) node [anchor=north west][inner sep=0.75pt]   [align=left] {$\mathfrak{G}_4$};
% Text Node
\draw (67.46,548.72) node [anchor=north west][inner sep=0.75pt]   [align=left] {$\Gamma_4$};
% Text Node
\draw (174,360) node [anchor=north west][inner sep=0.75pt] [font=\tiny]  [align=left] {$d_{2k}(\widehat{\gamma}^{(1)}_0)$};
% Text Node
\draw (172,396) node [anchor=north west][inner sep=0.75pt] [font=\tiny]  [align=left] {$d_{2k}(\widehat{\gamma}^{(3)}_0)$};
% Text Node
\draw (164,562) node [anchor=north west][inner sep=0.75pt] [font=\tiny]  [align=left] {$d_{2k}(\widehat{\gamma}^{(1)}_0)$};
% Text Node
\draw (201.81,562) node [anchor=north west][inner sep=0.75pt] [font=\tiny]  [align=left] {$d_{2k}(\widehat{\gamma}^{(3)}_0)$};

\end{tikzpicture}

\caption{The contribution of $\mathfrak{G}_4$ and $\Gamma_4$ is $\log(\Ch_g^\star(\sp_{2n-2k}^g))$. }\label{graph4}
\end{figure}
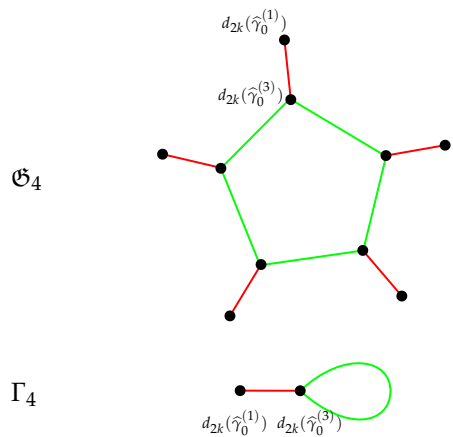
\end{center}

\item Let $\mathfrak{G}_5$ be the  set of wheels with both blue   and  green edges. Since
$$
\int_{S_1} \bracket{u-\frac{1}{2}}du =0, 
$$
the contribution of graphs in $\mathfrak{G}_5$ vanishes. See figure \ref{graph5}.
\begin{center}
\begin{figure}

\tikzset{every picture/.style={line width=0.75pt}} %set default line width to 0.75pt        

\begin{tikzpicture}[x=0.75pt,y=0.75pt,yscale=-1,xscale=1]
%uncomment if require: \path (0,824); %set diagram left start at 0, and has height of 824

%Straight Lines [id:da5846598423100041] 
\draw [color={rgb, 255:red, 255; green, 0; blue, 0 }  ,draw opacity=1 ]   (204.76,640.29) -- (208.85,674.64) ;
%Shape: Circle [id:dp42868798857393686] 
\draw  [fill={rgb, 255:red, 0; green, 0; blue, 0 }  ,fill opacity=1 ] (204.58,638.04) .. controls (205.81,637.94) and (206.9,638.86) .. (207,640.1) .. controls (207.11,641.34) and (206.19,642.42) .. (204.95,642.53) .. controls (203.71,642.63) and (202.62,641.71) .. (202.52,640.47) .. controls (202.42,639.23) and (203.34,638.15) .. (204.58,638.04) -- cycle ;
%Shape: Circle [id:dp10357094513832033] 
\draw  [fill={rgb, 255:red, 0; green, 0; blue, 0 }  ,fill opacity=1 ] (208.48,670.16) .. controls (209.71,670.05) and (210.8,670.97) .. (210.9,672.21) .. controls (211.01,673.45) and (210.09,674.54) .. (208.85,674.64) .. controls (207.61,674.74) and (206.52,673.82) .. (206.42,672.58) .. controls (206.32,671.35) and (207.24,670.26) .. (208.48,670.16) -- cycle ;
%Straight Lines [id:da16017561049583195] 
\draw [color={rgb, 255:red, 255; green, 0; blue, 0 }  ,draw opacity=1 ]   (269.82,669.28) -- (242.26,690.2) ;
%Shape: Circle [id:dp24998093554320544] 
\draw  [fill={rgb, 255:red, 0; green, 0; blue, 0 }  ,fill opacity=1 ] (271.66,667.99) .. controls (272.38,669) and (272.13,670.41) .. (271.12,671.12) .. controls (270.1,671.84) and (268.7,671.59) .. (267.98,670.58) .. controls (267.27,669.56) and (267.51,668.16) .. (268.53,667.44) .. controls (269.54,666.73) and (270.95,666.97) .. (271.66,667.99) -- cycle ;
%Shape: Circle [id:dp6832233814585505] 
\draw  [fill={rgb, 255:red, 0; green, 0; blue, 0 }  ,fill opacity=1 ] (245.94,687.6) .. controls (246.66,688.62) and (246.41,690.02) .. (245.4,690.74) .. controls (244.38,691.46) and (242.98,691.21) .. (242.26,690.2) .. controls (241.55,689.18) and (241.79,687.78) .. (242.81,687.06) .. controls (243.82,686.35) and (245.23,686.59) .. (245.94,687.6) -- cycle ;
%Straight Lines [id:da9859832360293587] 
\draw [color={rgb, 255:red, 255; green, 0; blue, 0 }  ,draw opacity=1 ]   (179.57,735.53) -- (152.01,756.45) ;
%Shape: Circle [id:dp18110521433951465] 
\draw  [fill={rgb, 255:red, 0; green, 0; blue, 0 }  ,fill opacity=1 ] (181.41,734.24) .. controls (182.13,735.25) and (181.88,736.66) .. (180.87,737.37) .. controls (179.85,738.09) and (178.45,737.84) .. (177.73,736.83) .. controls (177.02,735.81) and (177.26,734.41) .. (178.28,733.69) .. controls (179.29,732.98) and (180.7,733.22) .. (181.41,734.24) -- cycle ;
%Shape: Circle [id:dp3092086526932498] 
\draw  [fill={rgb, 255:red, 0; green, 0; blue, 0 }  ,fill opacity=1 ] (155.69,753.85) .. controls (156.41,754.87) and (156.16,756.27) .. (155.15,756.99) .. controls (154.13,757.71) and (152.73,757.46) .. (152.01,756.45) .. controls (151.3,755.43) and (151.54,754.03) .. (152.56,753.31) .. controls (153.57,752.6) and (154.98,752.84) .. (155.69,753.85) -- cycle ;
%Straight Lines [id:da7059058795133876] 
\draw [color={rgb, 255:red, 255; green, 0; blue, 0 }  ,draw opacity=1 ]   (215.8,751.81) -- (220.34,786.11) ;
%Shape: Circle [id:dp7040157123027212] 
\draw  [fill={rgb, 255:red, 0; green, 0; blue, 0 }  ,fill opacity=1 ] (215.58,749.57) .. controls (216.82,749.45) and (217.92,750.36) .. (218.04,751.6) .. controls (218.16,752.83) and (217.25,753.93) .. (216.02,754.05) .. controls (214.78,754.17) and (213.68,753.27) .. (213.56,752.03) .. controls (213.44,750.79) and (214.35,749.69) .. (215.58,749.57) -- cycle ;
%Shape: Circle [id:dp7071945731965915] 
\draw  [fill={rgb, 255:red, 0; green, 0; blue, 0 }  ,fill opacity=1 ] (219.91,781.63) .. controls (221.15,781.51) and (222.25,782.42) .. (222.37,783.65) .. controls (222.48,784.89) and (221.58,785.99) .. (220.34,786.11) .. controls (219.11,786.23) and (218.01,785.32) .. (217.89,784.09) .. controls (217.77,782.85) and (218.67,781.75) .. (219.91,781.63) -- cycle ;
%Straight Lines [id:da5778848004914099] 
\draw [color={rgb, 255:red, 255; green, 0; blue, 0 }  ,draw opacity=1 ]   (249.02,728.32) -- (280.41,742.88) ;
%Shape: Circle [id:dp7941540045034073] 
\draw  [fill={rgb, 255:red, 0; green, 0; blue, 0 }  ,fill opacity=1 ] (247.02,727.3) .. controls (247.58,726.19) and (248.93,725.75) .. (250.04,726.31) .. controls (251.15,726.88) and (251.59,728.23) .. (251.03,729.34) .. controls (250.47,730.45) and (249.11,730.89) .. (248.01,730.33) .. controls (246.9,729.76) and (246.45,728.41) .. (247.02,727.3) -- cycle ;
%Shape: Circle [id:dp2778591570078549] 
\draw  [fill={rgb, 255:red, 0; green, 0; blue, 0 }  ,fill opacity=1 ] (276.39,740.84) .. controls (276.96,739.73) and (278.31,739.29) .. (279.42,739.85) .. controls (280.53,740.42) and (280.97,741.77) .. (280.41,742.88) .. controls (279.85,743.99) and (278.49,744.43) .. (277.38,743.87) .. controls (276.28,743.3) and (275.83,741.95) .. (276.39,740.84) -- cycle ;
%Straight Lines [id:da8463075769859125] 
\draw [color={rgb, 255:red, 255; green, 0; blue, 0 }  ,draw opacity=1 ]   (147.02,681.57) -- (178.41,696.13) ;
%Shape: Circle [id:dp07995325114650143] 
\draw  [fill={rgb, 255:red, 0; green, 0; blue, 0 }  ,fill opacity=1 ] (145.02,680.55) .. controls (145.58,679.44) and (146.93,679) .. (148.04,679.56) .. controls (149.15,680.13) and (149.59,681.48) .. (149.03,682.59) .. controls (148.47,683.7) and (147.11,684.14) .. (146.01,683.58) .. controls (144.9,683.01) and (144.45,681.66) .. (145.02,680.55) -- cycle ;
%Shape: Circle [id:dp7330047971595194] 
\draw  [fill={rgb, 255:red, 0; green, 0; blue, 0 }  ,fill opacity=1 ] (174.39,694.09) .. controls (174.96,692.98) and (176.31,692.54) .. (177.42,693.1) .. controls (178.53,693.67) and (178.97,695.02) .. (178.41,696.13) .. controls (177.85,697.24) and (176.49,697.68) .. (175.38,697.12) .. controls (174.28,696.55) and (173.83,695.2) .. (174.39,694.09) -- cycle ;
%Straight Lines [id:da19513283326906194] 
\draw [color={rgb, 255:red, 0; green, 0; blue, 255 }  ,draw opacity=1 ]   (208.66,672.4) -- (244.1,688.9) ;
%Straight Lines [id:da5670421417796835] 
\draw [color={rgb, 255:red, 0; green, 0; blue, 255 }  ,draw opacity=1 ]   (176.4,695.11) -- (179.57,735.53) ;
%Straight Lines [id:da5751310463519982] 
\draw [color={rgb, 255:red, 0; green, 255; blue, 0 }  ,draw opacity=1 ]   (176.4,695.11) -- (208.66,672.4) ;
%Straight Lines [id:da7633276509430376] 
\draw [color={rgb, 255:red, 0; green, 255; blue, 0 }  ,draw opacity=1 ]   (244.1,688.9) -- (249.02,728.32) ;
%Straight Lines [id:da3025895767915896] 
\draw [color={rgb, 255:red, 0; green, 0; blue, 255 }  ,draw opacity=1 ]   (179.57,735.53) -- (215.8,751.81) ;
%Straight Lines [id:da9843551230918544] 
\draw [color={rgb, 255:red, 0; green, 255; blue, 0 }  ,draw opacity=1 ]   (249.02,728.32) -- (215.8,751.81) ;
%Shape: Circle [id:dp3507284995234997] 
\draw  [fill={rgb, 255:red, 0; green, 0; blue, 0 }  ,fill opacity=1 ] (181.41,734.24) .. controls (182.13,735.25) and (181.88,736.66) .. (180.87,737.37) .. controls (179.85,738.09) and (178.45,737.84) .. (177.73,736.83) .. controls (177.02,735.81) and (177.26,734.41) .. (178.28,733.69) .. controls (179.29,732.98) and (180.7,733.22) .. (181.41,734.24) -- cycle ;
%Shape: Circle [id:dp11893235578047145] 
\draw  [fill={rgb, 255:red, 0; green, 0; blue, 0 }  ,fill opacity=1 ] (178.24,693.81) .. controls (178.96,694.83) and (178.71,696.23) .. (177.7,696.95) .. controls (176.68,697.66) and (175.28,697.42) .. (174.56,696.41) .. controls (173.85,695.39) and (174.09,693.99) .. (175.11,693.27) .. controls (176.12,692.56) and (177.52,692.8) .. (178.24,693.81) -- cycle ;
%Shape: Circle [id:dp7782964231976319] 
\draw  [fill={rgb, 255:red, 0; green, 0; blue, 0 }  ,fill opacity=1 ] (210.5,671.1) .. controls (211.22,672.12) and (210.97,673.52) .. (209.96,674.24) .. controls (208.94,674.95) and (207.54,674.71) .. (206.82,673.69) .. controls (206.11,672.68) and (206.35,671.27) .. (207.37,670.56) .. controls (208.38,669.84) and (209.79,670.09) .. (210.5,671.1) -- cycle ;
%Shape: Circle [id:dp5678332061543673] 
\draw  [fill={rgb, 255:red, 0; green, 0; blue, 0 }  ,fill opacity=1 ] (245.94,687.6) .. controls (246.66,688.62) and (246.41,690.02) .. (245.4,690.74) .. controls (244.38,691.46) and (242.98,691.21) .. (242.26,690.2) .. controls (241.55,689.18) and (241.79,687.78) .. (242.81,687.06) .. controls (243.82,686.35) and (245.23,686.59) .. (245.94,687.6) -- cycle ;
%Shape: Circle [id:dp889964972688327] 
\draw  [fill={rgb, 255:red, 0; green, 0; blue, 0 }  ,fill opacity=1 ] (217.64,750.52) .. controls (218.35,751.53) and (218.11,752.94) .. (217.1,753.65) .. controls (216.08,754.37) and (214.68,754.13) .. (213.96,753.11) .. controls (213.24,752.09) and (213.49,750.69) .. (214.5,749.97) .. controls (215.52,749.26) and (216.92,749.5) .. (217.64,750.52) -- cycle ;
%Shape: Circle [id:dp23207454650089976] 
\draw  [fill={rgb, 255:red, 0; green, 0; blue, 0 }  ,fill opacity=1 ] (250.86,727.02) .. controls (251.58,728.04) and (251.34,729.44) .. (250.32,730.16) .. controls (249.3,730.88) and (247.9,730.63) .. (247.18,729.62) .. controls (246.47,728.6) and (246.71,727.2) .. (247.73,726.48) .. controls (248.74,725.77) and (250.15,726.01) .. (250.86,727.02) -- cycle ;

% Text Node
\draw (65.5,704.52) node [anchor=north west][inner sep=0.75pt]   [align=left] {$\mathfrak{G}_5$};

\end{tikzpicture}

\caption{The contribution of $\mathfrak{G}_5$ is $0$.}\label{graph5}
\end{figure}
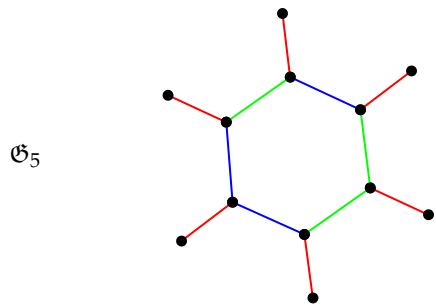
\end{center}
\end{itemize}

\end{document}